\documentclass[reqno,11pt]{amsart}

\usepackage[margin=1.0in]{geometry}
\usepackage{amscd,amssymb,amsmath,epsfig,subfig}
\usepackage{cancel}
\usepackage{mathrsfs}
\usepackage{epsfig}
\usepackage{color}
\definecolor{brightlavender}{rgb}{0.75, 0.58, 0.89}
\definecolor{carnelian}{rgb}{0.7, 0.11, 0.11}
\usepackage[all]{xy}
\usepackage{verbatim}
\usepackage{stmaryrd}
\usepackage{mathtools}
\usepackage{enumitem}
\usepackage[final]{microtype}
\usepackage[hidelinks]{hyperref}

\usepackage{pdfpages}

\usepackage{tikz}
\usepackage{tikz-cd}
\usetikzlibrary{matrix,arrows,positioning}
\usetikzlibrary{decorations.markings,shapes.geometric}

\tikzset{anchorbase/.style={baseline={([yshift=-0.5ex]current bounding box.center)}},
  int/.style={thick},
  cross line/.style={preaction={draw=white,line width=6pt,-}},
  wall/.style={thin,double,blue},
  middlearrow/.style={postaction=decorate,decoration={markings,mark=at
    position .55 with {\arrow{stealth};}}},
  middlearrowrev/.style={postaction=decorate,decoration={markings,mark=at
    position .55 with {\arrowreversed{stealth};}}},
  ev/.style={shape=rectangle, draw}
}

\newcommand{\ptr}{\operatorname{ptr}}

\newtheorem{Df}{Definition}[section]
\newtheorem{definition}[Df]{Definition}
\newtheorem{theorem}[Df]{Theorem}

\newtheorem{prop}[Df]{Proposition}
\newtheorem{proposition}[Df]{Proposition}
\newtheorem{lemma}[Df]{Lemma}

\newtheorem{corollary}[Df]{Corollary}

\newtheorem{proposal}[Df]{Proposal}

\newtheorem{thm}{Theorem}

\theoremstyle{definition}
\newtheorem{remark}[Df]{Remark}

\newtheorem{examplex}[Df]{Example}
\newenvironment{example}
{\pushQED{\qed}\examplex}
{\popQED\endexamplex}

\newcommand{\tcoev}{\stackrel{\longleftarrow}{\operatorname{coev}}}
\newcommand{\tev}{\stackrel{\longleftarrow}{\operatorname{ev}}}
\newcommand{\ev}{\stackrel{\longrightarrow}{\operatorname{ev}}}
\newcommand{\coev}{\stackrel{\longrightarrow}{\operatorname{coev}}}
\newcommand{\p}[1]{\ensuremath{\bar {#1}}}

\newcommand{\cat}{\mathcal{C}}

\newcommand{\cath}{\mathcal{C}}

\newcommand{\catq}{\mathcal{D}^{q}}
\newcommand{\catInt}{\mathcal{D}^{q,\textnormal{int}}}
\newcommand{\C}{\ensuremath{\mathbb{C}}}
\newcommand{\Z}{\ensuremath{\mathbb{Z}}}
\newcommand{\R}{\ensuremath{\mathbb{R}}}

\newcommand{\slt}{\ensuremath{\mathfrak{sl}(2)}}
\newcommand{\gloo}{\ensuremath{\mathfrak{gl}(1\vert 1)}}
\newcommand{\psl}{\ensuremath{\mathfrak{psl}(1 \vert 1)}}
\newcommand{\pgl}{\ensuremath{\mathfrak{pgl}(1 \vert 1)}}
\newcommand{\Uoo}{\ensuremath{U(1 \vert 1)}}
\newcommand{\GLoo}{\ensuremath{GL(1 \vert 1)}}
\newcommand{\Uhgloo}{\ensuremath{U_{h}(\gloo)}}

\newcommand{\Uq}{\ensuremath{U_{q}^{E}(\gloo)}}
\newcommand{\UqB}{\ensuremath{U_{q}^{E, \geq 0}(\gloo)}}
\newcommand{\Uqsl}{\ensuremath{U_{q}^{H}(\slt)}}
\newcommand{\End}{\operatorname{End}}
\newcommand{\Hom}{\operatorname{Hom}}

\newcommand{\unit}{\ensuremath{\mathbb{I}}}
\newcommand{\Id}{\operatorname{Id}}
\newcommand{\qdim}{\operatorname{qdim}}

\newcommand{\qd}{\operatorname{\mathsf{d}}}
\newcommand{\Gr}{\operatorname{\mathsf{G}}}

\newcommand{\ve}{\varepsilon}

\newcommand{\ms}[1]{\mbox{\tiny$#1$}}

\newcommand{\coh}{\omega}

\newcommand{\kk}{\Bbbk}
\newcommand{\ideal}{\mathcal{I}} 
\newcommand{\mt}{\operatorname{\mathsf{t}}}
\newcommand{\catb}{\mathcal{B}} 
\newcommand{\Zt}{\ensuremath{\mathsf{Z}}}
\newcommand{\XX}{\ensuremath{\mathsf{X}}}
\newcommand{\spvs}{\operatorname{span}}
\newcommand{\Proj}{{\mathcal{P}}}
\newcommand{\D}{{\mathscr{D}}}
\newcommand{\Cob}{\mathsf{C}\mathsf{ob}^{\textnormal{ad}}}
\newcommand{\Vect}{\mathsf{V}\mathsf{ect}}
\newcommand{\ZVect}{\mathsf{V}\mathsf{ect}^{\Zt\operatorname{-}\textnormal{gr}}}
\newcommand{\CobExt}{\check{\mathsf{C}}\mathsf{ob}^{\mathsf{ad}}}
\newcommand{\ZCat}{\check{\mathsf{C}}\mathsf{at}^{\Zt\operatorname{-}\textnormal{gr}}}
\newcommand{\state}{\mathsf{V}}
\newcommand{\TQFT}{\mathcal{Z}}
\newcommand{\hS}{\widehat{S}}
\newcommand{\Co}{\mathsf{Col}}
\newcommand{\spin}{\operatorname{spin}}
\newcommand{\PP}{W}

\newcommand{\wre}{\textnormal{wr}}
\newcommand{\Ztwo}{\Z \slash 2 \Z}

\newcommand{\CGP}{{\rm CGP}}
\newcommand{\CS}{{\mathcal{S}}}
\newcommand{\CM}{{\mathcal M}}

\newcommand{\epsh}[2]
         {\begin{array}{c} \hspace{-1.3mm}
        \raisebox{-4pt}{\epsfig{figure=#1,height=#2}}
        \hspace{-1.9mm}\end{array}}

\newcounter{exo} \newcounter{numexercice}
\renewcommand{\theexo}{\arabic{exo}}

\begin{document}

\title[$U_q(\mathfrak{gl}(1 \vert 1))$ and $U(1 \vert 1)$ Chern--Simons theory]{Three dimensional topological quantum field theory from $U_q(\mathfrak{gl}(1 \vert 1))$ and $U(1 \vert 1)$ Chern--Simons theory}

\author[N. Geer]{Nathan Geer}
\address{Department of Mathematics and Statistics \\
  Utah State University \\
  Logan, Utah 84322, USA}
\email{nathan.geer@gmail.com}

\author[M.\,B. Young]{Matthew B. Young}
\address{Department of Mathematics and Statistics \\ Utah State University\\
Logan, Utah 84322 \\ USA}
\email{matthew.young@usu.edu}

\date{\today}


\begin{abstract}
We introduce an unrolled quantization $\Uq$ of the complex Lie superalgebra $\gloo$ and use its categories of weight modules to construct and study new three dimensional non-semisimple topological quantum field theories. These theories are defined on categories of cobordisms which are decorated by ribbon graphs and cohomology classes and take values in categories of graded super vector spaces. Computations in these theories are enabled by a detailed study of the representation theory of $\Uq$, both for generic and root of unity $q$. We argue that by restricting to subcategories of integral weight modules we obtain topological quantum field theories which are mathematical models of Chern--Simons theories with gauge supergroups $\psl$ and $\Uoo$ coupled to background flat $\C^{\times}$-connections, as studied in the physics literature by Rozansky--Saleur and Mikhaylov. In particular, we match Verlinde formulae and mapping class group actions on state spaces of non-generic tori with results in the physics literature. We also obtain explicit descriptions of state spaces of generic surfaces, including their graded dimensions, which go beyond results in the physics literature.
\end{abstract}

\maketitle
\setcounter{tocdepth}{1}
\tableofcontents

\section*{Introduction}

This paper constructs and studies new three dimensional topological quantum field theories from non-semisimple categories of representations of the unrolled quantum group of the complex Lie superalgebra $\gloo$ and establishes a relationship between these theories and various supergroup Chern--Simons theories studied in the physics literature. Before stating our results in more detail, we provide some context.

\subsection*{Background and motivation}

Chern--Simons theory is a three dimensional quantum gauge theory which was introduced by Witten to give a physical realization of the Jones polynomial \cite{witten1989}. The input data is a compact Lie group $G$, the gauge group, and a class $k \in H^4(BG ; \Z)$, the level, satisfying a non-degeneracy condition. At the physical level of rigor, Chern--Simons theory produces invariants of links in closed oriented $3$-manifolds which are local in the sense that they can be computed using cutting and gluing techniques. Witten also argued that calculations in Chern--Simons theory can be made using its boundary conformal field theory, a Wess--Zumino--Witten theory with target $G$, thereby importing techniques from the theories of rational vertex operator algebras and affine Lie algebras to knot theory and $3$-manifold topology. Since the physical definition of Chern--Simons theory relies on path integrals, it cannot, at present, be used to give a mathematical construction of the theory. Motivated by this, Reshetikhin and Turaev constructed from a modular tensor category $\cat$ a three dimensional topological quantum field theory $\TQFT_{\cat} : \mathsf{Cob}_{\cat} \rightarrow \Vect_{\C}$ which, in particular, encodes invariants of $\mathcal{C}$-colored links in $3$-manifolds with the expected locality properties \cite{reshetikhin1990,reshetikhin1991,Tu}. When $G$ is simple and simply connected, in which case $k$ is an integer, there is a modular tensor category $\mathcal{C}(G,k)$ of semisimplified representations of the quantum group $U_q(\mathfrak{g}_{\C})$ at a $k$-dependent root of unity $q$ and $\TQFT_{\mathcal{C}(G,k)}$ is a mathematical model of Chern--Simons theory \cite{reshetikhin1991,andersen1992,turaev1993,sawin2006}. Crucial to the construction of $\TQFT_{\cat}$ is that modular tensor categories are semisimple, have only finitely many isomorphism classes of simple objects and have the property that simple objects have non-zero quantum dimension.

Extensions of Chern--Simons theory to more general classes of gauge groups have been proposed in the physics literature. This includes gauge groups which are non-compact Lie groups and complex reductive groups \cite{witten1991,barnatan1991,gukov2005,dimofte2009} and Lie supergroups \cite{witten1989b,horne1990,rozansky1992,rozansky1993,rozansky1994,kapustin2009b,GaiottoWitten-Janus,Mik2015,mikhaylov2015b}. Such extensions are expected to have applications to many areas, including the Volume Conjecture, logarithmic conformal field theory and three dimensional quantum gravity. Mathematical constructions of these extensions have largely been obstructed by technical and conceptual difficulties which appear when moving beyond compact gauge groups. For example, the Chern--Simons/Wess--Zumino--Witten correspondence is fundamentally unclear in these extensions, thereby preventing the use of recent advances in the theory of logarithmic vertex operator algebras \cite{Gotz:2006qp, Quella:2007hr,Creutzig:2011cu,creutzig2022}. Since Chern--Simons theories with non-compact gauge groups involve categories of line operators which are non-semisimple, have infinitely many isomorphism classes of simple objects and have simple objects with vanishing quantum dimension, there are serious obstructions to applying the Reshetikhin--Turaev  construction.

It is therefore of interest to extend Reshetikhin--Turaev-type constructions beyond modular tensor categories. Early approaches to such extensions are given in the works of Hennings \cite{hennings1996} and Kerler and Lyubashenko \cite{kerler2001}. More recently, the first author and collaborators have created a theory of renormalized quantum invariants of low dimensional manifolds \cite{GPT,CGP14,BCGP,derenzi2020,derenzi2022}. Central categorical structures of this theory include relative pre-modular categories, non-degenerate relative pre-modular categories and relative modular categories which produce invariants of links, invariants of closed $3$-manifolds and three dimensional topological quantum field theories, respectively. In this paper we focus on relative modular categories, the strongest of these structures, which are generalizations of modular tensor categories that allow for non-semisimplicity, infinitely many simple objects and simple objects with vanishing quantum dimension. Roughly speaking, a relative modular category $\cat$ is a ribbon category which is compatibly graded by an abelian group $\Gr$ and carries a degree $0$ monoidal action of a group $\Zt$. It is required that there exists a sufficiently small subset $\XX \subset \Gr$ such that the full subcategories $\cat_g \subset \cat$, $g \in \Gr \setminus \XX$, are semisimple and have only finitely many isomorphism classes of simple objects modulo $\Zt$. The associated three dimensional topological quantum field theory $\TQFT_{\cat}: \Cob_{\cat} \rightarrow \ZVect_{\C}$, constructed by De Renzi \cite{derenzi2022}, is defined on a category of admissible decorated three dimensional cobordisms and takes values in a braided monoidal category of $\Zt$-graded complex vector spaces. When $\cat$ is in fact a modular tensor category, the theory $\TQFT_{\cat}$ reduces to that of Reshetikhin and Turaev. In general, $\TQFT_{\cat}$ enjoys many new features not shared by modular theories, including the ability to distinguish homotopy classes of lens spaces and produce representations of mapping class groups with interesting properties, such as having Dehn twists act with infinite order.

Categories of weight modules over unrolled quantum groups of complex simple Lie algebras are relative modular and their associated topological quantum field theories have been the subject of recent interest \cite{BCGP,derenzi2020,creutzig2021,derenzi2022}. The case of unrolled quantum groups of complex Lie superalgebras is more subtle. For example, depending on the precise class of weight modules being considered, the Lie superalgebras $\mathfrak{sl}(m \vert n)$, $m \neq n$, produce categories which are relative modular or only non-degenerate pre-relative modular \cite{ha2018,AGP,geer2021,ha2022}. The resulting topological quantum field theories have not been studied. This paper presents the first systematic study of topological quantum field theories arising from quantum supergroups and suggests that examples arising from higher rank Lie superalgebras admit natural physical realizations, in contrast to the original expectations of Mikhaylov and Witten \cite{mikhaylov2015b}. Further examples of (non-degenerate) relative pre-modular categories, some of which are conjectured to extend to relative modular categories, and their applications to knot theory and $3$-manifold topology can be found in \cite{geer2007,GP1,AGP}.

\subsection*{Main results}

We construct new examples of relative modular categories using the representation theory of an unrolled quantization of the complex Lie superalgebra $\gloo$. We study in detail the resulting topological quantum field theories and connect them to $\psl$ and $\Uoo$ Chern--Simons theories and $\Uoo$ Wess--Zumino--Witten theory, as studied in the physics literature by Rozansky and Saleur \cite{rozansky1992,rozansky1993,rozansky1994} and Mikhaylov \cite{Mik2015}. We also connect our work to mathematical results on the quantum topology of $\gloo$ \cite{frohman1991,kauffman1991,reshetikhin1992,Viro,sartori2015,bao2022}. In the remainder of this introduction we outline the structure of the paper and state the main results.

We begin in Section \ref{sec:preliminaries} by establishing our conventions for relative modular categories and recalling how these categories can be used to define invariants of links and $3$-manifolds and, ultimately, three dimensional topological quantum field theories. Our first main result asserts finite dimensionality of the state spaces of these field theories under the assumption that the input relative modular category is TQFT finite in the sense of Definition \ref{def:relModFinite}. TQFT finiteness is a relatively weak condition and is straightforward to verify in concrete examples. For example, a relative modular category which is locally finite abelian with finitely many projective indecomposable objects modulo $\Zt$ in each degree $g \in \Gr$ is TQFT finite. 

\begin{thm}[{Theorem \ref{thm:finDimTQFT}}]
\label{thm:finDimTQFTIntro}
Let $\cat$ be a relative modular category which is TQFT finite. Then for each decorated surface $\CS \in \Cob_{\cat}$, the state space $\TQFT_{\cat}(\CS) \in \ZVect_{\C}$ is finite dimensional.
\end{thm}

Theorem \ref{thm:finDimTQFTIntro} provides a general reason for the observed finite dimensionality of state spaces in all known examples, namely those arising from relative modular categories of modules over unrolled quantum groups of complex simple Lie (super)algebras \cite{BCGP,derenzi2020,AGP,ha2022} and those of this paper. Theorem \ref{thm:finDimTQFTIntro} is proved by exhibiting an explicit, combinatorially defined spanning set of $\TQFT_{\cat}(\CS)$ using special $\cat$-colorings of a fixed spine of $\CS$.

In Section \ref{sec:relModgloo} we introduce a non-standard quantization $\Uq$ of $\gloo$. The algebra $\Uq$ is an unrolled version of standard quantizations of $\gloo$ \cite{kulish1989,khoroshkin1991,reshetikhin1992}, in the sense of \cite{CGP2}. Fix $q \in \C \setminus \{0, \pm 1\}$. The superalgebra $\Uq$ is generated by even Cartan generators $E$, $G$, $K^{\pm 1}$ and odd Serre generators $X$, $Y$. The generator $E$ should be viewed as a logarithm of $K$, but this relation is not imposed at the level of the algebra. Instead, we consider the category $\catq$ of all weight $\Uq$-modules on which $K$ acts by $q^E$. Let also $\catInt \subset \catq$ be the full subcategory of weight modules whose $G$-weights are integral; no integrality of $E$-weights is assumed. A natural Hopf superalgebra structure on $\Uq$ gives $\catq$ and $\catInt$ the structure of rigid monoidal categories. We study $\catq$ and $\catInt$ in detail, obtaining complete descriptions of their simple and projective indecomposable objects. The culmination of our results in Section \ref{sec:relModgloo} is summarized as follows.

\begin{thm}[{Theorems \ref{thm:relModArb}, \ref{thm:relModROUOdd}, \ref{thm:relModROUEven}, \ref{thm:relModIntArb}, \ref{thm:relModIntROUOdd}}]
\label{thm:relModIntro}
The categories $\catq$ and $\catInt$ admit relative modular structures which depend on whether or not $q$ is a root of unity and, if so, the parity of the order of the root of unity. In particular, $\catq$ and $\catInt$ are generically semisimple ribbon categories. Moreover, $\catq$ and $\catInt$ are TQFT finite with respect to any of the above relative modular structures.
\end{thm}

More precisely, $\catq$ and $\catInt$ admit relative modular structures for any $q \in \C \setminus \{0,\pm 1\}$, which we refer to as the case of arbitrary $q$, and admit second, distinct, relative modular structures when $q$ is a root of unity. Denote by $\cat$ either of the categories $\catq$ and $\catInt$ with any of the relative modular categories of Theorem \ref{thm:relModIntro} and by $\TQFT_{\cat}: \Cob_{\cat} \rightarrow \ZVect_{\C}$ the associated topological quantum field theory. In all cases, the braided category $\ZVect_{\C}$ is a graded version of complex super vector spaces.

In the remainder of the paper we study in detail $\TQFT_{\cat}$ (Sections \ref{sec:TQFTArb}-\ref{sec:intTQFT}) and their relationship to $\psl$ and $\Uoo$ Chern--Simons and Wess--Zumino--Witten theories (Section \ref{sec:physCompare}). To avoid cumbersome statements, in the introduction we state precise results only for $\cat=\catq$ with $q$ a primitive $r$\textsuperscript{th} root of unity, $r \geq 3$ odd. In this case, the category $\catq$ is graded by $\Gr = \C \slash \Z \times \C \slash \Z$, corresponding to $(E,G)$-weights modulo $\Z \times \Z$, with $\XX= \frac{1}{2} \Z \slash \Z \times \{\p 0\}$ and $\Zt = \Z \times \Z \times \Ztwo$. In the body of the paper we treat all cases of Theorem \ref{thm:relModIntro}.

Our first series of results concerns the $\Zt$-graded vector space $\TQFT_{\cat}(\CS) = \bigoplus_{k \in \Zt} \TQFT_{\cat,k}(\CS)$ assigned to a decorated surface $\CS \in \Cob_{\cat}$. Part of the data of $\CS$ is a cohomology class $\coh \in H^1(\CS_0; \Gr)$ on the underlying closed surface $\CS_0$ of $\CS$. The description of $\TQFT_{\cat}(\CS)$ simplifies considerably when $\coh$ is generic in the sense that $\coh(\gamma) \in \Gr \setminus \XX$ for some simple closed curve $\gamma \subset \CS_0$.
 
\begin{thm}[{Theorems \ref{thm:verlindeArb} and \ref{thm:verlindeROU}}]
\label{thm:verlindeIntro}
Let $\CS$ be a decorated connected surface of genus $g \geq 1$ without marked points and with generic cohomology class $\coh$. For any $(\p \beta, \p b) \in \Gr$, the partition function of $\CS \times S^1_{(\p \beta,\p b)}$, the closed decorated $3$-manifold obtained by crossing $\CS$ with $S^1$ and extending $\coh$ to $\coh \oplus (\p \beta, \p b)$, is
\[
\TQFT_{\catq}(\CS \times S^1_{(\p \beta,\p b)})
=
(-1)^{g+1} r^{2g-1}\sum_{i=0}^{r-1}(q^{\p{\beta}+i}-q^{-\p{\beta}-i})^{2g-2}.
\]
Moreover, the partition function $\TQFT_{\catq}(\CS \times S^1_{(\p \beta,\p b)})$ and state space $\TQFT_{\catq}(\CS)$ are related through the Verlinde formula
\[
\TQFT_{\catq}(\CS \times S^1_{(\p \beta, \p b)})
=
\sum_{(n,n^{\prime}) \in \Z^2} \chi(\TQFT_{\catq,(n,n^{\prime},\bullet)}(\CS)) q^{-2r(\overline{\beta} n^{\prime} + \overline{b} n)},
\]
where $\chi(\TQFT_{\catq,(n,n^{\prime},\bullet)}(\CS))$ denotes the Euler characteristic of the $\Ztwo$-graded subspace of $\TQFT_{\catq}(\CS)$ consisting of vectors with $\Zt$-degree of the form $(n,n^{\prime},\bullet)$.
\end{thm}

Theorem \ref{thm:verlindeIntro} is proved using an explicit surgery presentation of trivial circle fibrations and the representation theoretic results of Section \ref{sec:relModgloo}. The strategy of proof is similar to its counterpart for topological quantum field theories arising from the unrolled quantum group $\Uqsl$ \cite{BCGP}. In the body of the paper we allow $\CS$ to carry marked points, in which case $\TQFT_{\catq}(\CS \times S^1_{(\p \beta,\p b)})$ depends also on $\p b$.

To obtain a more detailed understanding of $\TQFT_{\cat}(\CS)$, we first prove in Theorems \ref{thm:genusgStateSpaceArb} and \ref{thm:genusgStateSpaceROU} that, in the present class of examples and under the assumed genericity of $\coh$, the general spanning set of $\TQFT_{\cat}(\CS)$ constructed in Theorem \ref{thm:finDimTQFTIntro} can be reduced to a much smaller set. Theorems \ref{thm:genusgStateSpaceArb} and \ref{thm:genusgStateSpaceROU} can be seen as vanishing results, asserting that $\TQFT_{\cat}(\CS)$ is concentrated in a restricted set of $\Zt$-degrees. Using these results, we prove that
\[
\lim_{\overline{\beta} \rightarrow \p{\frac{1}{4}}} \TQFT_{\catq}(\CS \times S^1_{(\p \beta, \p b)}) = \dim_{\C} \TQFT_{\catq}(\CS).
\]
Together with Theorem \ref{thm:verlindeIntro}, this leads to the following explicit description of state spaces of generic surfaces.

\begin{thm}[{Corollaries \ref{cor:stateBasisArb} and \ref{cor:stateBasisROU}}]
\label{cor:stateBasisIntro}
Let $\CS$ be a decorated connected surface of genus $g \geq 1$ without marked points and with generic cohomology class. Then $\TQFT_{\catq,-k}(\CS)$ is trivial unless $k=(0,d, \p d)$ for some $d \in [-(g-1),g-1] \cap r\Z$, in which case it is of dimension $r^{2g}{2g-2 \choose g-1-\vert d \vert}$. In particular, the total dimension of $\TQFT_{\catq}(\CS)$ is
\[
\dim_{\C} \TQFT_{\catq}(\CS)
=
r^{2g} \sum_{n^{\prime}=-\left \lfloor{\frac{g-1}{r}}\right \rfloor}^{\left \lfloor{\frac{g-1}{r}}\right \rfloor} {2g-2 \choose g-1-\vert n^{\prime} \vert r}.
\]
\end{thm}

When the cohomology class $\coh$ is not generic, the vector space $\TQFT_{\cat}(\CS)$ is considerably more complicated. In this setting we restrict attention to the torus, where we again obtain a complete description of $\TQFT_{\cat}(\CS)$. We prove that $\TQFT_{\catq}(\CS)=\TQFT_{\cat,0}(\CS)$ is two dimensional for arbitrary $q$ (Propositions \ref{prop:nonGenCohTorusArb} and \ref{prop:nonGenCohTorusHighDegArb}) and that $\TQFT_{\catq}(\CS)=\TQFT_{\cat,0}(\CS)$ is $r^2+1$ dimensional for $q$ a primitive $r$\textsuperscript{th} root of unity (Proposition \ref{prop:nonGenCohTorusROU}). The result for arbitrary $q$ is particularly surprising since its contrasts the conjectured behavior of TQFTs constructed from $\Uqsl$ \cite{BCGP}. We construct explicit bases of $\TQFT_{\cat}(\CS)$ to prove the following result.

\begin{thm}[{Theorems \ref{thm:MCGArb} and \ref{thm:MCGROU}}]
\label{thm:MCGROUIntro}
Let $\CS$ be a decorated connected surface of genus one without marked points and with non-generic cohomology class $\coh$. The mapping class group action of $SL(2, \Z)$ on $\TQFT_{\catq}(\CS)$ admits an explicit description which, in particular, shows that the Dehn twist acts with infinite order.
\end{thm}

Mapping class group actions with properties similar to those of Theorem \ref{thm:MCGROUIntro} for arbitrary $q$ are obtained using the representation theory of $\Uqsl$ at a root of unity in \cite{BCGP}.

For a given $q$, the theories $\TQFT_{\catq}$ and $\TQFT_{\catInt}$ are closely related. \emph{A priori}, the significant difference in the gradings of these categories- the grading group for $\catInt$ is much smaller than that of $\catq$- could lead to significant differences in $\TQFT_{\catq}$ and $\TQFT_{\catInt}$. However, the constraints on the $\Zt$-support of $\TQFT_{\catq}$, as in Theorem \ref{cor:stateBasisIntro}, show that this is not the case. In particular, Theorems \ref{thm:verlindeIntro}, \ref{cor:stateBasisIntro} and \ref{thm:MCGROUIntro} hold for the relative modular categories $\catInt$, with essentially the same proofs.

Finally, in Section \ref{sec:physCompare} we connect our results with the physics literature. When $q$ is arbitrary, Proposal \ref{proposal:psl} states that $\TQFT_{\catInt}$ is Chern--Simons theory with gauge Lie superalgebra $\psl$, the two dimensional purely odd Lie algebra, as studied by Mikhaylov \cite{Mik2015}. We observe that $\TQFT_{\catInt}$ is effectively independent of $q$. This reflects the physical expectation that, because $\psl$ is purely odd, there is no quantization of the level. If instead $q$ is a primitive $r$\textsuperscript{th} root of unity, then Proposal \ref{proposal:uoo} states that $\TQFT_{\catInt}$ is $U(1 \vert 1)$ Chern--Simons theory at level $r$, as studied by Rozansky--Saleur and Mikhaylov \cite{rozansky1992,rozansky1993,rozansky1994,Mik2015}. More precisely, since $\psl$ and $\Uoo$ Chern--Simons theories arise as topological twists of supersymmetric quantum field theories \cite{rozansky1997,GaiottoWitten-Janus,kapustin2009b}, they are expected to be examples of topological quantum field theories valued in derived or differential graded categories and we expect the theories $\TQFT_{\catInt}$ to be their homological truncations. Proposals \ref{proposal:psl} and \ref{proposal:uoo} use in an essential way the group $\C^{\times}$ of global symmetries of these supergroup Chern--Simons theories which allows them to be coupled to background flat $\C^{\times}$-connections. As evidence for these proposals, we match the Verlinde formulae (Theorem \ref{thm:verlindeIntro}) and dimension formulae for generic state spaces (Theorem \ref{cor:stateBasisIntro}) with physical predictions \cite{rozansky1993,rozansky1994,Mik2015}. We also match the mapping class group action of Theorem \ref{thm:MCGROUIntro} with that obtained using $\Uoo$ Wess--Zumino--Witten theory \cite{rozansky1993}. Theorem \ref{thm:MCGROUIntro} is very similar to mapping class group actions obtained using $\Uoo$ Chern--Simons theory \cite{Mik2015} and combinatorial quantization \cite{aghaei2018}; we give a precise comparison in Section \ref{sec:UCompare}. The connection between $\TQFT_{\catq}$ and $\TQFT_{\catInt}$ and the Alexander polynomial, discussed in Section \ref{sec:alexPoly}, matches the expected connection for $\Uoo$ Wess--Zumino--Witten theory \cite{rozansky1992,rozansky1993} and $\psl$ and $\Uoo$ Chern--Simons theories \cite{rozansky1994,Mik2015}. However, the results of this paper go beyond what has appeared in the physics literature. This includes the construction of a topological quantum field theory and an explicit description of the state spaces of generic surfaces of all genera. As discussed above, the final point is strictly stronger than the Verlinde formula in isolation.

Creutzig, Dimofte, Garner and the first author proposed in \cite{creutzig2021} that the topological quantum field theory associated to the relative modular category of weight modules over the unrolled quantum group $U_q^H(\mathfrak{sl}(n))$ admits a physical realization as the homological truncation of a topological $A$-twist of $3$d $\mathcal{N}=4$ Chern--Simons-matter theory with gauge group $SU(n)$. According to Kapustin and Saulina \cite{kapustin2009b}, Chern--Simons theory with gauge group a Lie supergroup can be realized as a gauged affine Rozansky--Witten theory and is therefore a $B$-twisted mirror of the theories studied in \cite{creutzig2021}. Implications of $3$d mirror symmetry at the level of topological quantum field theories are conjectured in \cite{creutzig2021}. The results of \cite{BCGP} and this paper provide mathematical foundations and calculations for the $A$-sides and $B$-sides, respectively, of these conjectures. It would be very interesting to use these results to study concrete instances of these conjectures.

In a different direction, it would be very interesting to study the relationship between the topological quantum field theories of this paper, constructed from the representation theory of $\Uq$, with decategorifications of the Heegaard Floer theory of \cite{manion2019,manion2020}, constructed from the categorical representation theory of $U_q(\gloo^+)$.

\subsubsection*{Acknowledgements}
The authors thank C. Blanchet, F. Costantino, M. De Renzi, T. Dimofte, J. Kujawa, A. Manion, B. Patureau-Mirand, M. Rupert, J.-L. Spellmann and C. Stroppel for discussions related to this paper. N.G.\ is partially supported by NSF grants DMS-1664387 and DMS-2104497.  He would also like to thank the Max Planck Institute for Mathematics in Bonn for its hospitality during work on this paper.  M.B.Y. is partially supported by a Simons Foundation Collaboration Grant for Mathematicians (Award ID 853541).

\section{Preliminary material}\label{sec:preliminaries}

Let $\kk$ be an algebraically closed ground field. 

\subsection{Ribbon categories}
\label{sec:ribbCat}
We refer the reader to \cite{etingof2015} for background on monoidal categories.

Let $\cat$ be a $\kk$-linear monoidal category. Throughout the paper, we assume that the functor $\otimes : \cat \times \cat \rightarrow \cat$ is $\kk$-bilinear, the monoidal unit $\mathbb{I}$ is simple and the $\kk$-algebra map $\kk \to \End_\cat(\unit), k \mapsto k \cdot \Id_\unit$, is an isomorphism. If, in addition, $\cat$ is rigid, braided and has a compatible twist $\theta=\{\theta_V: V \rightarrow V\}_{V \in \cat}$, then $\cat$ is called a \emph{$\kk$-linear ribbon category}. The left and right duality structure maps are denoted
\begin{equation*}
\overleftarrow{\mathrm{ev}}_V:V^* \otimes V \to \unit,
\qquad \overleftarrow{\mathrm{coev}}_V:\unit \to V  \otimes V^*
\end{equation*}
\and
\begin{equation*}
\qquad \overrightarrow{\mathrm{ev}}_V:V\otimes V^*  \to \unit,
\qquad \overrightarrow{\mathrm{coev}}_V: \unit \to V^*  \otimes V, \end{equation*}
respectively, and the braiding is $c=\{c_{V,W}: V \otimes W \rightarrow W \otimes V\}_{V,W \in \cat}$. The \emph{quantum dimension} of $V \in \cat$ is
\[
\qdim V =\ev_V \circ \tcoev_V \in \End_{\cat}(\unit).
\]
An object $V\in\cat$ is called \emph{regular} if $\tev_V$ is an epimorphism.

Our conventions for diagrammatic computations with ribbon categories are that diagrams are read from left to right, bottom to top and
\[
\Id_V
=
\begin{tikzpicture}[anchorbase]
\draw[->,thick] (0,0) -- node[left] {\small$V$} (0,1);
\end{tikzpicture}
\qquad \qquad
\Id_{V^*}
=
\begin{tikzpicture}[anchorbase]
\draw[<-,thick] (0,0) -- node[left] {\small$V$} (0,1);
\end{tikzpicture}
\]
\[
\tev_V
=
\begin{tikzpicture}[anchorbase]
\draw[->,thick] (0,0)  arc (0:180:0.5 and 0.75);
\node at (-1.4,0)  {$V$};
\end{tikzpicture}
\qquad, \qquad
\tcoev_V
=
\begin{tikzpicture}[anchorbase]
\draw[<-,thick] (0,0)  arc (180:360:0.5 and 0.75);
\node at (-0.5,-0.1)  {$V$};
\end{tikzpicture}
\]
\[
\ev_V
=
\begin{tikzpicture}[anchorbase]
\draw[<-,thick] (0,0)  arc (0:180:0.5 and 0.75);
\node at (0.4,0)  {$V$};
\end{tikzpicture}
\qquad, \qquad
\coev_V
=
\begin{tikzpicture}[anchorbase]
\draw[->,thick] (0,0)  arc (180:360:0.5 and 0.75);
\node at (1.5,-0.1)  {$V$};
\end{tikzpicture}
\]
\[
c_{V,W}
=
\begin{tikzpicture}[anchorbase]
\draw[->,thick] (0.5,0) -- node[right,near start] {\small $W$} (0,1);
\draw[->,thick,cross line] (0,0) -- node[left,near start] {\small$V$} (0.5,1);
\end{tikzpicture}
\qquad, \qquad
\theta_V
=
\begin{tikzpicture}[anchorbase]
\draw[->,thick,rounded corners=8pt] (0.25,0.25) -- (0,0.5) -- (0,1);
\draw[thick,rounded corners=8pt,cross line] (0,0) -- (0,0.5) -- (0.25,0.75);
\draw[thick] (0.25,0.75) to [out=30,in=330] (0.25,0.25);
\node at (-0.2,0.2)  {$V$};
\end{tikzpicture}.
\]
A morphism
$f:V_1\otimes{\cdots}\otimes V_n \rightarrow W_1\otimes{\cdots}\otimes W_m$ in $\cat$ is represented by the diagram
$$ 
\xymatrix{\;\\
 *+[F]\txt{ \: \; $f$ \; \;} 
 \ar@< 8pt>[u]^{W_1}_{... \hspace{1pt}}
  \ar@< -8pt>[u]_{W_m}\\
 \ar@< 8pt>[u]^{V_1}_{... \hspace{1pt}}
  \ar@< -8pt>[u]_{V_n}
  }
$$
whose box is called a {\it coupon}. Following Turaev \cite[\S I.2]{Tu}, a \emph{ribbon graph} in an oriented manifold $M$ is a compact oriented surface embedded in $M$ which decomposes into elementary pieces, consisting of bands, annuli and coupons, and is the thickening of an oriented graph. In particular, the vertices of the graph which lie in the interior $\mathring{M} =M \setminus\partial M$ are thickened to coupons. A \emph{$\cat$-colored ribbon graph} is a ribbon graph whose (thickened) edges are colored by objects of $\cat$ and whose coupons are colored by morphisms of~$\cat$. The intersection of a $\cat$-colored ribbon graph with $\partial M$ is required to be empty or consist of univalent vertices. In diagrams, we represent the induced framing of the core of the ribbon graph using the blackboard framing.

Associated to a $\kk$-linear ribbon category $\cat$ is the \emph{Reshetikhin--Turaev ribbon functor} $F_{\cat} : \mathcal{R}_{\cat} \rightarrow \cat$, where $\mathcal{R}_{\cat}$ is the ribbon category of $\cat$-colored ribbon graphs in $\R^2 \times [0,1]$ \cite[Theorem I.2.5]{Tu}. Given a $(1,1)$-tangle $T_V$ whose open strand is colored by a simple object $V \in \cat$, define $\langle T_V \rangle \in \kk$ by the following equality in $\End_{\cat}(V)$:
$$ F_{\cat}(T_V)=\langle T_{V} \rangle \cdot \Id_{V}. $$
Two formal $\kk$-linear combinations of $\cat$-colored ribbon graphs are called \emph{skein equivalent} if their images under $F_{\cat}$ agree. The corresponding equivalence relation is denoted $\dot{=}$.

\subsection{Modified traces}
\label{sec:mTrace}

Let $\cat$ be a $\kk$-linear ribbon category.

\begin{definition}
An {\em ideal} in $\cat$ is a full subcategory $\ideal \subset \cat$ which
\begin{enumerate}
\item is stable under retracts: if $W \in \ideal$ and $V \in \cat$ and there exist morphisms $f:V\to W$ and $g:W\to V$ such that $g \circ f=\Id_{V}$, then $V\in \ideal$, and

\item absorbs tensor products: if $U\in\ideal$ and $V \in \cat$, then $U\otimes V \in \ideal$.
\end{enumerate}
\end{definition}

Since $\cat$ is ribbon, an ideal also absorbs tensor products from the left.

\begin{example}
An object $P \in \cat$ is called \emph{projective} if every epimorphism $V \rightarrow P$ admits a section. The full subcategory $\Proj \subset \cat$ of projective objects is an ideal.
\end{example}

Let $V, W \in \cat$. Recall that the right partial trace $\ptr_W : \End_{\cat}(V \otimes W) \rightarrow \End_{\cat}(V)$ is defined by
\[
\ptr_W(f) = (\Id_V \otimes \ev_W) \circ (f \otimes \Id_{W^*}) \circ (\Id_V \otimes \tcoev_W).
\]

\begin{definition}[{\cite[\S 3]{GKP1}}]\label{def:mtrace}
\begin{enumerate}
\item A \emph{modified trace} (or \emph{m-trace}) on an ideal $\ideal \subset \cat$ is a family of $\kk$-linear functions 
$\mt=\{\mt_V:\End_\cat(V)\to\kk\}_{V\in\ideal}$ which satisfies the following:
\begin{enumerate}
 \item \emph{Cyclicity property}: $\mt_V(f \circ g)=\mt_W(g \circ f)$ for all $V,W \in \ideal$ and $f \in \Hom_{\cat}(W,V)$ and $g \in \Hom_{\cat}(V,W)$.
\item \emph{Partial trace property}: 
$\mt_{V\otimes W}(f)=\mt_V(\ptr_W(f))$ for all $V \in \ideal$, $W \in \cat$ and $f\in\End_{\cat}(V\otimes W)$.
\end{enumerate}

\item Given an m-trace $\mt$ on $\ideal$, the {\em modified dimension} of $V\in \ideal$ is $\qd(V)=\mt_V(\Id_V) \in \kk$.
\end{enumerate}
\end{definition}

\subsection{Relative modular categories}\label{sec:relModDef}

We recall a number of definitions from \cite{CGP14,derenzi2022}.

\begin{definition}  
Let $\catb$ be a $\kk$-linear category. 
\begin{enumerate}
\item A set $\mathcal{E}=\{ V_i \mid i \in J \}$ of objects of $\catb$ is \emph{dominating} if for any $V \in \catb$ there exist $\{i_1,\dots,i_m \} \subseteq J $ and morphisms $\iota_k \in \Hom_{\catb}(V_{i_k},V)$ and $s_k \in \Hom_{\catb}(V,V_{i_k})$ such that $\Id_{V}=\sum_{k=1}^m \iota_k \circ s_k$.

\item A dominating set $\mathcal{E}$ is \emph{completely reduced} if $\dim_\kk \Hom_{\catb}(V_i,V_j)=\delta_{i,j}$ for all $i,j \in J$.
\end{enumerate}
\end{definition}

Let $\cat$ be a $\kk$-linear ribbon category and $\Zt$ an additive abelian group. We often view $\Zt$ as a discrete monoidal category with object set $\Zt$.

\begin{definition}
\label{def:free}
A \emph{free realisation of $\Zt$ in $\cat$} is a monoidal functor $\sigma: \Zt \rightarrow \cat$ such that
\begin{enumerate}
  \item $\sigma(0)=\unit$,
  \item $\qdim \sigma(k) \in \{\pm 1\}$ for all $k\in \Zt$,
  \item $\theta_{\sigma(k)}=\Id_{\sigma(k)} \text{ for all } k \in \Zt $, and
\item for any simple object $V \in \cat$, we have $V\otimes \sigma(k) \simeq V$ if and only if $k=0$.  
\end{enumerate}
\end{definition}

We often identify a free realisation $\sigma : \Zt \rightarrow \cat$ with a collection of objects $\{\sigma(k)\}_{k \in \Zt}$, omitting from the notation the monoidal coherence data $\sigma(k_1) \otimes \sigma(k_2) \xrightarrow[]{\sim} \sigma(k_1+k_2)$, $k_1,k_2 \in \Zt$.

\begin{remark}
The definition of a free realisation was first given in \cite[\S 4.3]{CGP14} with the condition that $\qdim \sigma(k) =1$ for all $k\in \Zt$. Definition \ref{def:free} first appeared in \cite[\S 1.3]{derenzi2022}.
\end{remark}

\begin{definition}
\label{def:Gstr}
Let $\Gr$ be an additive abelian group.  A \emph{$\Gr$-grading on $\cat$} is an equivalence of $\kk$-linear categories $\cat \simeq \bigoplus_{g \in \Gr} \cat_g$, where $\{ \cat_g \}_{g \in \Gr}$ are full subcategories of $\cat$ satisfying the following conditions:
  \begin{enumerate}
  \item $\unit \in \cat_0$,
  \item  if $V\in\cat_g$,  then  $V^{*}\in\cat_{-g}$, and
  \item  if $V\in\cat_g$ and $V^{\prime}\in\cat_{g'}$, then $V\otimes
    V^{\prime}\in\cat_{g+g'}$.
    \end{enumerate}
\end{definition}

It follows from the definition that if $V\in\cat_g$, $V^{\prime}\in\cat_{g'}$ and $\Hom_\cat(V,V^{\prime})\neq 0$, then $g=g'$. Note also that the full subcategory $\cat_0 \subset \cat$ is ribbon.

\begin{definition}\label{def:smallsymm}
A subset $\XX$ of an abelian group $\Gr$ is
\begin{enumerate}
\item \emph{symmetric} if $\XX=-\XX$ and
\item \emph{small} if $\bigcup_{i=1}^n (g_i+\XX) \neq \Gr$ for all $g_1,\ldots ,g_n\in \Gr$.
\end{enumerate}
\end{definition}

 \begin{definition}
 \label{def:preMod}
 Let $\Gr$, $\Zt$ be abelian groups, $\XX \subset \Gr$ a small symmetric subset and $\cat$ a $\kk$-linear ribbon category with the following data:
 \begin{enumerate}
 \item a $\Gr$-grading on $\cat$,
 \item a free realisation $\sigma$ of $\Zt$ in $\cat_0$, and
 \item a non-zero m-trace $\mt$ on the ideal of projective objects of $\cat$.
 \end{enumerate}
 A category $\cat$ with this data is called a \emph{pre-modular $\Gr$-category relative to $(\Zt,\XX)$} if it has the following properties:
 \begin{enumerate}
\item \emph{Generic semisimplicity}: \label{def:genSS} For every $g \in \Gr \setminus \XX$, there exists a finite set of regular simple objects $\Theta(g):=\{ V_i \mid i \in I_g  \}$ such that
$$\Theta(g) \otimes \sigma(\Zt):=\{ V_i \otimes \sigma(k) \mid i \in I_g, \; k\in \Zt  \}$$
is a completely reduced dominating set for $\cat_g$.

\item \emph{Compatibility}: \label{def:compat}
There exists a bicharacter $\psi: \Gr \times \Zt \rightarrow \kk^{\times}$ such that
  \begin{equation}
    \label{eq:psi}
    c_{\sigma(k),V}\circ c_{V,\sigma(k)}= \psi(g,k) \cdot  \Id_{V \otimes \sigma(k)}
  \end{equation}
for any $g\in \Gr$, $V \in \cat_g$ and $k \in \Zt$.
\end{enumerate}
 \end{definition}


\begin{definition}
\label{def:ndeg}
Let $\cat$ be a pre-modular $\Gr$-category relative to $(\Zt,\XX)$.
\begin{enumerate}
\item For each $g \in \Gr \setminus \XX$, the \emph{Kirby color of index $g$} is the formal $\kk$-linear combination of objects $$\Omega_g:= \sum_{i \in I_g}\qd(V_i) \cdot V_i.$$ 
\item For each $g \in \Gr \setminus \XX$ and $V\in \cat_g$, the \emph{stabilization coefficients} $\Delta_\pm\in \kk$ are defined by the skein equivalences
$$\epsh{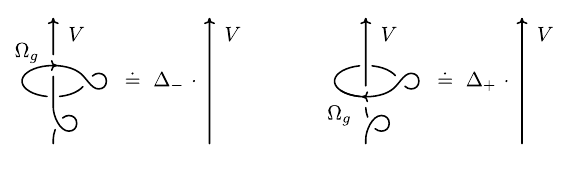}{20ex}
$$
\item The pre-modular $\Gr$-category $\cat$ is called \emph{non-degenerate} if $\Delta_{+}\Delta_{-}\neq 0$. 
\end{enumerate}

\end{definition}

As the notation suggests, $\Delta_{\pm}$ are independent of $g$ and $V \in \cat_g$ \cite[Lemma 5.10]{CGP14}.


\begin{definition}
\label{def:modG}
A \emph{modular $\Gr$-category relative to $(\Zt,\XX)$} is a pre-modular $\Gr$-category $\mathcal{C}$ relative to $(\Zt,\XX)$ for which there exists a scalar $\zeta \in \kk^{\times}$, called the \emph{relative modularity parameter}, such that for any $g,h \in \Gr \setminus \XX$ and $i,j \in I_g$ the skein equivalence
  \begin{equation}\label{eq:mod}
    \epsh{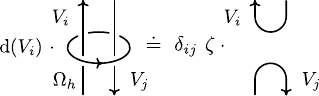}{12ex}\vspace*{4ex}
  \end{equation}
holds.
\end{definition}

The relative modularity parameter satisfies $\zeta = \Delta_+ \Delta_-$ \cite[Proposition 1.2]{derenzi2022}. In particular, modular $\Gr$-categories are non-degenerate pre-modular.

\subsection{Invariants of closed $3$-manifolds}
\label{sec:CGPInvt}

Let $\cat$ be a non-degenerate pre-modular $\Gr$-category relative to $(\Zt,\XX)$. The data $\cat$ was used in \cite{CGP14} to define invariants of decorated $3$-manifolds. We briefly recall this construction. Note that relative modularity in \cite[Definition 4.2]{CGP14} is what is called here relative pre-modularity in Definition \ref{def:preMod}. All manifolds in this paper are assumed oriented and compact.

A \emph{compatible triple} $(M,T, \coh)$ consists of a closed connected $3$-manifold $M$, a $\cat$-colored ribbon graph $T \subset M$ and a cohomology class $\coh\in H^{1}(M\setminus T ; \Gr)$ such that the $\cat$-coloring of $T$ is \emph{$\coh$-compatible}, in the sense that each oriented edge $e$ of $T$ is colored by a non-zero object of $\cat_{\coh(m_e)}$, where $m_e$ is the oriented meridian of $e$. A surgery presentation $L\subset S^3$ for $M$ is called \emph{computable} for $(M,T,\coh)$ if one of the following conditions holds:
  \begin{enumerate}
  \item $L\neq \emptyset$ and $\coh(m_{L_i}) \in \Gr \setminus \XX$ for all connected components $L_i$ of $L$.  
\item $L=\emptyset$ and there exists an edge $e$ of $T$ with $\coh(m_e) \in \Gr \setminus \XX$.
\end{enumerate}

Let $R$ be a $\cat$-colored ribbon graph. Suppose that at least one edge of $R$ is colored by a generic simple object $V \in \cat$ and let $T_V$ be the $(1,1)$-ribbon graph obtained from $R$ by cutting an edge labeled by $V$. Properties of m-traces imply that
$$F^{\prime}_{\cat}(R) :=\mt_V(T_V) =\qd(V) \langle T_V \rangle \in \kk$$ is a well-defined invariant of $R$. See \cite{GPT,GKP1}. In \cite[\S 2]{CGP14} it is shown that $F^{\prime}_{\cat}$ can be used to define a Reshetikhin--Turaev-type diffeomorphism invariant of triples $(M,T,\coh)$. See equation \eqref{eq:cgpInvt} below for a slight generalization.

\subsection{A topological quantum field theory extension}
\label{sec:CGPTQFT}

When $\cat$ is a modular $\Gr$-category relative to $(\Zt,\XX)$, the decorated $3$-manifold invariants of Section \ref{sec:CGPInvt} is part of a once-extended topological quantum field theory (TQFT) \cite{BCGP, derenzi2022}. In this paper, we restrict attention to the truncation of this theory to a non-extended TQFT.

A \emph{decorated surface} $\CS=(\Sigma, \{p_i\}, \coh, \mathcal{L})$ consists of
\begin{itemize}
\item a closed surface $\Sigma$ with a set $*$ of distinguished base points, exactly one for each connected component of $\Sigma$,
\item a (possibly empty) finite set $\{p_i\} \subset \Sigma \setminus *$ of oriented framed $\cat$-colored points, where the coloring of $p_i$ is an object of some $\cat_{g_i}$,
\item a cohomology class $\coh \in H^1(\Sigma\setminus\{p_i\}, * ;\Gr) \simeq H^1(\Sigma\setminus\{p_i\} ;\Gr)$ such that $\omega(m_i) = g_i$, where $m_i$ is the oriented boundary of a regular neighborhood of $p_i$, and
\item a Lagrangian subspace ${\mathcal L}\subset H_1(\Sigma; \R)$.
\end{itemize}

A \emph{decorated cobordism} $\mathcal M = (M,T,\coh,m) : \CS_1 \rightarrow \CS_2$ between decorated surfaces consists of
\begin{itemize}
\item a cobordism $M: \Sigma_1 \rightarrow \Sigma_2$,
\item a $\cat$-colored ribbon graph $T \subset M$ whose coloring is compatible with that of the marked points of $\CS_i$, $i=1,2$,
\item a cohomology class $\coh \in H^1(M\setminus T, *_1 \cup *_2; \Gr)$ for which $T$ is $\coh$-compatible and which restricts to $\omega_i$ on $\Sigma_i$, $i=1,2$, and
\item an integer $m \in\Z$ called the \emph{signature defect}.
\end{itemize}
This data is required to be \emph{admissible} in the sense that for each connected component $M_c$ of $M$ which is disjoint from the incoming boundary $\Sigma_1$, there exists an edge of $T \cap M_c$ colored by a projective object or there exists an embedded closed oriented curve $\gamma \subset M_c$ such that $\coh(\gamma) \in \Gr \setminus \XX$.

When it will not cause confusion, we refer to $\mathcal{S}$ and $\mathcal{M}$ simply as surfaces and cobordisms, respectively. Decorated surfaces and their diffeomorphism classes of decorated cobordisms form a category $\Cob_{\cat}$. Disjoint union gives $\Cob_{\cat}$ a symmetric monoidal structure.

Fix a square root $\D \in \kk^{\times}$ of $\Delta_+ \Delta_-$ and set $\delta=\frac{\D}{\Delta_{-}}$. Given a closed cobordism $\mathcal M=(M,T,\coh,m)$ with computable surgery presentation $L \subset S^3$, define 
\begin{equation}
\label{eq:cgpInvt}
\CGP_\cat(M,T,\coh,m)=\D^{-1-l} \delta^{m-\sigma(L)} F'_{\cat}(L\cup T) \in \kk.
\end{equation}
Here $l$ is the number of connected components of $L$, $\sigma(L)$ is the signature of the linking matrix of $L$ and each component $L_i$ of $L$ is colored by the Kirby color $\Omega_{\coh(m_{L_i})}$. Setting $m=0$ in equation \eqref{eq:cgpInvt} recovers a scalar multiple of the decorated $3$-manifold invariants of \cite{CGP14}. 

Let $\ZVect_{\kk}$ be the monoidal category of $\Zt$-graded vector spaces over $\kk$ and their degree preserving morphisms. Let $\gamma: \Zt \times \Zt \rightarrow \{\pm 1\}$ be the unique pairing which makes the diagrams
\[
\begin{tikzpicture}[baseline= (a).base]
\node[scale=1.0] (a) at (0,0){
\begin{tikzcd}[column sep=7.0em,row sep=2.0em]
\sigma(k_1) \otimes \sigma(k_2) \arrow{r}[above]{c_{\sigma(k_1),\sigma(k_2)}} \arrow{d}[left]{\wr} & \sigma(k_2) \otimes \sigma(k_1) \arrow{d}[right]{\wr} \\
\sigma(k_1+k_2) \arrow{r}[below]{\gamma(k_1,k_2) \cdot \Id_{\sigma(k_1+k_2)}} & \sigma(k_2+k_1)
\end{tikzcd}
};
\end{tikzpicture},
\qquad
k_1, k_2 \in \Zt
\]
commute. The vertical arrows are the monoidal coherence data of the free realisation. The pairing $\gamma$ induces a symmetric braiding on $\ZVect_{\kk}$.

\begin{theorem}[{\cite[Theorem 6.2]{derenzi2022}}]
\label{thm:relModTQFT}
A modular $\Gr$-category $\cat$ relative to $(\Zt,\XX)$ with a choice $\D$ of square root of $\Delta_+ \Delta_-$ defines a symmetric monoidal functor $\TQFT_{\cat}: \Cob_{\cat} \rightarrow \ZVect_{\kk}$ whose values on closed cobordisms coincide with $\CGP_{\cat}$.
\end{theorem}

The TQFT $\TQFT_{\cat}$, which we denote by $\TQFT$ if it will not cause confusion, can be described as follows. Given a decorated surface $\CS$, let $\mathcal{V}(\CS)$ be the (infinite dimensional) vector space over $\kk$ with basis the set of all decorated cobordisms $\emptyset\to \CS$ and $\mathcal{V}'(\CS)$ the vector space with basis the set of all decorated cobordisms $ \CS\to \emptyset$. Define a pairing
\[
\langle -, - \rangle: \mathcal{V}' (\CS) \otimes_{\kk} \mathcal{V} (\CS) \to \kk,
\qquad
\langle \CM', \mathcal \CM\rangle = \CGP_\cat (\CM' \circ \CM)
\]
and let $\state (\CS)$ be the quotient of $\mathcal{V} (\mathcal  S)$ by the right kernel of $\langle -, - \rangle$:
$$ \state (\CS) :=\mathcal{V}(\CS) / \text{kerR}\langle -, - \rangle. 
$$
The assignment $\CS \mapsto \state(\CS)$ extends to a functor $\state : \Cob_{\cat} \rightarrow \Vect_{\kk}$ which, however, is not in general monoidal. To remedy this, fix $g_0 \in \Gr$ and a simple projective object $V_{g_0} \in \cat_{g_0}$. For each $k\in \Zt$, denote by $\hS_k$ the decorated $2$-sphere
\[
\hS_{k} = (S^2,\{(V_{g_0},1),(\sigma(k),1),(V_{g_0},-1)\},\coh,\{ 0 \}) \in \Cob_{\cat}
\]
determined by the oriented colored points $\{(V_{g_0},1),(\sigma(k),1),(V_{g_0},-1)\}$ and cohomology class $\coh$ uniquely determined by compatibility. See \cite[\S 4.5]{BCGP}, \cite[\S 5.1]{derenzi2022}. The TQFT state space of $\CS$ is defined to be the $\Zt$-graded vector space
\begin{equation}
\label{eq:stateSpace}
\TQFT(\CS)=\bigoplus_{k\in\Zt} \TQFT_k(\CS)
\end{equation}
where $\TQFT_k(\CS)= \state(\CS\sqcup  \hS_k)$. Up to monoidal natural isomorphism, $\TQFT$ is independent of $g_0$ and $V_{g_0}$.

\begin{lemma}[{\cite[Lemmas 4.7 and 4.8]{derenzi2022}}] \label{lem:PropOfCGP}
The state space $\TQFT(\CS)$ satisfies the following properties.
\begin{enumerate}
\item\label{ite:spanning} Let $\Sigma$ be the underlying surface of $\CS$ and $M$ a connected manifold with $\partial M =\Sigma$. Then the cobordisms $\{\mathcal{M}=(M,T, \coh,0): \emptyset\to \CS\}_{T,\coh}$ span $\state(\CS)$.
\item \label{ite:zero} A $\kk$-linear combination $\sum_i a_i  \mathcal{M}_i$ of cobordisms $\mathcal{M}_i: \emptyset \rightarrow \CS$ is zero in $\state(\CS)$ if and only if $\sum_i a_i \langle \mathcal{M}_i, \mathcal{M}\rangle=0$ for all cobordisms $\mathcal{M} : \CS \rightarrow \emptyset$.  
\end{enumerate}
\end{lemma}

Although $\mathcal{V}(\CS)$ is infinite dimensional, $\TQFT(\CS)$ is finite dimensional in all known examples; see Section \ref{SS:FiniteDimTQFT} below. Moreover, Lemma \ref{lem:PropOfCGP} allows one to make finitely many computations to determine $\TQFT(\CS)$.  The main tools used to do such computations are \emph{$\sigma$-equivalence} and \emph{skein equivalence}, described in \cite[\S 4.1]{BCGP} and \cite[\S 4.1]{derenzi2022}.

The $\Zt$-grading of $\TQFT(\CS)$ can be refined to a $H_0(\Sigma; \Zt)$-grading. This refinement, used in Sections \ref{sec:verlindeArb} and \ref{sec:verlindeROU}, is discussed for $\cat$ the category of weight modules over $\Uqsl$ in \cite[\S 4.7]{BCGP} and can be generalized as follows. Given a decorated surface $\CS$ with underlying surface $\Sigma$ and $\varphi \in H^0(\Sigma; \Gr)$, let $\Id_{\CS}^{\varphi} : \CS \rightarrow \CS$ be the morphism obtained from the cylinder $\Sigma \times [0,1]$ by adding to the pullback of the cohomology class of $\CS$ the unique class $\partial \varphi$ which vanishes on cycles of $\Sigma \times [0,1]$ and takes the value $\varphi(*)$ on the path $* \times [0,1]$, where $*$ is the basepoint of a connected component of $\Sigma$. The group $H^0(\Sigma; \Gr)$ acts on $\state(\CS)$ by post-composition with $\state(\Id_{\CS}^{\varphi})$. Denote by
\[
\state(\CS) = \bigoplus_{\mathbf{k} \in H_0(\Sigma; \Zt)} \state_{\mathbf{k}}(\CS)
\]
the associated spectral decomposition, where
\[
\state_{\mathbf{k}}(\CS)
=
\{v \in \state(\CS) \mid \state(\Id_{\CS}^{\varphi}) v = \big( \prod_{c \in \pi_0(\Sigma)} \psi(k_c, \varphi_c) \big) v \; \; \forall \varphi \in H^0(\Sigma; \Gr)\}
\]
and we have identified $H_0(\Sigma; \Zt) \simeq \bigoplus_{c \in \pi_0(\Sigma)} \Zt$ and $H^0(\Sigma; \Gr) \simeq \bigoplus_{c \in \pi_0(\Sigma)} \Gr$. Write $\vert \mathbf{k} \vert$ for $\sum_{c \in \pi_0(\Sigma)} k_c$. This refined grading has the following properties, which can be proved in the same way as \cite[Proposition 4.28]{BCGP}:
\begin{enumerate}
\item If $\state_{\mathbf{k}}(\CS) \neq 0$, then $\vert \mathbf{k} \vert = 0$.

\item If $\state_{(\mathbf{k}, k_{\infty})}(\CS \sqcup \hS_k) \neq 0$, then $k_{\infty}=k$.
\end{enumerate}
For $\mathbf{k} \in H_0(\Sigma; \Zt)$, set $\TQFT_{\mathbf{k}} (\CS) = \state_{(\mathbf{k},- \vert \mathbf{k} \vert)} (\CS)$. Then we have
\[
\TQFT_k(\CS) = \bigoplus_{\substack{\mathbf{k} \in H_0(\Sigma; \Zt) \\ \vert \mathbf{k} \vert = k}} \TQFT_{\mathbf{k}}(\CS).
\]
In particular, if $\Sigma$ is connected, then $H^0(\Sigma; \Gr)$ acts on $\TQFT_k(\CS)$ by $\TQFT(\Id_{\CS}^{\varphi})v = \psi(\varphi,k) v$, $\varphi \in H^0(\Sigma; \Gr)$.

\begin{remark}
As mentioned above, $\CGP_{\cat}$ is part of a once-extended TQFT $\check{\TQFT}_{\cat}: \CobExt_{\cat} \rightarrow \ZCat_{\kk}$. Here $\CobExt_{\cat}$ is the bicategory of closed decorated $1$-manifolds, their decorated admissible $2$-cobordisms and their equivalence classes of decorated $3$-cobordisms with corners and $\ZCat_{\kk}$ is the bicategory of $\Zt$-graded complete $\kk$-linear categories with symmetric monoidal structure determined by $\gamma$ \cite[Theorem 6.1]{derenzi2022}. The value of $\check{\TQFT}_{\cat}$ on the $S^1$ with cohomology class of holonomy $g \in \Gr$ is Morita equivalent to the ideal of projective objects of $\cat_g$ \cite[Proposition 7.1]{derenzi2022}.
\end{remark}

\subsection{Finite dimensionality of TQFT state spaces}\label{SS:FiniteDimTQFT}
\label{sec:finDim}

Let $\cat$ be a modular $\Gr$-category relative to $(\Zt, \XX)$ with associated TQFT $\TQFT$. In this section we prove that, under mild assumptions on $\cat$, the state spaces of $\TQFT$ are finite dimensional.

An object $V \in \cat$ is called a \emph{direct sum} of a finite set of objects $\{V_i\}_{i \in I}$ of $\cat$, denoted $V = \bigoplus_{i\in I} V_i$, if there exist morphisms $\iota_i \in \Hom_{\cat}(V_i,V)$ and $s_i \in \Hom_{\cat}(V,V_i)$ which satisfy $\Id_V = \sum_{i \in I} \iota_i \circ s_i$ and $s_j \circ \iota_i = \delta_{i,j} \Id_{V_i}$. When $\cat$ is additive, this notion of direct sum agrees with that of the additive structure of $\cat$. An object $V \in \cat$ is called \emph{indecomposable} if $V = \bigoplus_{i\in I} V_i$ implies that there is a unique $i \in I$ such that $V \simeq V_i$ and $V_j =0$ if $j \neq i$.

\begin{definition}
\label{def:relModFinite}
A modular $\Gr$-category $\cat$ relative to $(\Zt,\XX)$ is called \emph{TQFT finite} if it has the following properties:
\begin{enumerate}[label=(F\arabic*)]
\item \label{ite:finCat1} For each $g\in \Gr$, there exists a finite set $\{P_j\}_{j\in J_g}$ of projective indecomposables of $\cath_{g}$ such that any projective indecomposable of $\cath_{g}$ is isomorphic to $P_j\otimes\sigma(k)$ for some $j\in J_g$ and $k\in\Zt$.
\item \label{ite:finCat2} For any projective $P \in \cath$, the vector space $\Hom_\cath(\sigma(k),P)$ is finite dimensional for all $k \in \Zt$ and non-zero for only finitely many $k\in \Zt$.
\item \label{ite:finCat3} the subcategory $\Proj \subset \cat$ of projectives is dominated by the set of projective indecomposables.
\end{enumerate}
\end{definition}

Since $\Proj \subset \cat$ is an ideal, property \ref{ite:finCat2} implies that $\Hom_\cath(V,P) \simeq \Hom_{\cat}(\sigma(0), P \otimes V^*)$ is finite dimensional for all $V \in \cat$ and $P \in \Proj$.

It is straightforward to verify that relative modular categories of modules over unrolled quantum groups of complex simple Lie algebras, as studied in \cite{BCGP,derenzi2020}, are TQFT finite. Further examples will be given in Section \ref{sec:relModgloo} below.

\begin{theorem}\label{thm:finDimTQFT}
Let $\cath$ be a TQFT finite relative modular $\Gr$-category. Then for any decorated surface $\CS \in \Cob_{\cat}$, the vector space $\TQFT(\CS)$ is finite dimensional.  
\end{theorem}

Theorem \ref{thm:finDimTQFT} follows directly from Proposition \ref{prop:FDofTQFT} below and the definition of the state spaces of $\TQFT$, given in equation \eqref{eq:stateSpace}. To formulate Proposition \ref{prop:FDofTQFT}, let $\CS=(\Sigma, \{p_i\}_{i=1}^n, \coh, \mathcal{L}) \in \Cob_{\cat}$. Write $g$ for the genus of $\Sigma$ and $\epsilon_i \in \{ \pm 1\}$ and $X_i \in \cat$ for the orientation and color of the point $p_i$, respectively. Fix $k \in \Zt$. We construct decorated cobordisms $(\tilde{\eta}, \Gamma_c', \coh,0): \emptyset\to \CS\sqcup  \hS_k$, and so vectors in $\TQFT_{k}(\CS)$, as follows. Let $\eta$ be a handlebody bounding $\Sigma$ and $\tilde{\eta} = \eta \setminus \mathring{B}$, where $B \subset \mathring{\eta}$ is a closed $3$-ball. Color the boundary component $\partial B \subset \partial \tilde{\eta}$ by $(V_{g_0},1),(\sigma(k),1)$ and $(V_{g_0},-1)$. When $g=0$, let $\Gamma^{\prime}$ be the ribbon graph consisting of a single coupon with $n+3$ legs attached to the points $\{p_i\}_{i=1}^n$ and $(V_{g_0},1),(\sigma(k),1)$ and $(V_{g_0},-1)$. 
When $g\geq 1$, choose a small $3$-ball $B^3 \subset \eta$ which contains $B$ and $\{p_i\}_{i=1}^n$ and let $\Gamma$ be an oriented spine of $\eta \setminus B$ which is combinatorially equivalent to the following oriented trivalent graph:
\[
\Gamma=
\begin{tikzpicture}[anchorbase] 
\node at (3.5,-2)  {\tiny $e_1$};
\node at (0,-0.6)  {\tiny $e_2$};
\node at (2,-0.6)  {\tiny $e_3$};
\node at (5,-0.6)  {\tiny $e_{g-1}$};
\node at (7,-0.6)  {\tiny $e_g$};

\node at (0,0.6)  {\tiny $f_1$};
\node at (1,0.3)  {\tiny $f_2$};
\node at (2,0.6)  {\tiny $f_3$};
\node at (5,0.6)  {\tiny $f_{2g-5}$};
\node at (6,0.3)  {\tiny $f_{2g-4}$};
\node at (7,0.6)  {\tiny $f_{2g-3}$};

\node at (3.5,0)  {$\cdots$};
        
\draw[thick,decoration={markings, mark=at position 0.275 with {\arrow{>}},mark=at position 0.775 with {\arrow{<}}},postaction={decorate}] (0,0) circle (0.35);
\draw[thick,decoration={markings, mark=at position 0.275 with {\arrow{>}},mark=at position 0.775 with {\arrow{<}}},postaction={decorate}] (2,0) circle (0.35);
\draw[thick,decoration={markings, mark=at position 0.275 with {\arrow{>}},mark=at position 0.775 with {\arrow{<}}},postaction={decorate}] (5,0) circle (0.35);
\draw[thick,decoration={markings, mark=at position 0.275 with {\arrow{>}},mark=at position 0.775 with {\arrow{<}}},postaction={decorate}] (7,0) circle (0.35);
       
\draw[-,thick,decoration={markings, mark=at position 0.5 with {\arrow{>}}},postaction={decorate}] (1.65,0) to (0.35,0);
\draw[-,thick,decoration={markings, mark=at position 0.5 with {\arrow{>}}},postaction={decorate}] (6.65,0) to (5.35,0);
\draw[-,thick,decoration={markings, mark=at position 0.5 with {\arrow{>}}},postaction={decorate}] (-0.35,0) to [out=220,in=-40] (7.35,0);
\end{tikzpicture}.
\]
Let $\Gamma^{\prime}$ the ribbon graph obtained by modifying $\Gamma$ by adding a coupon with $n+4$ legs:
\[
\Gamma^{\prime}=
\qquad
\begin{tikzpicture}[anchorbase] 
\draw[thick,decoration={markings, mark=at position 0.275 with {\arrow{>}},mark=at position 0.775 with {\arrow{<}}},postaction={decorate}] (4.5,0) circle (0.35);
\draw[thick,decoration={markings, mark=at position 0.275 with {\arrow{>}},mark=at position 0.775 with {\arrow{<}}},postaction={decorate}] (6.5,0) circle (0.35);
\node at (4.5,0.6)  {\tiny $f_{2g-5}$};
\node at (5.5,0.3)  {\tiny $f_{2g-4}$};
\node at (6.5,0.6)  {\tiny $f_{2g-3}$};
\node at (4.5,-0.6)  {\tiny $e_{g-1}$};
\node at (6.5,-0.6)  {\tiny $e_g$};

\node at (8.1,-0.3)  {\tiny $e_{g+1}$};
\node at (7.5,-2.0)  {\tiny $e_1$};
\node at (11.2,-0.6)  {\tiny $f_0$};

\draw[-,thick,decoration={markings, mark=at position 0.35 with {\arrow{<}}},postaction={decorate}] (6.85,0) to [out=0,in=90] (10.0,-0.95);
\draw[-,thick,decoration={markings, mark=at position 0.5 with {\arrow{>}}},postaction={decorate}] (6.85,-1.7) to [out=0,in=-90] (10.0,-0.95);
\draw[-,thick,decoration={markings, mark=at position 0.5 with {\arrow{>}}},postaction={decorate}] (6.15,0) to (4.85,0);

\draw[-,thick,decoration={markings, mark=at position 0.5 with {\arrow{>}}},postaction={decorate}] (10.0,-0.95) to [out=0,in=-90] (11.0,0.1);

\draw[thick,draw=black] (10.0,0.1) rectangle ++(1.8,0.3);
\draw[-,thick] (10.1,0.4) to (10.1,1.0);
\draw[-,thick] (10.8,0.4) to (10.8,1.0);
\draw[-,thick,thick,decoration={markings, mark=at position 0.7 with {\arrow{>}}},postaction={decorate}] (11.1,0.4) to (11.1,1.0);
\draw[-,thick,decoration={markings, mark=at position 0.7 with {\arrow{>}}},postaction={decorate}] (11.4,0.4) to (11.4,1.0);
\draw[-,thick,decoration={markings, mark=at position 0.6 with {\arrow{>}}},postaction={decorate}] (11.7,1.0) to (11.7,0.4);
\node at (10.45,0.8)  {$\cdots$};
\node at (10.45,1.3)  {$\overbrace{}^{n}$};

\node at (3.5,0)  {$\cdots$};
\node at (6.5,-1.7)  {$\dots$};
\end{tikzpicture}.
\]
By convention, we set $e_1=e_{g+1}$ when $g=1$. The orientations of $n$ legs of the coupon are determined by the orientations $\{\epsilon_i\}_{i=1}^n$ of $\{p_i\}_{i=1}^n$. There exists a unique element of $H^1(\tilde{\eta} \setminus \Gamma^{\prime}; \Gr)$ whose restriction to $\partial \tilde{\eta}$ is $\coh$ on $\Sigma$ and is compatible with the coloring of $\hS_k$. We denote this cohomology class again by $\coh$, although we stress that it depends on the fixed choice of $k$. 

Finally, with the above notation, a $\cath$-coloring of $\Gamma'$ is called \emph{finite projective} if the following conditions hold:
\begin{enumerate}
\item The coloring is $\coh$-compatible.
\item The edges $e_1,\dots,e_{g+1}, f_0,f_1,\dots,f_{2g-3}$ are colored by projective indecomposables.
\item For each $i=1,\dots,g$, the color of $e_i$ is in the finite set $\{P_j\}_{j \in J_{\coh(m_{e_i})}}$.
\end{enumerate}
Such a coloring $c$ defines a decorated cobordism $(\tilde{\eta}, \Gamma'_c, \coh,0): \emptyset\to \CS\sqcup  \hS_k$.  

\begin{proposition}\label{prop:FDofTQFT}
Let $\cath$ be a TQFT finite modular $\Gr$-category relative to $(\Zt, \XX)$. For any $\CS \in \Cob_{\cat}$, the vector space $\state(\CS\sqcup  \hS_k)$ is finite dimensional for all $k \in \Zt$ and non-zero for only finitely many $k\in \Zt$. Moreover, $\state(\CS\sqcup  \hS_k)$ is spanned by the decorated cobordisms $\{(\tilde{\eta}, \Gamma'_c, \coh,0):  \emptyset\to \CS\sqcup \hS_k\}_c$, where $c$ runs over the finite set of finite projective $\cath$-colorings of $\Gamma'$.  
\end{proposition}
\begin{proof}
Let $\CS=(\Sigma, \{p_i\}_{i=1}^n, \coh, \mathcal{L})$ and $\tilde{\eta}$ be as above. By part (\ref{ite:spanning}) of Lemma \ref{lem:PropOfCGP}, $\state(\CS\sqcup  \hS_k)$ is spanned by $\{(\tilde{\eta}, T, \coh,0):  \emptyset\to \CS\}_T$, where $T\subset \tilde{\eta}$ runs over all $\coh$-compatible $\cath$-colored ribbon graphs.

Suppose that $g=0$. Let $(\tilde{\eta}, T, \coh,0) \in \state(\CS \sqcup \hS_k)$. Isotope all non-trivial parts of $T$ into a $3$-ball and put it into a coupon. The underlying graph is then $\Gamma^{\prime}$, as described before the proposition, and the coupon is colored by an element of 
\begin{equation}
\label{eq:gen0Coupon}
\Hom_\cath(\unit,  \bigotimes_{i=1}^n X^{\epsilon_i}_i\otimes V_{g_0} \otimes \sigma(k) \otimes V_{g_0}^*).
\end{equation}
The vectors from the statement of the proposition therefore span $\state(\CS \sqcup \hS_k)$. The vector space \eqref{eq:gen0Coupon} is isomorphic to
\[
\Hom_\cath(V_{g_0}, \bigotimes_{i=1}^n X^{\epsilon_i}_i\otimes V_{g_0} \otimes \sigma(k)  )
\simeq
\Hom_\cath(\sigma(-k),\bigotimes_{i=1}^n X^{\epsilon_i}_i\otimes V_{g_0} \otimes V_{g_0}^* ).
\]
Since $V_{g_0}$ is projective, property \ref{ite:finCat2} implies that the vector space on the left hand side of this isomorphism is finite dimensional for all $k \in \Zt$ while that on the right is non-zero for only finitely many $k\in \Zt$. This completes the proof when $g=0$.

Suppose now that $g \geq 1$. Define a cell decomposition of $\tilde{\eta}$ into $3$-balls $B_1,\dots,B_{2g-1}$ and $B_0=B\setminus \mathring {B^3}$, a $3$-ball minus the interior of a smaller $3$-ball, by cutting $\tilde{\eta}$ along $3g-1$ disks $D_{e_1},\dots, D_{e_{g+1}}, D_{f_0}, D_{f_1}, \dots,D_{f_{2g-3}}$ as follows. The disk $D_{f_0}$ is such that cutting along $D_{f_0}$ leaves $B_0$ which contains the points $\{p_i\}_{i=1}^n$ and $(V_{g_0},1),(\sigma(k),1)$, $ (V_{g_0},-1)$ on its boundary. The remaining disks are the obvious bounding disks in $\tilde{\eta}$ which intersect $\Gamma'$ exactly once at the edge by which they are labeled.

First, we show that $\state(\CS\sqcup  \hS_k)$ is spanned by $\{(\tilde{\eta}, \Gamma'_c, \coh,0)\}_{c}$, where $c$ runs over all $\coh$-compatible $\cath$-colorings of $\Gamma'$ in which the edges $e_1,\dots,e_{g+1}, f_0,\dots,f_{2g-3}$ are colored by projective indecomposables. Let $(\tilde{\eta}, T, \coh,0) : \emptyset\to \CS \sqcup \hS_k$ be any decorated cobordism. Let $D$ be one of $D_{e_1},\dots,D_{e_{g+1}}, D_{f_0}, \dots,D_{f_{2g-3}}$. Since $V_{g_0}$ is projective, we can entangle $T$ with itself (more precisely, with the edge colored by $V_{g_0}$) by an isotopy to assume that $T$ intersects $D$ in at least one edge colored by a projective. Next, we modify $T$ in a neighborhood of $D$ as follows. We can assume that $T$ intersects $D$ transversally at $\cath$-colored points forming a sequence $U=((U_1,\delta_1),\ldots,(U_t,\delta_t))$, where $U_i \in \cat_{\coh(m_{e_{U_i}})}$ and $\delta_i\in\{\pm1\}$ is determined by the orientation of the edge $e_{U_i}$. Since at least one $U_i$ is projective, the tensor product $F_\cath(U)$ is projective. By property \ref{ite:finCat3}, there exist projective indecomposables $\PP_s$ and morphisms $f_s:F_\cath(U)\to {\PP}_s$ and $h_s:{\PP}_s\to F_\cath(U)$ such that $\Id_{F_\cath(U)}=\sum_s h_s\circ f_s$.
This implies that a  tubular neighborhood 
of $D$ which consists of a cylinder containing $t$ strands colored by the $U_i$ is skein equivalent to
$$\sum_s\epsh{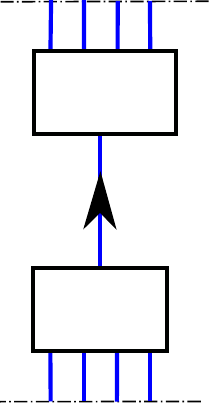}{12ex}\put(-18,17){\ms{h_s}}
  \put(-18,-14){\ms{f_s}}\put(-10,0){\ms{{\PP}_s}}\,.$$ 
By applying these transformations at each disk $D$ we obtain a skein equivalent element which is the sum of ribbon graphs $T'$ which meet each $D$ on a unique edge colored by a projective indecomposable. Moreover, up to skein equivalence, we can replace the contents of each $B_0, B_1,\dots,B_{2g-1}$ with a sum of the same $B_i$ with a single coupon. For $j\geq 1$, the ball $B_j$ has a single trivalent coupon while $B_0$ has a coupon with one incoming leg connected to $D_{f_0}$ and $n+3$ outgoing legs connected to points labeled by $\{X^{\epsilon_i}_i\}_{i=1}^n$, $V_{g_0}$, $\sigma(k)$ and $V^*_{g_0}$. It follows that $(\tilde{\eta}, T, \coh,0) \in \state(\CS\sqcup  \hS_k)$ is a linear combination of decorated cobordisms of the form $(\tilde{\eta}, \Gamma_c', \coh,0)$, where $c$ is an $\coh$-compatible $\cath$-coloring such that the colors of the edges $e_1,\dots,e_{g+1}, f_0,\dots,f_{2g-3}$ are projective indecomposables.

Next, we show that we can restrict to finitely many colorings of $\Gamma^{\prime}$ while still preserving the spanning property. Let $c$ be a $\cath$-coloring of $\Gamma'$ as in the previous paragraph. Consider the pair of edges $(e_2, f_1)$ of $\Gamma'$ with corresponding projective indecomposable colors $W_{e_2}$ and $W_{f_1}$. Property \ref{ite:finCat1} implies that $W_{e_2} \simeq P_j\otimes\sigma(k')$ for some $j\in J_{\coh(m_{e_2})}$ and $k'\in\Zt$. Consider an oriented $\sigma(k')$-colored curve close to the edges $e_2$ and $f_1$ in $\Gamma'$, the orientation chosen to be opposite to that of $e_2$. The union of this curve with $\Gamma'$ is $\sigma$-equivalent to $\Gamma'$ and, moreover, is skein equivalent to a coloring of $\Gamma'$ in which the edges $e_2$ and $f_1$ are colored by $W_{e_2}\otimes \sigma(k')^{-1} \simeq P_j$ and $W_{f_1}\otimes \sigma(k')$, respectively. Repeating this process iteratively for each of the pairs $(e_3,f_3), (e_4,f_5),\dots,(e_{g-1},f_{2g-5}), (e_g,f_{2g-3})$, it follows that the vector corresponding to $\Gamma'_c$ is proportional in $\state(\CS\sqcup  \hS_k)$ to a coloring of $\Gamma'$ where the color of each edge $e_i$, $i=2,\dots,g$, is in the finite set $\{P_j\}_{j\in J_{\coh(m_{e_i})}}$. Starting from this last coloring one can also put a $\sigma(k'')$-colored curve near the edge $e_1$ and all the edges $f_1, f_2,\dots,f_{2g-3}$ to find a skein equivalent element in which the color of $e_1$ is in $ \{P_j\}_{j\in J_{\coh(m_{e_1})}}$ and the colors of the edges $e_i$, $i=2,\dots,g$, are unchanged.  
 
Thus, $\state(\CS\sqcup  \hS_k)$ is spanned by $\{(\tilde{\eta}, \Gamma'_c, \coh,0)  \}_{c}$, where $c$ runs over all finite projective $\cath$-colorings of $\Gamma'$. It remains to show that such colorings are finite and non-zero for only finitely many $k\in \Zt$. Let $c$ be such a coloring. We showed above that in the cobordism $(\tilde{\eta}, \Gamma'_c, \coh,0)$ the ball containing the trivalent vertex with the legs $e_1$, $e_2$ and $f_1$ can be represented by a single coupon having one incoming leg labeled with $P_j$, for some $j\in J_{\omega(m_{e_1})}$, and two outgoing legs labeled with projective indecomposables $W_{f_1}$ and $P_{j'}$ for some $j'\in J_{\omega(m_{e_2})}$. Property \ref{ite:finCat1} implies that $W_{f_1} \simeq P_{j''}\otimes \sigma(k')$ for some $j''\in J_{\omega(m_{f_1})}$ and $k'\in \Zt$. Thus, colorings of this coupon are elements of
 $$
 \Hom_{\cath}(P_j,W_{f_1}\otimes P_{j'}) \simeq \Hom_{\cath}(P_j,P_{j''}\otimes \sigma(k') \otimes P_{j'}) \simeq  \Hom_{\cath}(\sigma(-k'), P_{j''}\otimes P_{j'}\otimes P_j^*),
 $$
which is finite dimensional for all $j,j',j'', k'$ and non-zero for only finitely many $j,j',j'', k'$ by property \ref{ite:finCat2}. Thus, there are only finitely many choices for the coloring $W_{f_1}$ of $f_1$. A similar argument now shows that the ball containing the trivalent vertex with legs $f_1$, $f_2$ and $e_2$ has only finitely many colorings of $f_2$ which correspond to non-zero cobordisms. Repeating this process for the edges $f_3,\dots,f_{2g-2}$ and $f_0$, we see that there are only finitely many colorings for each of these edges and all the corresponding colorings of coupons are finite dimensional. Finally, consider $B_0=B\setminus \mathring {B^3}$, which contains a single coupon with one incoming leg and $n+3$ outgoing legs. We have seen that $f_0$ must be colored by a projective indecomposable from the finite set 
$$
\{P_j \otimes \sigma(k_i) \mid j\in  J_{\omega(m_{f_1})} \text{ and } k_i\in \Zt \text{ for } i=1,\dots,t \}.
$$ 
Thus, the coupon in $B_0$ is colored by an element of
$$\Hom_\cath(P_j \otimes \sigma(k_i), \bigotimes_{i=1}^n X^{\epsilon_i}_i\otimes V_{g_0} \otimes \sigma(k) \otimes V_{g_0}^* )
\simeq \Hom_\cath( \sigma(-k), \sigma(-k_i) \otimes P_j^* \otimes \bigotimes_{i=1}^n X^{\epsilon_i}_i \otimes V_{g_0} \otimes  V_{g_0}^* )
$$
for some $ j\in J_{\omega(m_{f_1})}$ and $k_i\in \Zt$, $ i=1,\dots,t$, which is finite dimensional and non-zero for only finitely many $k \in \Zt$ by property \ref{ite:finCat2}. Since there are only finitely many choices of $j$ and $k_i$, it follows that there are only finitely many $k\in \Zt$ where the coloring $\Gamma'_c$ corresponds to a non-zero vector in $\state(\CS\sqcup \hS_k)$.  
\end{proof}

\begin{proposition}\label{P:VSVSS0}
Let $\cath$ be a modular $\Gr$-category relative to $(\Zt, \XX)$. For any $\CS \in \Cob_{\cat}$, the vector spaces $\state(\CS)$ and $\state(\CS\sqcup  \hS_0)$ are  isomorphic. Moreover, if $\cath$ is TQFT finite, then $\state(\CS)$ is spanned by decorated cobordisms $\{(\eta, \Gamma''_c, \coh,0):  \emptyset\to \CS \}_c$, where $\Gamma''$ is obtained from $\Gamma'$ by deleting the three edges attached to $\hS_k$
and $c$ runs over the finite set of all finite projective $\cath$-colorings of $\Gamma''$.  
\end{proposition}
\begin{proof}
In \cite[Definition 5.2]{derenzi2022} De Renzi defines decorated cobordisms 
$\widehat{D^3_0} : \emptyset \to\hS_{0}$ and $\widehat{\bar{D}^3_0}:  \hS_{0}\to \emptyset $. The cobordism $\widehat{D^3_0}$ is the closed $3$-ball with an unknotted strand connecting the points $(V_{g_0},1)$ and $ (V_{g_0},-1)$ on its boundary together with the unique compatible cohomology class. The cobordism $\widehat{\bar{D}^3_0}$ is defined similarly, by reversing all orientations involved. Up to a factor of $\D^{-1}\qd(V_{g_0}) \in \kk^{\times}$, the images of $\widehat{D^3_0}$ and $\widehat{\bar{D}^3_0}$ under $\state$ are mutually inverse \cite[Lemma 5.2]{derenzi2022}. It follows that $\Id_{\CS} \sqcup \widehat{D^3_0} : \CS  \to\CS \sqcup \hS_{0}$ and $\Id_{\CS} \sqcup\widehat{\bar{D}^3_0}: \CS \sqcup \hS_{0}\to \CS$ satisfy $\D\qd(V_{g_0})^{-1}\state(\Id_{\CS} \sqcup \widehat{D^3_0})\circ (\Id_{\CS} \sqcup\widehat{\bar{D}}^3_0 ) = \Id_{\CS \sqcup\hS_{0}}$ and thereby establish an isomorphism $\state(\CS) \simeq \state(\CS \sqcup \hS_0)$.

Assume now that $\cath$ is TQFT finite. Let $(\tilde{\eta}, \Gamma'_c, \coh,0):  \emptyset\to \CS\sqcup \hS_0$ be a decorated cobordism, where $c$ is a finite projective $\cath$-coloring of $\Gamma'$.  By Proposition \ref{prop:FDofTQFT}, such cobordisms span $\state(\CS\sqcup  \hS_0)$. The cobordism $(\Id_{\CS} \sqcup\widehat{\bar{D}^3_0})\circ (\tilde{\eta}, \Gamma'_c, \coh,0)$ is skein equivalent to a cobordism of the form $(\eta, \Gamma''_c, \coh,0):  \emptyset\to \CS$, where $c$ is a finite projective $\cath$-coloring of $\Gamma''$. Such colorings therefore span $\state(\CS)$. An argument similar to that from the proof of Proposition \ref{prop:FDofTQFT} shows that the set of finite projective $\cath$-colorings of $\Gamma''$ is finite.  \end{proof}

\section{Relative modular categories from $\Uq$}
\label{sec:relModgloo}

\subsection{Conventions for superalgebra}

Let $\Ztwo = \{\p 0, \p 1\}$ be the additive group of order two. A \emph{super vector space} is a $\Ztwo$-graded vector space $V = V_{\p 0} \oplus V_{\p 1}$. The degree of a homogeneous element $v \in V$ is denoted $\p v \in \Ztwo$. A \emph{morphism of super vector spaces of degree $\p d \in \Ztwo$} is a linear map $f: V \rightarrow W$ which satisfies $\overline{f(v)} = \p v + \p d$ for each homogeneous $v \in V$. A (left) \emph{module} over a superalgebra $A$ is a super vector space $M$ together with a superalgebra homomorphism $A \rightarrow \End_{\kk}(M)$ of degree $\p 0$. We refer the reader to \cite[I-Supersymmetry]{deligne1999} for a detailed discussion of superalgebra.

\subsection{The unrolled quantum group of $\gloo$}
\label{sec:unrolledgloo}

Motivated by standard quantizations of $\gloo$ \cite{kulish1989,khoroshkin1991,reshetikhin1992,Viro} and the definition of the unrolled quantum group $U^H_q(\mathfrak{sl}(2))$ \cite{GPT,CGP2}, we define an unrolled version of the universal enveloping algebra of the complex Lie superalgebra $\gloo$.

Fix $\hbar \in \C \setminus \pi \sqrt{-1} \Z$ and set $q=e^{\hbar} \in \C \setminus \{0, \pm 1\}$. For $z \in \C$, define $q^z = e^{\hbar z}$. Let $\Uq$ be the unital associative superalgebra over $\C$ with generators $E$, $K$, $K^{-1}$, $G$ of degree $\p 0$ and $X$, $Y$ of degree $\p 1$ subject to the relations
\[
KK^{-1}=K^{-1}K=1,
\]
\[
[K^s,E]=[K^s,G]=[K^s,X]=[K^s,Y]=0, \qquad s \in \{\pm 1\}
\]
\[
[E,X]=0, \qquad [E,Y]=0,\qquad [E,G]=0,
\]
\[
[G,X]=X,\qquad [G,Y]=-Y,
\]
\[
XY+YX=\frac{K-K^{-1}}{q-q^{-1}},
\]
\begin{equation*}
\label{E:RelUhgloo}
X^2=Y^2=0.
\end{equation*}
Note that both $E$ and $K$ are central in $\Uq$. Define a counit $\epsilon$, coproduct $\Delta$ and antipode $S$ on $\Uq$ by
\[
\epsilon(E)=\epsilon(G)=\epsilon(X)=\epsilon(Y)=0,
\qquad
\epsilon(K)=1,
\]
\[
\Delta(E)=E\otimes 1 + 1\otimes E,\qquad \Delta(G)= G\otimes 1 + 1\otimes G, \qquad \Delta K = K \otimes K,
\]
\[
\Delta(X)=X\otimes K^{-1} + 1\otimes X,\qquad \Delta(Y)=Y\otimes 1 + K\otimes Y,
\]
\begin{equation*}
\label{E:HopfUh}
S(E)=-E, \qquad S(G)=-G,\qquad S(K)=K^{-1},\qquad S(X)=-XK,\qquad S(Y)=-YK^{-1}
\end{equation*}
and requiring $\epsilon$ and $\Delta$ to be superalgebra homomorphisms and $S$ to be a superalgebra anti-homomorphism. This gives $\Uq$ the structure of a Hopf superalgebra. The subalgebra $U_q(\gloo) \subset \Uq$ generated by $K^{\pm 1}$, $G$, $X$, $Y$ is a standard quantization of $\gloo$.

\subsection{Weight modules}
\label{sec:weightMod}

In this section we describe the category of weight modules over $\Uq$. See \cite[\S 11]{Viro}, \cite[\S 3]{sartori2015} for similar descriptions of module categories for standard quantizations of $\gloo$.

\subsubsection{The category of weight modules}
Let $V$ be a $\Uq$-module. A homogeneous vector $v \in V$ is called a \emph{weight vector} if it is a simultaneous eigenvector of $E$ and $G$, say, $Ev = \lambda_E v$ and $Gv=\lambda_G v$, in which case $\lambda=(\lambda_E,\lambda_G) \in \C^2$ is the \emph{weight} of $v$. A finite dimensional $\Uq$-module $V$ is called a \emph{weight module} if it is a direct sum of weight spaces and $K v=q^{\lambda_E}v$ for any weight vector $v$ of weight $\lambda$. Since the action of $K$ on a weight module can be deduced from that of $E$, we often omit it. Let $\catq$ be the abelian category of weight $\Uq$-modules and their $\Uq$-linear maps of degree $\p 0$. Unless mentioned otherwise, all $\Uq$-modules considered in this paper are assumed to be weight modules.

Given $V \in \catq$, let $V^* \in \catq$ be the $\C$-linear dual of $V$ with $\Uq$-module structure
\[
(x \cdot f)(v) = (-1)^{\p f \p x} f(S(x)v),
\qquad
v \in V, \; f \in V^*, \; x \in \Uq.
\]
Let $\{v_i\}_i$ be a homogeneous basis of $V$ with dual basis $\{v_i^*\}_i$. Define
\begin{equation}\label{E:tcoev}
\tev_V(f \otimes v) = f(v),
\qquad
\tcoev(1)=\sum_i v_i \otimes v_i^*
\end{equation}
and
\begin{equation}\label{E:coev}
\ev_V(v \otimes f) = (-1)^{\p f \p v}f(K v),
\qquad
\coev_V(1) =\sum_i (-1)^{\p v_i}v_i^* \otimes K^{-1} v_i.
\end{equation}

\begin{lemma}\label{lem:catqPivot}
The maps \eqref{E:tcoev} and \eqref{E:coev} define a pivotal structure on $\catq$.
\end{lemma}

\begin{proof}
That the maps \eqref{E:tcoev} and \eqref{E:coev} are $\Uq$-linear and satisfy the snake relations is a straightforward calculation. For pivotality, we need to verify that for each morphism $f : V \rightarrow W$ in $\catq$, its left and right duals $W^* \rightarrow V^*$ are equal and that the canonical isomorphisms $(V \otimes W)^* \rightarrow W^* \otimes V^*$ constructed using left and right dualities are equal. See \cite[\S 2.10]{etingof2015}. In both cases, the morphisms constructed from left duality are the standard super vector space such morphisms. That right duality gives the same morphisms follows from the fact that $K$ appears in $\ev$ while $K^{-1}$ appears in $\coev$.
\end{proof}

It follows that the quantum dimension of $V \in \catq$ is $\qdim V = \sum_i (-1)^{\p v_i} v_i^*(K v_i)$.

Let $V,W \in \catq$. Define $\Upsilon_{V,W} \in \End_{\C}(V \otimes W)$ by
\[
\Upsilon_{V,W} (v \otimes w)=q^{-\lambda_E \mu_G-\lambda_G \mu_E}v\otimes w,
\]
where $v \in V$ and $w \in W$ are of weight $\lambda$ and $\mu$, respectively. Define $\tilde{R}_{V,W} \in \End_{\C}(V \otimes W)$ to be left multiplication by $1+(q-q^{-1})(X\otimes Y)(K \otimes K^{-1})$ and $c_{V,W} \in \Hom_{\C}(V \otimes W, W \otimes V)$ by
\[
c_{V,W}=\tau_{V,W} \circ \tilde{R}_{V,W} \circ \Upsilon_{V,W},
\]
where $\tau_{V,W}: V \otimes W \rightarrow W \otimes V$, $v \otimes w \mapsto (-1)^{\p v \p w} w \otimes v$, is the standard symmetric braiding on the category of super vector spaces.

\begin{proposition}
\label{prop:braiding}
The maps $\{c_{V,W} : V \otimes W \rightarrow W \otimes V\}_{V,W \in \catq}$ define a braiding on $\catq$.
\end{proposition}

\begin{proof}
That $c_{V,W}$ is a $\C$-linear isomorphism follows from the corresponding statements for $\tau_{V,W}$ and $\Upsilon_{V,W}$ and the observation that the inverse of $\tilde{R}_{V,W}$ is left multiplication by $1-(q-q^{-1})(X\otimes Y)(K \otimes K^{-1})$. To verify that $c_{V,W}$ is $\Uq$-linear, we verify that it commutes with the action of the generators of $\Uq$. We give the calculation only for the generator $X$. Fix $v \in V$ of weight $\lambda$ and $w\in W$ of weight $\mu$. We compute
\begin{multline*}
c_{V,W} (X \cdot v \otimes w)
=
(-1)^{(\p v + \p 1) \p w}q^{-\lambda_E \mu_G - (\lambda_G+2) \mu_E} w \otimes X v + (-1)^{\p v \p w} q^{-\lambda_E (\mu_G+1) - \lambda_G \mu_E}Xw \otimes v \\
+(-1)^{\p v (\p w+ \p 1)}(q-q^{-1})q^{-\lambda_E \mu_G - (\lambda_G +1) \mu_E} YX w \otimes Xv
\end{multline*}
and
\[
c_{V,W}(v \otimes w)
=
(-1)^{\p v \p w} q^{-\lambda_E \mu_G - \lambda_G \mu_E} w \otimes v + (-1)^{\p v \p w + \p w + \p 1}q^{-\lambda_E \mu_G - \lambda_G \mu_E}(q-q^{-1})q^{\lambda_E - \mu_E} Yw \otimes Xv
\]
so that
\begin{eqnarray*}
X \cdot c_{V,W}(v \otimes w)
&=&
(-1)^{\p v \p w} q^{-\lambda_E \mu_G - \lambda_G \mu_E-\lambda_E} Xw \otimes v + (-1)^{\p v \p w + \p w} q^{-\lambda_E \mu_G - \lambda_G \mu_E} w \otimes Xv +\\
&& (-1)^{\p v \p w + \p w + \p 1}q^{-\lambda_E \mu_G - \lambda_G \mu_E}(q-q^{-1})q^{\lambda_E - \mu_E - \lambda_E} XYw \otimes Xv \\
&=&
(-1)^{\p v \p w} q^{-\lambda_E(\mu_G +1) - \lambda_G \mu_E} Xw \otimes v + (-1)^{\p v \p w + \p w} q^{-\lambda_E \mu_G - \lambda_G \mu_E} w \otimes Xv +\\
&& (-1)^{\p v \p w + \p w + \p 1}q^{-\lambda_E \mu_G - (\lambda_G +1) \mu_E}(q-q^{-1}) (-YX + \frac{q^{\mu_E}-q^{-\mu_E}}{q-q^{-1}})w \otimes Xv \\
&=&
(-1)^{\p v \p w} q^{-\lambda_E(\mu_G +1) - \lambda_G \mu_E} Xw \otimes v + (-1)^{\p v \p w + \p w} q^{-\lambda_E \mu_G - \lambda_G \mu_E} w \otimes Xv +\\
&& (-1)^{\p v \p w + \p w}q^{-\lambda_E \mu_G - (\lambda_G +1) \mu_E}(q-q^{-1}) YXw \otimes Xv +\\
&& (-1)^{\p v \p w + \p w + \p 1}q^{-\lambda_E \mu_G - (\lambda_G +1) \mu_E}(q^{\mu_E}-q^{-\mu_E})w \otimes Xv,
\end{eqnarray*}
which is equal to $c_{V,W} (X \cdot v \otimes w)$.

Next, consider the hexagon identities, in which we suppress all associators. Let $U,V,W \in \catq$. We verify that $c_{U, V \otimes W} = (\Id_V \otimes c_{U,W}) \circ (c_{U,V} \otimes \Id_W)$; verification of the second hexagon identity is analogous. Fix vectors $u\in U$, $v\in V$ and $w\in W$ of weights $\lambda$, $\mu$ and $\nu$, respectively. We compute
\begin{eqnarray*}
c_{U,V \otimes W}(u \otimes v \otimes w)
&=&
(-1)^{(\p v + \p w) \p u} q^{-\lambda_E(\mu_G + \nu_G) - \lambda_G(\mu_E+\nu_E)} v \otimes w \otimes u  +(-1)^{\p u + (\p u+ \p 1)(\p v + \p 1 + \p w)} \cdot \\
&&
q^{-\lambda_E(\mu_G+\nu_G)-\lambda_G(\mu_E+\nu_E)} (q-q^{-1})q^{\lambda_E-\mu_E-\nu_E} Y v \otimes w \otimes Xu \\
&&
+(-1)^{\p u + \p v + (\p u + \p 1)(\p v+ \p w +1)} q^{-\lambda_E(\mu_G+\nu_G)-\lambda_G(\mu_E+\nu_E)} (q-q^{-1}) \cdot \\
&&q^{\lambda_E-\nu_E} v \otimes Yw \otimes Xu.
\end{eqnarray*}
The image of $u \otimes v \otimes w$ under $(\Id_V \otimes c_{U,W}) \circ (c_{U,V} \otimes \Id_W)$ is
\begin{eqnarray*}
u \otimes v \otimes w
&\xmapsto{c_{U,V} \otimes \Id}&
(-1)^{\p u + (\p u + \p 1)(\p v + \p 1)} q^{-\lambda_E \mu_G - \lambda_G \mu_E} q^{\lambda_E - \mu_E}(q-q^{-1}) Yv \otimes Xu \otimes w \\
&\xmapsto{\Id \otimes c_{U,W}}&
(-1)^{\p u \p v} (-1)^{\p u \p w} q^{-\lambda_E \mu_G - \lambda_G\mu_E} q^{-\lambda_E \nu_G - \lambda_G \nu_E} v \otimes w \otimes u \\
&&
+  (-1)^{\p u \p v}(-1)^{(\p u + \p 1)(\p w + \p 1)} q^{-\lambda_E \mu_G - \lambda_G\mu_E} q^{-\lambda_E \nu_G - \lambda_G \nu_E} q^{\lambda_E - \nu_E}(q-q^{-1}) \cdot \\
&&
q^{-\lambda_E \nu_G - (\lambda_G+1)\nu_E} (-1)^{\p u + \p v} (-1)^{\p v} v \otimes Yw \otimes Xu \\
&&+
(-1)^{\p u} (-1)^{(\p u + \p 1)(\p v + \p 1)} (-1)^{\p v + \p 1)\p w} q^{-\lambda_E \mu_G - \lambda_G \mu_E} q^{-\lambda_E-\mu_E} (q-q^{-1}) \cdot \\
&&
q^{-\lambda_E \nu_G - (\lambda_G+1)\nu_E} Yv \otimes w \otimes X u.
\end{eqnarray*}
Direct comparison shows that this expression agrees with $c_{U,V \otimes W}(u \otimes v \otimes w)$.
\end{proof}

\subsubsection{Simple modules}
\label{sec:UqSimples}

Given $z \in \C$, put $[z]_q = \frac{q^z - q^{-z}}{q-q^{-1}}$. Note that $[z]_q=0$ if and only if $z \in \frac{\pi \sqrt{-1}}{\hbar} \Z$.

Given $(n,b) \in \Z \times \C$ and $\p p \in \Ztwo$, let $\ve(\frac{n \pi \sqrt{-1}}{\hbar},b)_{\p p}$ be the one dimensional $\Uq$-module with basis vector $v$ of degree $\overline{p}$ and module structure
\[
Ev=\frac{n \pi \sqrt{-1}}{\hbar} v, \qquad Gv=bv, \qquad Xv=0, \qquad Yv=0.
\]

Following \cite[\S 11.4]{Viro}, the \emph{quantum Kac module} of weight $(\alpha,a) \in \C^2$ and degree $\p p \in \Ztwo$ is the $\Uq$-module $V(\alpha,a)_{\p p}$ with basis vectors $v$ of degree $\p p$ and $v^{\prime}$ of degree $\p p + \p 1$ and
$$Ev=\alpha v, \qquad Gv=a v, \qquad Xv=0,  \qquad Yv=v^{\prime},$$
$$Ev^{\prime}=\alpha v^{\prime},  \qquad Gv^{\prime}=(a-1) v^{\prime},  \qquad Yv^{\prime}=0,  \qquad Xv^{\prime}=[\alpha]_qv.$$
The structure of $V(\alpha,a)_{\p p}$ is illustrated by the diagram
\[
V(\alpha,a)_{\p p} :
\qquad
\begin{tikzpicture}[baseline= (a).base]
\node[scale=1.0] (a) at (0,0){
\begin{tikzcd}[column sep=5em]
v^{\prime} \arrow[r, bend left, "X={[\alpha]}_q"] \arrow[out=120,in=60,loop,,looseness=3,"a-1"] \arrow[out=-120,in=-60,loop,looseness=3,"\alpha" below] & v \arrow[out=120,in=60,loop,looseness=3,"a"] \arrow[l, bend left, "Y"] \arrow[out=-120,in=-60,loop,looseness=3,"\alpha" below]
\end{tikzcd}.
};
\end{tikzpicture}
\]
In this diagram, and those which follow, the actions of $G$ and $E$ are depicted as loops above and below the weight vectors, respectively, and an arrow labeled by a generator indicates that it acts by the identity in the given basis. The module $V(\alpha,a)_{\p p}$ is simple if and only if $\alpha \in \C \setminus \frac{\pi \sqrt{-1}}{\hbar} \Z$, in which case the dual vector $v^{\prime *} \in V(\alpha,a)_{\p p}^*$ is highest weight and
\begin{equation}
\label{eq:dualIrred}
V(\alpha,a)_{\p p}^*
\simeq
V(-\alpha, -(a-1))_{\p p +\p1}.
\end{equation}
If instead $\alpha = \frac{n \pi \sqrt{-1}}{\hbar}$, $n \in \Z$, then $V(\alpha,a)_{\p p}$ is indecomposable and fits in the exact sequence
\begin{equation}
\label{eq:filtKac}
0 \rightarrow \ve(\frac{n \pi \sqrt{-1}}{\hbar},a -1)_{\p p + \p 1} \rightarrow V(\frac{n \pi \sqrt{-1}}{\hbar},a)_{\p p} \rightarrow \ve(\frac{n \pi \sqrt{-1}}{\hbar},a)_{\p p} \rightarrow 0.
\end{equation}

\begin{proposition}
\label{prop:irrepsgloo}
Every simple object of $\catq$ is isomorphic to exactly one of the modules on the list
\begin{itemize}
\item $\ve(\frac{n \pi \sqrt{-1}}{\hbar},b)_{\p p}$ with $(n,b) \in \Z \times \C$ and $\p p \in \Ztwo$,
\item $V(\alpha,a)_{\p p}$ with $(\alpha,a) \in (\C \setminus \frac{\pi \sqrt{-1}}{\hbar} \Z) \times \C$ and $\p p \in \Ztwo$.
\end{itemize}
\end{proposition}

\begin{proof}
Let $V \in \catq$ be simple. Since $X^2=0$, there exists a highest weight vector $v \in V$, say of weight $\lambda (\lambda_E,\lambda_G)$, and $V=\Uq \cdot v$. If $Yv=0$, then $(K-K^{-1})v=0$, which implies $\lambda_E \in \frac{\pi \sqrt{-1}}{\hbar} \Z$ and $V \simeq \ve(\lambda_E,\lambda_G)_{\p v}$. If instead $Yv \neq 0$, then $V \simeq V(\lambda_E,\lambda_G)_{\p v}$ and, by simplicity and the discussion above the proposition, $\lambda_E \in \C \setminus \frac{\pi \sqrt{-1}}{\hbar} \Z$.
\end{proof}

\begin{lemma}
\label{lem:locFin}
The category $\catq$ is locally finite.
\end{lemma}

\begin{proof}
That morphism spaces in $\catq$ are finite dimensional is clear. To see that $\catq$ has finite length, let $V \in \catq$ be non-zero. There exists a highest weight vector $v \in V$ which generates a submodule $\langle v \rangle$ of the form $\ve(\lambda_E,\lambda_G)_{\p v}$ with $\lambda_E \in \frac{\pi \sqrt{-1}}{\hbar} \Z$ or $V(\lambda_E,\lambda_G)_{\p v}$ with $\lambda_E \in \C$. Using the filtration \eqref{eq:filtKac} of $V(\lambda_E,\lambda_G)_{\p v}$ if $\lambda_E \in \frac{\pi \sqrt{-1}}{\hbar} \Z$ and simplicity of $\langle v \rangle$ otherwise, the problem of finding a Jordan--H\"{o}lder filtration of $V$ reduces to that of $V \slash \langle v \rangle$, whose dimension strictly less than that of $V$. Iterating this argument completes the proof.
\end{proof}

\begin{prop}
\label{prop:genericProjective}
If $\alpha \in \C \setminus \frac{\pi \sqrt{-1}}{\hbar} \Z$, then $V(\alpha,a)_{\p p} \in \catq$ is projective and injective.
\end{prop}

\begin{proof}
It follows easily from Lemma \ref{lem:locFin} that $\catq$ is multitensor, whence projective objects are injective \cite[Proposition 6.1.3]{etingof2015}. It therefore suffices to prove that $V(\alpha,a)_{\p p}$ is projective. There is a $\Uq$-module isomorphism
\[
V(\alpha,a)_{\p p} \simeq \Uq \otimes_{\UqB} \C(\alpha,a)_{\p p},
\]
where $\UqB$ is the subalgebra of $\Uq$ generated by $E$, $G$, $K^{\pm 1}$ and $X$ and $\C(\alpha,a)_{\p p}$ is the one dimensional weight $\UqB$-module concentrated in degree $\p p$ of weight $(\alpha,a)$ and on which $X$ acts by zero. Let $f: V \rightarrow W$ be an epimorphism in $\catq$ and $\phi: V(\alpha,a)_{\p p} \rightarrow W$ a non-zero morphism. Since
\[
\Hom_{\catq}(V(\alpha,a)_{\p p}, W)
\simeq
\Hom_{\catq_{\geq 0}}(\C(\alpha,a)_{\p p},W_{\vert \UqB}),
\]
where $\catq_{\geq 0 }$ denotes the category of weight $\UqB$-modules, the map $\phi$ is determined by a morphism $\C(\alpha,a)_{\p p} \rightarrow W_{\vert \UqB}$ in $\catq_{\geq 0}$, that is, a highest weight vector $w \in W$ of weight $(\alpha,a)$ and degree $\p p$. By assumption, $w$ has a weight vector preimage in $V_{\p p}$, say $v$. Since $\alpha \in \C \setminus \frac{\pi \sqrt{-1}}{\hbar} \Z$, we may form $v^{\prime} =v-[\alpha]_q^{-1} FEv \in V$, which is highest weight of weight $(\alpha,a)$ and degree $\p p$. The assignment $1 \mapsto v^{\prime}$ defines a morphism $\C(\alpha,a)_{\p p} \rightarrow V_{\vert \UqB}$ in $\catq_{\geq 0}$ which in turn determines a morphism $\tilde{\phi}:V(\alpha,a)_{\p p} \rightarrow V$ in $\catq$, by the argument above. Since $f(v^{\prime})=w$, the map $\tilde{\phi}$ satisfies $f \circ \tilde{\phi} = \phi$. Hence, $V(\alpha,a)_{\p p}$ is projective.
\end{proof}

For later use, define the \emph{quantum anti-Kac module} of lowest weight $(\alpha,a)\in \C^2$ and degree $\p p \in \Ztwo$ to be the $\Uq$-module $\overline{V}(\alpha,a)_{\p p}$ with basis vectors $v^{\prime}$ of degree $\p p$ and $v$ of degree $\p p + \p 1$ and module structure
\[
\overline{V}(\alpha,a)_{\p p} :
\qquad
\begin{tikzpicture}[baseline= (a).base]
\node[scale=1.0] (a) at (0,0){
\begin{tikzcd}[column sep=5em]
v^{\prime} \arrow[r, bend left, "X"] \arrow[out=120,in=60,loop,,looseness=3,"a"] \arrow[out=-120,in=-60,loop,looseness=3,"\alpha" below] & v \arrow[out=120,in=60,loop,looseness=3,"a+1"] \arrow[l, bend left, "Y={[\alpha]}_q"] \arrow[out=-120,in=-60,loop,looseness=3,"\alpha" below]
\end{tikzcd}
};
\end{tikzpicture}.
\]
The module $\overline{V}(\alpha,a)_{\p p}$ is simple if and only if $\alpha \notin \frac{\pi \sqrt{-1}}{\hbar} \Z$, in which case $\overline{V}(\alpha,a)_{\p p} \simeq V(\alpha,a+1)_{\p p + \p 1}$. If $\alpha= \frac{n \pi \sqrt{-1}}{\hbar}$, $n \in \Z$, then $\overline{V}(\frac{n \pi \sqrt{-1}}{\hbar},a)_{\p p}$ is indecomposable and fits in the exact sequence
\[
0 \rightarrow \ve(\frac{n \pi \sqrt{-1}}{\hbar},a+1)_{\p p + \p 1} \rightarrow \overline{V}(\frac{n \pi \sqrt{-1}}{\hbar},a)_{\p p} \rightarrow \ve(\frac{n \pi \sqrt{-1}}{\hbar},a)_{\p p} \rightarrow 0.
\]

Finally, we record that
\begin{equation}\label{eq:qdimVE}
\qdim \ve(\frac{n \pi \sqrt{-1}}{\hbar},b)_{\p p} = (-1)^{\p p + n}
\end{equation}
and $\qdim V(\alpha,a)_{\p p} = 0$.

\subsubsection{Projective indecomposable modules}
\label{sec:projIndGeneric}

Given $(n,b) \in \Z \times \C$ and $\p p \in \Ztwo$, let $P(\frac{n \pi \sqrt{-1}}{\hbar},b)_{\p p}$ be the $\Uq$-module with basis $\{v,v^{\prime}, v_+, v_-\}$, where $v,v^{\prime}$ are of degree $\p p$ and $v_+, v_-$ are of degree $\p p + \p 1$, and module structure
\[
P(\frac{n \pi \sqrt{-1}}{\hbar},b)_{\p p} :
\qquad
\begin{tikzpicture}[baseline= (a).base]
\node[scale=1.0] (a) at (0,0){
\begin{tikzcd}
{} & v_+ \arrow{dr}[above right]{Y=-1} \arrow[out=120,in=60,loop,looseness=3,"b+1"] & {} \\
v^{\prime} \arrow{ur}[above left]{X} \arrow{dr}[below left]{Y} \arrow[out=120,in=60,loop,looseness=3,"b"] & {} & v \arrow[out=120,in=60,loop,looseness=3,"b"] \\
{} & v_- \arrow{ur}[below right]{X} \arrow[out=120,in=60,loop,looseness=3,"b-1" above]& {}
\end{tikzcd}.
};
\end{tikzpicture}
\]
Here $E$ acts everywhere by $\frac{n \pi \sqrt{-1}}{\hbar}$ and is omitted from the diagram. Setting $M_1 = \spvs\{v\}$ and $M_2 = \spvs\{v,v_+,v_-\}$ defines a filtration $0 \subset M_1 \subset M_2 \subset M_3 = P(\frac{n \pi \sqrt{-1}}{\hbar},b)_{\p p}$ with
\[
M_1 \simeq M_3 \slash M_2 \simeq \ve(\frac{n \pi \sqrt{-1}}{\hbar},b)_{\p p}
\]
and 
\[
M_2 \slash M_1 \simeq \ve(\frac{n \pi \sqrt{-1}}{\hbar},b+1)_{\p p + \p 1} \oplus \ve(\frac{n \pi \sqrt{-1}}{\hbar},b-1)_{\p p + \p 1}.
\]

In this paragraph, fix $n \in \Z$ and write $V(b)_{\p p}$ for $V(\frac{n \pi \sqrt{-1}}{\hbar},b)_{\p p}$, and similarly for $P(b)_{\p p}$ and $\ve(b)_{\p p}$. If a basis vector is sent to zero under a map, then we omit it from the notation. There are morphisms
\begin{equation}\label{eq:projCov}
P(b)_{\p p} \twoheadrightarrow \ve(b)_{\p p},
\qquad
v^{\prime} \mapsto v
\end{equation}
and $\ve(b)_{\p p} \hookrightarrow P(b)_{\p p}$, $v \mapsto v$. The composition
\[
x_{\frac{n \pi \sqrt{-1}}{\hbar},b,\p p}: P(b)_{\p p} \twoheadrightarrow \ve(b)_{\p p}\hookrightarrow P(b)_{\p p},
\qquad
v^{\prime} \mapsto v
\]
is nilpotent of order two. There are canonical projections
\[
P(b)_{\p p} \twoheadrightarrow V(b)_{\p p},
\qquad
v^{\prime} \mapsto v,
\;\;
v_- \mapsto v^{\prime}
\]
and
\[
P(b)_{\p p} \twoheadrightarrow \overline{V}(b)_{\p p},
\qquad
v^{\prime} \mapsto v^{\prime},
\;\;
v_+ \mapsto v.
\]
There are canonical inclusions
\[
V(b+1)_{\p p + \p 1} \hookrightarrow P(b)_{\p p}
\qquad
v^{\prime} \mapsto v, \;\; v \mapsto - v_+
\]
and
\[
\overline{V}(b-1)_{\p p + \p 1 } \hookrightarrow P(b)_{\p p},
\qquad
v^{\prime} \mapsto v_-, \;\; v \mapsto v.
\]
Consider the compositions
\[
a_{\frac{n \pi \sqrt{-1}}{\hbar},b,\p p}^-: 
P(b)_{\p p} \twoheadrightarrow V(b)_{\p p} \hookrightarrow P(b-1)_{\p p+ \p1},
\qquad
v^{\prime} \mapsto - v_+, \;\; v_- \mapsto v
\]
and
\[
a_{\frac{n \pi \sqrt{-1}}{\hbar},b,\p p}^+: 
P(b)_{\p p} \twoheadrightarrow \overline{V}(b)_{\p p} \hookrightarrow P(b+1)_{\p p+1},
\qquad
v^{\prime} \mapsto v_-, \;\; v_+ \mapsto v.
\]

\begin{proposition}\label{prop:HomProj}
If $n_1 \neq n_2$, then
\[
\Hom_{\catq}(P(\frac{n_1 \pi \sqrt{-1}}{\hbar},b_1)_{\p p_1},P(\frac{n_2 \pi \sqrt{-1}}{\hbar},b_2)_{\p p_2}) =0.
\]
Otherwise, there are isomorphisms
\begin{multline*}
\Hom_{\catq}(P(\frac{n \pi \sqrt{-1}}{\hbar},b_1)_{\p p_1},P(\frac{n \pi \sqrt{-1}}{\hbar},b_2)_{\p p_2})
\simeq \\
\begin{cases}
\C [x_{\frac{n \pi \sqrt{-1}}{\hbar},b_1,\p p_1}] \slash \langle x_{\frac{n \pi \sqrt{-1}}{\hbar},b_1,\p p_1}^2 \rangle & \mbox{if } (b_2, \p p_2) = (b_1, \p p_1), \\
\C \cdot a_{\frac{n \pi \sqrt{-1}}{\hbar},b_1, \p p_1}^{\pm} & \mbox{if } (b_2, \p p_2) = (b_1 \pm 1, \p p_1 + \p 1), \\
0 & \mbox{otherwise}.
\end{cases}
\end{multline*}
Moreover, the relations $a_{\frac{n \pi \sqrt{-1}}{\hbar},b \mp 1,\p p+1}^{\pm} \circ a_{\frac{n \pi \sqrt{-1}}{\hbar},b,\p p}^{\mp} = \mp x_{\frac{n \pi \sqrt{-1}}{\hbar},b,\p p}$ hold.
\end{proposition}

\begin{proof}
That $x$ and $a^{\pm}$ are the only morphisms follows by analyzing the filtration $M_{\bullet}$ of $P(\frac{n \pi \sqrt{-1}}{\hbar},b)_{\p p}$ given above. The remaining statements are clear.
\end{proof}

\subsubsection{Tensor products}
\label{sec:tenProd}

By comparing weights, it is easy to verify the isomorphisms
\begin{equation}
\label{eq:tensorOneDim}
\ve(\frac{n_1 \pi \sqrt{-1}}{\hbar},b_1)_{\p p_1} \otimes \ve(\frac{n_2 \pi \sqrt{-1}}{\hbar},b_2)_{\p p_2}
\simeq
\ve(\frac{(n_1+n_2) \pi \sqrt{-1}}{\hbar},b_1+b_2)_{\p p_1 + \p p_2},
\end{equation}
\begin{equation}
\label{eq:tensorOneWithTwoDim}
V(\alpha,a)_{\p p} \otimes \ve(\frac{n \pi \sqrt{-1}}{\hbar},b)_{\p q}
\simeq
V(\alpha + \frac{n \pi \sqrt{-1}}{\hbar},a+b)_{\p p + \p q}
\end{equation}
and
\begin{equation*}
\label{eq:tensorOneDimWithProj}
P(\frac{n \pi \sqrt{-1}}{\hbar},b)_{\p p} \simeq P(0,b)_{\p 0} \otimes \ve(\frac{n \pi \sqrt{-1}}{\hbar},0)_{\p p}.
\end{equation*}

Let $\alpha_1,\alpha_2 \in \C \setminus \frac{\pi \sqrt{-1}}{\hbar} \Z$. If $\alpha_1 + \alpha_2 \in \C \setminus \frac{\pi \sqrt{-1}}{\hbar} \Z$, then an argument using weight vectors and the injectivity of Proposition \ref{prop:genericProjective} gives
\begin{equation}\label{E:tensorSimpGen}
V(\alpha_1,a_1)_{\p p_1}\otimes V(\alpha_2,a_2)_{\p p_2}\simeq V(\alpha_1+\alpha_2,a_1+a_2)_{\p p_1 +\p p_2}\oplus V(\alpha_1+\alpha_2,a_1+a_2-1)_{\p p_1+\p p_2 +\p 1}.
\end{equation}
See \cite[\S 11.6.B]{Viro} for an explicit isomorphism. If instead $\alpha_1 + \alpha_2 = \frac{n \pi \sqrt{-1}}{\hbar}$, $n \in \Z$, then there is an isomorphism
\begin{equation}\label{E:tensorSimpNongen}
I: V(\alpha_1,a_1)_{\p p_1}\otimes V(\alpha_2,a_2)_{\p p_2}
\xrightarrow[]{\sim}
P(\alpha_1+\alpha_2,a_1+a_2-1)_{\p p_1 + \p p_2 + \p 1}.
\end{equation}
See \cite[\S 3]{sartori2015}. Explicitly, if $v_i \in V(\alpha_i,a_i)_{\p p_i}$ is highest weight and $v_i^{\prime} = Y v_i$, then
$I$ is defined by
\[
I(X(v^{\prime}_1 \otimes v^{\prime}_2)) = v^{\prime},
\qquad
I(v_1 \otimes v_2) =v_+,
\]
\[
I(v_1^{\prime} \otimes v_2^{\prime}) = v_-,
\qquad
I(v_1^{\prime} \otimes v_2^{\prime}) = v.
\]

\begin{lemma}
\label{lem:projIndec}
Let $(n,b) \in \Z \times \C$ and $\p p \in \Ztwo$. The module $P(\frac{n \pi \sqrt{-1}}{\hbar},b)_{\p p}$ is projective indecomposable and the morphism \eqref{eq:projCov} is the projective cover of $\ve(\frac{n \pi \sqrt{-1}}{\hbar},b)_{\p p}$.
\end{lemma}

\begin{proof}
By Proposition \ref{prop:HomProj}, the only idempotent endomorphisms of $P(\frac{n \pi \sqrt{-1}}{\hbar},b)_{\p p}$ are multiples of the identity. Hence, $P(\frac{n \pi \sqrt{-1}}{\hbar},b)_{\p p}$ is indecomposable. For any $\alpha \in \C \setminus \frac{\pi \sqrt{-1}}{\hbar} \Z$, the isomorphism \eqref{E:tensorSimpNongen} gives
\[
V(\alpha,b+1)_{\p 1} \otimes V(-\alpha + \frac{n \pi \sqrt{-1}}{\hbar},b)_{\p p}
\simeq
P(\frac{n \pi \sqrt{-1}}{\hbar},b)_{\p p}.
\]
Since $V(\alpha,b+1)_{\p 1}$ is projective by Proposition \ref{prop:genericProjective}, so too is $P(\frac{n \pi \sqrt{-1}}{\hbar},b)_{\p p}$. This proves the first statement. The second statement follows from the fact that $v^{\prime}$ generates $P(\frac{n \pi \sqrt{-1}}{\hbar},b)_{\p p}$.
\end{proof}

\begin{lemma}
\label{lem:enoughProj}
The category $\catq$ has enough projectives.
\end{lemma}

\begin{proof}
Let $V \in \catq$. For each $\alpha \in \C \setminus \frac{\pi \sqrt{-1}}{\hbar} \Z$, there is a surjection
\[
V(\alpha,0)_{\p 0} \otimes V(\alpha,0)_{\p 0}^* \otimes V \xrightarrow[]{\ev_{V(\alpha,0)_{\p 0}} \otimes \Id_V} \C \otimes V \simeq V.
\]
Since $V(\alpha,0)_{\p 0} $ is projective by Proposition \ref{prop:genericProjective}, so too is $V(\alpha,0)_{\p 0} \otimes V(\alpha,0)_{\p 0}^* \otimes V$.
\end{proof}

\subsection{Modified traces}
\label{sec:mTracegloo}
We establish the existence of a modified trace on the ideal $\Proj$ of projectives of $\catq$ and make some computations required for later use.

The following result can also be proved using the notion of ambidextrous objects introduced in \cite{GPT, GKP1}.
 
\begin{proposition}\label{prop:mTrace}
Up to a global scalar, there exists a unique m-trace $\mt$ on the ideal $\Proj \subset \catq$ of projective objects.
\end{proposition}
\begin{proof}
By Lemmas \ref{lem:locFin} and \ref{lem:enoughProj}, $\catq$ is a locally finite pivotal $\C$-linear tensor category with enough projectives. It follows that $\Proj$ has a unique non-trivial right m-trace if and only if the projective cover of the trivial module $\C \simeq \ve(0,0)_{\p 0}$ is self-dual \cite[Corollary 5.6]{GKP3}. By Lemma \ref{lem:projIndec}, the projective cover in question is $P(0,0)_{\p 0}$. For any $\alpha \in \C \setminus \frac{\pi \sqrt{-1}}{\hbar} \Z$, the isomorphisms \eqref{eq:dualIrred} and \eqref{E:tensorSimpNongen} give $P(0,0)_{\p 0} \simeq  V(\alpha,1)_{\p 0} \otimes V(\alpha,1)_{\p 0}^{*}$, establishing self-duality. By Proposition \ref{prop:braiding}, $\catq$ is braided. Hence, this right m-trace is an m-trace. 
\end{proof}

Note that $[1]_q=1$ for all $q \in \C \setminus \{0,\pm 1\}$. In particular, the module $V(1,0)_{\p 0}$ is projective. Fix the normalization of the m-trace of Proposition \ref{prop:mTrace} by requiring $\qd(V(1,0)_{\p0})=(q-q^{-1})^{-1}$.

Given $V, V^{\prime} \in \catq$, define
\[
\Phi_{V,V^{\prime}} = (\Id_{V^{\prime}} \otimes \ev_V)\circ (c_{V,{V^{\prime}}}\otimes \Id_{V^*})\circ (c_{{V^{\prime}},V}\otimes \Id_{V^*})\circ (\Id_{V^{\prime}}\otimes \tcoev_V)\in \End_{\catq}({V^{\prime}}).
\]
Extend the definition of $\Phi_{V,V^{\prime}}$ to formal $\C$-linear combinations of objects of $\catq$ by bilinearity. When $\End_{\catq}(V^{\prime}) \simeq \C$ set 
\[
S'(V,V^{\prime})=\left< \epsh{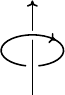}{9ex}\right>\put(-15,10){\ms{V}}\put(-36,18){\ms{V^{\prime}}}=\left<\Phi_{V,{V^{\prime}}}\right> \in \C.
\]

\begin{lemma}\label{lem:Phi's}
Let $\alpha,\beta, a,b \in\C$ and $\p p,\p s\in \Ztwo $. The following equalities hold:
$$\Phi_{V(\beta,b)_{\p p},V(\alpha,a)_{\p s}}=(-1)^{\p p + \p 1}q^{-2(\alpha b+ a\beta)}q^{\alpha+\beta}(q^{\alpha}-q^{-\alpha})\Id_{V(\alpha,a)_{\p s}}$$
$$\Phi_{P(\frac{n \pi \sqrt{-1}}{\hbar},b)_{\p p},V(\alpha,a)_{\p s}}=(-1)^{\p p+1} q^{-2(\alpha b + a \frac{n \pi \sqrt{-1}}{\hbar})}(q^\alpha-q^{-\alpha})^2\Id_{V(\alpha,a)_{\p s}}$$
$$\Phi_{V(\beta,b)_{\p p},P(\frac{n \pi \sqrt{-1}}{\hbar},a)_{\p s}}=(-1)^{\p p +\p 1}q^{-2(a\beta + b \frac{n \pi \sqrt{-1}}{\hbar})}(q-q^{-1})(q^{\beta}-q^{-\beta})x_{\frac{n \pi \sqrt{-1}}{\hbar},a,\p s}.$$
\end{lemma}

\begin{proof}
Using that $V(\alpha,a)_{\p s}$ is a highest weight module, it is straightforward to verify that $\End_{\catq}(V(\alpha,a)_{\p s}) \simeq \C$. It follows that $\Phi_{V,V(\alpha,a)_{\p s}}$ maps a highest weight vector of $V(\alpha,a)_{\p s}$ to a multiple of itself, whence only the diagonal part of the braiding contributes to $\Phi_{V,V(\alpha,a)_{\p s}}$. Using this observation, the first two equalities are direct calculations. Turning to the final equality, Proposition \ref{prop:HomProj} implies that
\[
\Phi_{V(\beta,b)_{\p p},P(\frac{n \pi \sqrt{-1}}{\hbar},a)_{\p s}}=c_1\Id_{P(\frac{n \pi \sqrt{-1}}{\hbar},a)_{\p s}}+ c_2 x_{\frac{n \pi \sqrt{-1}}{\hbar},a,\p s}
\]
for unique $c_1,c_2 \in \C$. Since $\Phi_{V(\beta,b)_{\p p},P(\frac{n \pi \sqrt{-1}}{\hbar},a)_{\p s}}$ annihilates the highest weight vector $v_+$, we have $c_1=0$. A  direct calculation shows that
\[
\Phi_{V(\beta,b)_{\p p},P(\frac{n \pi \sqrt{-1}}{\hbar},a)_{\p s}}(v^{\prime})=(-1)^{\p p +\p 1}q^{-2(a\beta + a \frac{n \pi \sqrt{-1}}{\hbar})}(q-q^{-1})(q^{\beta}-q^{-\beta})v = c_2 v,
\]
from which the lemma follows.  
\end{proof}

In preparation for the next result, we compute some modified dimensions. Let $\alpha \in \C \setminus \frac{\pi \sqrt{-1}}{\hbar} \Z$. Cyclicity of the m-trace implies $\mt_{V(1,0)_{\p0}}(\Phi_{V(\alpha,a)_{\p p},V(1,0)_{\p 0}})=\mt_{V(\alpha,a)_{\p p}}(\Phi_{V(1,0)_{\p 0},V(\alpha,a)_{\p p}})$. Using Lemma \ref{lem:Phi's} to compute both sides of this equality gives
\begin{equation}
\qd(V(\alpha,a)_{\p p})
=
(-1)^{\p p}(q^{\alpha}-q^{-\alpha})^{-1}.
\label{eq:mdimTypical}
\end{equation}
The isomorphism \eqref{eq:dualIrred} and equation \eqref{eq:mdimTypical} imply
\begin{equation}\label{eq:dualDimEqual}
\qd(V(\alpha,a)_{\p p}^*)
=
\qd(V(\alpha,a)_{\p p}).
\end{equation}

\begin{lemma}\label{lem:modTrComp}
Let $(n,b) \in \Z \times \C$ and $\p p \in \Ztwo$. Then $\qd(P(\frac{n \pi \sqrt{-1}}{\hbar},b)_{\p p})=0$ and
\[
\mt_{P(\frac{n \pi \sqrt{-1}}{\hbar},b)_{\p p}}(x_{\frac{n \pi \sqrt{-1}}{\hbar},b,\p p})= (-1)^{\p p}(q-q^{-1})^{-1}.
\]
\end{lemma}

\begin{proof}
Let $\alpha \in \C \setminus \frac{\pi \sqrt{-1}}{\hbar} \Z$. The isomorphism \eqref{E:tensorSimpNongen} gives
\[
V(\alpha + \frac{n \pi \sqrt{-1}}{\hbar},b)_{\p p}\otimes V(-\alpha,1)_{\p 1}\simeq  P(\frac{n \pi \sqrt{-1}}{\hbar},b)_{\p p}.
\]
The partial trace property of the m-trace gives
\begin{eqnarray*}
\qd(P(\frac{n \pi \sqrt{-1}}{\hbar},b)_{\p p})
&=&
\mt_{V(\alpha+\frac{n \pi \sqrt{-1}}{\hbar},b)_{\p p}\otimes V(-\alpha,1)_{\p 1}}(\Id_{V(\alpha+\frac{n \pi \sqrt{-1}}{\hbar},b)_{\p p}\otimes V(-\alpha,1)_{\p 1}}) \\
&=&
\mt_{V(\alpha+\frac{n \pi \sqrt{-1}}{\hbar},b)_{\p p}}(\ptr_{V(-\alpha,1)_{\p 1}}(\Id_{V(\alpha+\frac{n \pi \sqrt{-1}}{\hbar},b)_{\p p}}\otimes \Id_{V(-\alpha,1)_{\p 1}}))\\
&=&
\mt_{V(\alpha+\frac{n \pi \sqrt{-1}}{\hbar},b)_{\p p}}(\Id_{V(\alpha+\frac{n \pi \sqrt{-1}}{\hbar},b)_{\p p}})\qdim V(-\alpha,1)_{\p 1}\\
&=&
0
\end{eqnarray*}
since $\qdim V(-\alpha,1)_{\p 1}=0$. Turning to the second equality, cyclicity of the m-trace implies
$$
\mt_{V(\alpha,a)_{\p s}}(\Phi_{P(\frac{n \pi \sqrt{-1}}{\hbar},b)_{\p p},V(\alpha,a)_{\p s}})=\mt_{{P(\frac{n \pi \sqrt{-1}}{\hbar},b)_{\p p}}}(\Phi_{V(\alpha,a)_{\p s},{P(\frac{n \pi \sqrt{-1}}{\hbar},b)_{\p p}}}).
$$
Taking the m-trace of the second and third equalities of Lemma \ref{lem:Phi's}, equating them and solving for $\mt_{P(\frac{n \pi \sqrt{-1}}{\hbar},a)_{\p p}}(x_{\frac{n \pi \sqrt{-1}}{\hbar},a,\p p})$ completes the proof.
\end{proof}

\begin{lemma}\label{lem:endAlgIsom}
Let $\alpha \in \C \setminus \frac{\pi \sqrt{-1}}{\hbar} \Z$ and $I$ the isomorphism \eqref{E:tensorSimpNongen}. The algebra isomorphism
\[
I^*: \End_{\catq}(V(\alpha,0)_{\p 0} \otimes V(\alpha,0)_{\p 0}^*)
\xrightarrow[]{\sim}
\End_{\catq}(P(0,0)_{\p 0}),
\qquad
f \mapsto I \circ f \circ I^{-1}
\]
maps $\tcoev_{V(\alpha,0)_{\p 0}} \circ \ev_{V(\alpha,0)_{\p 0}}$ to $(q-q^{-1}) \qd(V(\alpha,0)_{\p 0}) \cdot x_{0,0,\p 0}$.
\end{lemma}

\begin{proof}
Write $V$ for $V(\alpha,0)_{\p 0}$. By Proposition \ref{prop:HomProj}, there exist $c_1,c_2 \in \C$ such that
\begin{equation}
\label{eq:muPullback}
I^* (\tcoev_{V} \circ \ev_{V}) = c_1 \Id_{P(0,0)_{\p 0}} + c_2 x_{0,0, \p 0}.
\end{equation}
Since $\tcoev_{V} \circ \ev_{V}$ is nilpotent, $c_1=0$. Taking the m-trace of the left hand side of equation \eqref{eq:muPullback} gives
\begin{eqnarray*}
\mt_{P(0,0)_{\p 0}} (I^* (\tcoev_{V} \circ \ev_{V}))
&=&
\mt_{V \otimes V^*} (\tcoev_{V} \circ \ev_{V}) \\
&=&
\mt_{V} \left( \ptr_{V^*} (\tcoev_{V} \circ \ev_{V}) \right) \\
&=&
\mt_{V} (\Id_{V}) \\
&=&
\qd(V).
\end{eqnarray*}
The first and second equalities follow from the cyclicity and partial trace properties of the m-trace, respectively, and the third from the snake axioms of left and right dualities after using the canonical identifications $\tcoev_{V^*} = \coev_V$ and $\ev_{V^*} = \tev_V$. On the other hand, Lemma \ref{lem:modTrComp} shows that the m-trace of the right hand side of equation \eqref{eq:muPullback} is $c_2 \mt_{P(0,0)_{\p 0}} (x_{0,0,\p0}) = c_2 (q-q^{-1})^{-1}$. It follows that $c_2 = (q-q^{-1}) \qd(V(\alpha,0)_{\p 0})$.
\end{proof}

\subsection{Generic semisimplicity and the ribbon structure of $\catq$}
\label{sec:genSSGeneric}

Let $\Gr = \C$. For each $\alpha \in \Gr$, denote by $\catq_{\alpha} \subset \catq$ the full subcategory of modules on which $E$ acts by $\alpha$. 

\begin{proposition}\label{prop:genericSS}
The $\Gr$-graded category $\catq \simeq \bigoplus_{\alpha \in \Gr}\catq_{\alpha}$ is generically semisimple with small symmetric subset $\XX = \frac{\pi \sqrt{-1}}{\hbar} \Z$. Moreover, if $\alpha \in \Gr \setminus \XX$, then $\{V(\alpha,a)_{\p p} \mid a \in \C, \, \p p \in \Ztwo\}$ is a completely reduced dominating set for $\catq_{\alpha}$.
\end{proposition}

\begin{proof}
That $\catq \simeq \bigoplus_{\alpha \in \Gr}\catq_{\alpha}$ is a $\Gr$-grading follows from the definition of the coproduct and antipode of $\Uq$. Let $\alpha \in \Gr \setminus \XX$ and $V \in \catq_{\alpha}$ non-zero. As in the proof of Proposition \ref{prop:irrepsgloo}, there exists a highest weight vector $v \in V$. Since $\alpha \notin \frac{\pi \sqrt{-1}}{\hbar} \Z$, this vector generates a submodule $\langle v \rangle \subset V$ isomorphic to $V(\alpha,a)_{\p p}$ for some $a \in \C$ and $\p p\in \Ztwo$ which, by Proposition \ref{prop:genericProjective}, is simple and injective. It follows that there exists a splitting $V\simeq V^{\prime} \oplus \langle v \rangle$. Iterating this process shows that $\catq_{\alpha}$ is semisimple with the claimed completely reduced dominating set.
\end{proof}

Let $\theta : \Id_{\catq} \Rightarrow \Id_{\catq}$ be the natural automorphism with components $\theta_V = \ptr_R(c_{V,V})$, $V \in \catq$.

\begin{theorem}
\label{thm:ribbonCat}
The natural automorphism $\theta$ gives $\catq$ the structure of a $\C$-linear ribbon category.
\end{theorem}

\begin{proof}
We have already seen in Lemma \ref{lem:catqPivot} that $\catq$ is pivotal and in Proposition \ref{prop:braiding} that $\catq$ is braided. It is automatic that $\theta$ satisfies the balancing conditions. Let $(\alpha,a) \in \C^2$ and $\p p \in \Ztwo$. We claim that $\theta_{V(\alpha,a)_{\p p}} = q^{-2 \alpha a+\alpha} \Id_{V(\alpha,a)_{\p p}}$. Since $\End_{\catq}(V(\alpha,a)_{\p p}) \simeq \C$, it suffices to compute $\theta_{V(\alpha,a)_{\p p}}v$ for a highest weight vector $v$:
\begin{eqnarray*}
v
&\xmapsto{\tcoev_{V(\alpha,a)_{\p p}}}&
v \otimes v \otimes v^* + v \otimes v^{\prime} \otimes v^{\prime*}\\
&\xmapsto{c_{V(\alpha,a)_{\p p},V(\alpha,a)_{\p p}}}&
(-1)^{\p p}q^{-2\alpha a} v\otimes v \otimes v^* + (-1)^{\p p + \p 1}q^{-2\alpha a + \alpha} v^{\prime} \otimes v \otimes v^{\prime*}\\
&\xmapsto{\ev_{V(\alpha,a)_{\p p}}}&
q^{-2\alpha a + \alpha} v.
\end{eqnarray*}
View $\catq$ as generically semisimple as in Proposition \ref{prop:genericSS}. Using the classification of simples from Proposition \ref{prop:irrepsgloo} and the isomorphism \eqref{eq:dualIrred}, we conclude that $\theta_{V(\alpha,a)_{\p p}^*} = \theta_{V(\alpha,a)_{\p p}}^*$ for all generic simples, that is, when $\alpha \in \Gr \setminus \XX$. The assumptions of \cite[Theorem 9]{GP1} are therefore satisfied and we conclude that $\theta$ is a compatible twist.
\end{proof}

\begin{remark}
Variants of Proposition \ref{prop:braiding} and Theorem \ref{thm:ribbonCat} are known:
\begin{enumerate}
\item Following Kulish \cite{kulish1989} and Khoroshkin and Tolstoy \cite[\S 7]{khoroshkin1991}, Viro explained that $\Uhgloo$, the $h$-adic quantum group of $\gloo$, admits a ribbon structure. In particular, the category of topological $\Uhgloo$-modules is ribbon \cite[\S 11.3]{Viro}.

\item Sartori proved that the category of finite dimensional weight modules over $U_{\mathbf{q}}(\gloo)$, the quantum group of $\gloo$ over $\C(\mathbf{q})$ with $\mathbf{q}$ an indeterminant, is ribbon by transferring the ribbon structure of $\Uhgloo$ \cite[\S 4.1]{sartori2015}, following the method of Tanisaki \cite[\S 4]{tanisaki1992}.
\end{enumerate}
Our proof of Proposition \ref{prop:braiding} is similar in spirit to that of Sartori \cite{sartori2015}, in that the definition of the ribbon structure of $\catq$ is motivated by the form of the universal $R$-matrix of $\Uhgloo$. Our proof of Theorem \ref{thm:ribbonCat} is however different, leveraging the generic semisimplicity of $\catq$.
\end{remark}

\subsection{Relative modularity of $\catq$ for arbitrary $q$}
\label{sec:relModArb}

Continue to denote $\Gr = \C$ and $\XX=\frac{\pi \sqrt{-1}}{\hbar} \Z$ and view $\catq$ as generically semisimple as in Proposition \ref{prop:genericSS}. Let $\Zt= \C \times \Ztwo$. Then $\{\ve(0,b)_{\p p}\}_{(b, \p p)\in \Zt} $ is a free realisation of $\Zt$ in $\catq_{0}$. Indeed, it is immediate that $\theta_{\ve(0,b)_{\p p}} = \Id_{\ve(0,b)_{\p p}}$. The remaining parts of Definition \ref{def:free} follow from equation \eqref{eq:qdimVE}, the isomorphisms \eqref{eq:tensorOneDim} and \eqref{eq:tensorOneWithTwoDim} and Proposition \ref{prop:irrepsgloo}. For each $V \in \catq_{\alpha}$, a direct computation gives
\begin{equation*}\label{E:CompCatq}
c_{V,\ve(0,b)_{\p s}}\circ c_{\ve(0,b)_{\p s},V}=q^{-2\alpha b}\Id_{\ve(0,b)_{\p s} \otimes V}.
\end{equation*}
Thus, $\psi: \Gr \times\Zt \rightarrow \C^{\times}$, $(\alpha, (b, \p p)) \mapsto q^{-2\alpha b}$, is a bicharacter satisfying equation \eqref{eq:psi}.

\begin{theorem}\label{thm:relModArb}
The category $\catq$ is a modular $\C$-category relative to $(\Zt=\C \times \Ztwo, \XX=\frac{\pi \sqrt{-1}}{\hbar}\Z)$ with relative modularity parameter $\zeta=-1$. Moreover, $\catq$ is TQFT finite.
\end{theorem}

\begin{proof}
As explained above, $\catq$ is $\Gr$-graded and $\{\ve(0,b)_{\p p}\}_{(b, \p p)\in \Zt}$ is a free realisation of $\Zt$ in $\catq_0$ satisfying the compatibility condition of Definition \ref{def:preMod}. Proposition \ref{prop:mTrace} shows that the full subcategory of projectives $\Proj \subset \catq$ has an m-trace. Define $\Theta(\alpha)=\{V(\alpha,0)_{\p 0}\}$, $\alpha \in \Gr \setminus \XX$. The isomorphisms \eqref{eq:tensorOneWithTwoDim} show that $\Theta(\alpha) \otimes \sigma(\Zt) = \{V(\alpha,a)_{\p p} \mid a \in \C, \; \p p \in \Ztwo\}$ which, by Proposition \ref{prop:genericSS}, is a completely reduced dominating set for $\catq_{\alpha}$. Thus, $\catq$ is a pre-modular $\Gr$-category relative to $(\XX,\Zt)$.

Next, we find a relative modularity parameter $\zeta$. Let $\alpha, \beta \in \Gr \setminus \XX$. Up to isomorphism and the action of $\Zt$, each of the categories $\catq_{\alpha}$ and $\catq_\beta$ have a unique simple object, namely $V(\alpha,0)_{\p 0}$ and $V(\beta,0)_{\p 0}$, respectively. Let $f$ be the morphism corresponding to the left hand side of equation \eqref{eq:mod} with $\Omega_h = \Omega_\beta$ and $V_i=V_j=V(\alpha,0)_{\p 0}$ and $\tilde{I}:V(\alpha,0)_{\p 0}\otimes V(\alpha,0)_{\p 0}^* \to P(0,0)_{\p 0}$ the isomorphism obtained by combining the isomorphisms \eqref{eq:dualIrred} and \eqref{E:tensorSimpNongen}. We compute
\begin{align*}
\tilde{I}^*f &= \qd(V(\alpha,0)_{\p 0}) \Phi_{\Omega_{\beta}, P(0,0)_{\p 0}} \\
&=
\qd(V(\alpha,0)_{\p 0}) \qd(V(\beta,0)_{\p 0}) \Phi_{V(\beta,0)_{\p 0}, P(0,0)_{\p 0}}\\
&=
-\qd(V(\alpha,0)_{\p 0}) (q-q^{-1}) x_{0,0,\p 0}.
\end{align*}
The third equality follows from Lemma \ref{lem:Phi's}. Lemma \ref{lem:endAlgIsom} then gives $f = - \tcoev_{V(\alpha,0)_{\p 0}}\circ \ev_{V(\alpha,0)_{\p 0}}$, whence $\zeta=-1$.

Turning to TQFT finiteness, in the notation of Definition \ref{def:relModFinite} we have $J_{\alpha}=\{V(\alpha,0)_{\p 0}\}$ if $\alpha \in \Gr \setminus \XX$ and $J_{\alpha}=\{P(\alpha,0)_{\p 0}\}$ if $\alpha \in \XX$, whence property \ref{ite:finCat1} holds. Since $\catq$ is locally finite abelian (see Lemma \ref{lem:locFin}), it is Krull--Schmidt. Properties \ref{ite:finCat2}-\ref{ite:finCat3} follow easily from this and the classification of projective indecomposables from Section \ref{sec:weightMod}.
\end{proof}

Let $\alpha \in \Gr \setminus \XX$. The left hand side of the diagram in Definition \ref{def:ndeg} which defines $\Delta_-$ is equal to
\[
\qd(V(\alpha,0)_{\p 0}) q^{-\alpha} q^{-\alpha}\Phi_{V(\alpha,0)_{\p 0},V(\alpha,0)_{\p 0}}
=
- \Id_{V(\alpha,0)_{\p 0}},
\]
the two factors of $q^{-\alpha}$ being due to the inverse twists and the equality following from Lemma \ref{lem:Phi's} and equation \eqref{eq:mdimTypical}. Hence, $\Delta_- = -1$, from which it follows that $\Delta_+=1$.

\begin{remark}
Since relative modularity implies non-degenerate relative pre-modularity, Theorem \ref{thm:relModArb} implies that $\catq$ defines invariants of decorated closed $3$-manifolds; see Section \ref{sec:CGPInvt}. We expect that these invariants recover those of Bao and Ito \cite{bao2022}, who define invariants using a category closely related to $\catq$, defined using Viro's \emph{$q$-less subalgebra} of $\Uhgloo$ \cite[\S 11.7]{Viro}. 
\end{remark}

\subsection{Relative modularity of $\catq$ for $q$ a root of unity}
\label{sec:relModROU}

Assume that $q$ is a root of unity. In this section we endow $\catq$ with a relative modular structure different from that of Theorem \ref{thm:relModArb}. Since the gradings involved depend on the parity of the order of the root of unity, we treat even and odd roots of unity separately. Many of the proofs are variations of those from Section \ref{sec:relModArb}, so we will at points be brief.

\subsubsection{Odd roots of unity}
\label{sec:relModROUOdd}

Let $r\geq 3$ be an odd positive integer and $\hbar=\frac{2 \pi \sqrt{-1}}{r}$. Then $V(\alpha,a)_{\p p}$ is simple if and only if $\alpha \in \C \setminus \frac{r}{2} \Z$. Let $\Gr = \C \slash \Z \times \C \slash \Z$. For each $(\p \alpha, \p a) \in \Gr$, let $\catq_{(\p \alpha,\p a)} \subset \catq$ be the full subcategory of modules whose weights are congruent to $(\p \alpha, \p a)$.

\begin{proposition}\label{prop:genericSSROUOdd}
The $\Gr$-graded category $\catq \simeq \bigoplus_{(\p \alpha, \p a) \in \Gr}\catq_{(\p \alpha, \p a)}$ is generically semisimple with small symmetric subset $\XX=\{(\p 0, \p 0),(\p {\frac{1}{2}}, \p 0)\}$. Moreover, if $(\p \alpha, \p a) \in \Gr \setminus \XX$, then $\{V(\alpha,a)_{\p p} \mid (\alpha,a,\p p)\in \p \alpha \times \p a \times \Ztwo\}$ is a completely reduced dominating set for $\catq_{(\p \alpha, \p a)}$.
\end{proposition}

\begin{proof}
Let $(\p \alpha, \p a) \in \Gr \setminus \XX$ and $V \in \catq_{(\p \alpha, \p a)}$ non-zero. Since $\p \alpha \neq \p 0, \p {\frac{1}{2}}$, a highest weight vector $v \in V$ generates a submodule $\langle v \rangle \subset V$ isomorphic to $V(\alpha,a)_{\p p}$ for some $\p p\in \Ztwo$ and lift $(\alpha,a) \in \C^2$ of $(\p \alpha, \p a)$. In particular, $\alpha \notin \frac{r}{2} \Z$ and $V(\alpha,a)_{\p p}$ is simple and injective by Proposition \ref{prop:genericProjective}.
 Thus, there is a splitting $V\simeq V^{\prime} \oplus \langle v \rangle$. Iterating this process proves the proposition.  
\end{proof}

Let $\Zt=\Z \times \Z \times \Ztwo$. For $(n,n^{\prime},\p p) \in \Zt$, we compute
\[
\theta_{\ve(nr, n^{\prime}r)_{\p p}}
=
q^{-2nn^{\prime} r^2+nr} \Id_{\ve(nr, n^{\prime}r)_{\p p}}
=
\Id_{\ve(nr, n^{\prime}r)_{\p p}}.
\]
It follows easily from this that $\{\ve(nr,n^{\prime} r)_{\p p}\}_{(n,n^{\prime}, \p p)\in \Zt} $ defines a free realisation of $\Zt$ in $\catq_{(\p 0, \p 0)}$.
For each $V \in \catq_{(\p \alpha, \p a)}$, we compute
\[
c_{V,\ve(nr,n^{\prime} r)_{\p s}}\circ c_{\ve(n r,n^{\prime} r)_{\p s},V}=q^{-2(n\p ar+n'\p \alpha r)}\Id_{\ve(nr,n^{\prime} r)_{\p s} \otimes V}.
\]
Thus, $\psi: \Gr \times\Zt \rightarrow \C^{\times}$, $((\p \alpha,\p a), (n,n^{\prime}, \p s)) \mapsto q^{-2(n\p a+n'\p \alpha )r}$, is a bicharacter satisfying equation \eqref{eq:psi}. Let $(\p \alpha, \p a) \in \Gr \setminus \XX$ with chosen lift $(\alpha_0,a_0) \in \C \times \C$. The set
\[
\Theta(\p \alpha, \p a)
=
\{V(\alpha_0+i, a_0 + j) _{\p 0} \mid i,j \in \{0, \dots, r-1\} \}
\]
satisfies $\Theta(\p \alpha, \p a) \otimes \sigma(\Zt)=\{V(\alpha,a)_{\p p} \mid (\alpha,a,\p p)\in \p \alpha \times \p a \times \Ztwo\}$, which is a completely reduced dominating set for $\catq_{(\p \alpha, \p a)}$ by Proposition \ref{prop:genericSSROUOdd}.

\begin{lemma}
\label{lem:betaPhiROU}
Let $(\p \beta, \p b) \in \Gr \setminus \XX$ and $i,j,k \in \{0, \dots, r-1\}$.
\begin{enumerate}
\item \label{ite:betaPhiROU1} If $i \neq 0$, then $\Phi_{\Omega_{(\p \beta, \p b)},V(i,j)_{\p p}}=0$.
\item \label{ite:betaPhiROU2} The equality $\Phi_{\Omega_{(\p \beta, \p b)},P(0,j)_{\p p}}= -\delta_{j,0}(q-q^{-1})r^2 x_{0,0,\p 0} $ holds.
\item \label{ite:betaPhiROU3} If $i \neq 0$, then $\Phi_{\Omega_{(\p \beta, \p b)},V(i,j)_{\p 0} \otimes P(0,k)_{\p 0}}=0$.
\end{enumerate}
\end{lemma}

\begin{proof}
\begin{enumerate}
\item Using Lemma \ref{lem:Phi's} and equation \eqref{eq:mdimTypical}, we compute
\begin{eqnarray*}
\Phi_{\Omega_{(\p \beta, \p b)},V(i,j)_{\p p}}
&=&
\sum_{m,n=0}^{r-1} \qd(V(\beta_0 + m, b_0 +n)_{\p 0}) \Phi_{V(\beta_0 + m, b_0 +n)_{\p 0},V(i,j)_{\p p}} \\
&=&
\sum_{m,n=0}^{r-1} (q^{\beta_0 + m}-q^{-\beta_0 -m})^{-1}q^{-2(j(\beta_0 + m) + i (b_0+n)} q^{\beta_0+m}(1-q^{2i}) \Id_{V(i,j)_{\p p}} \\
&=&
\sum_{m=0}^{r-1} \Big( \sum_{n=0}^{r-1} q^{-2in} \Big) (q^{\beta_0 + m}-q^{-\beta_0 -m})^{-1}q^{-2(j(\beta_0 + m) + i b_0} q^{\beta_0+m}(1-q^{2i}) \Id_{V(i,j)_{\p p}}.
\end{eqnarray*}
Since $i \neq 0$ and $r$ is odd, $r$ does not divide $2i$, whence  $q^{-2i} \neq 1$. The sum over $n$ therefore vanishes.

\item Lemma \ref{lem:Phi's} gives
\[
\Phi_{\Omega_{(\p \beta, \p b)}, P(0,j)_{\p 0}}
=
-(q - q^{-1}) \sum_{n=0}^{r-1} \sum_{m=0}^{r-1} q^{-2j(\beta_0 + m)} x_{(0,j,\p 0)}.
\]
If $j \neq 0$, then the sum over $m$ vanishes for the same reason as in the previous part.

\item The isomorphisms \eqref{E:tensorSimpGen} and \eqref{E:tensorSimpNongen} give
\[
V(i,j)_{\p 0} \otimes P(0,k)_{\p 0}
\simeq
V(i,j+k-1)_{\p 1} \oplus V(i,j+k)^{\oplus 2}_{\p 0} \oplus V(i,j+k+1)_{\p 1} .
\]
The desired equality therefore follows from the first part of the lemma and additivity of $\Phi$ in its right index.\qedhere
\end{enumerate}
\end{proof}

\begin{theorem}\label{thm:relModROUOdd}
The category $\catq$ is a modular $\C \slash \Z \times \C \slash \Z$-category relative to $\Zt= \Z \times \Z \times \Ztwo$ and $\XX=\frac{1}{2}  \Z \slash \Z$ with relative modularity parameter $\zeta=-r^2$.  Moreover, $\catq$ is TQFT finite.
\end{theorem}

\begin{proof}
Relative pre-modularity and TQFT finiteness are proved as in Theorem \ref{thm:relModArb}, the only difference being that Proposition \ref{prop:genericSSROUOdd} is used in place of Proposition \ref{prop:genericSS}. To find a relative modularity parameter $\zeta$, let $(\p \alpha, \p a), (\p \beta, \p b) \in \Gr \setminus \XX$. In the notation of equation \eqref{eq:mod}, take
\[
\Omega_h
=
\Omega_{(\p \beta, \p b)}
=
\sum_{i,j=0}^{r-1} \qd(V(\beta_0 + i,b_0+j)_{\p 0}) \cdot V(\beta_0 + i,b_0+j)_{\p 0}
\]
and $V_i = V(\alpha_0+i_1,a_0+j_1)_{\p 0}$ and $V_j = V(\alpha_0+i_2,a_0+j_2)_{\p 0}$ for some $i_1,i_2,j_1,j_2 \in \{0, \dots, r-1\}$. Write $M$ for $V(\alpha_0 + i_1,a_0 + j_1)_{\p 0}\otimes V(\alpha_0 + i_2,a_0 + j_2)^*_{\p 0}$, so that the morphism corresponding to the left hand side of equation \eqref{eq:mod} is $f \in \End_{\catq} (M)$. Using the isomorphisms \eqref{E:tensorSimpGen} and \eqref{E:tensorSimpNongen}, we have
\[
I: M \xrightarrow[]{\sim}
\begin{cases}
V(i_1 - i_2,j_1-j_2+1)_{\p 0} \oplus V(i_1 - i_2,j_1-j_2)_{\p 1} & \mbox{if } i_1 \neq i_2, \\
P(0,j_1-j_2)_{\p 0} & \mbox{if } i_1=i_2.
\end{cases}
\]
If $i_1 \neq i_2$, then
\begin{align*}
I^*f &= \qd(V(\alpha_0+i_1,a_0+j_1)_{\p 0}) \Phi_{\Omega_{(\p \beta, \p b)}, V(i_1 - i_2,j_1-j_2+1)_{\p 0} \oplus V(i_1 - i_2,j_1-j_2)_{\p 1})} \\
&=
\qd(V(\alpha_0+i_1,a_0+j_1)_{\p 0}) \cdot \\ & \sum_{i,j=0}^{r-1}
\qd(V(\beta_0 + i,b_0+j)_{\p 0}) \left(\Phi_{V(\beta_0 + i,b_0+j)_{\p 0},V(i_1 - i_2,j_1-j_2+1)_{\p 0}} + \Phi_{V(\beta_0 + i,b_0+j)_{\p 0},V(i_1 - i_2,j_1-j_2)_{\p 1})} \right),
\end{align*}
which vanishes by part (\ref{ite:betaPhiROU1}) of Lemma \ref{lem:betaPhiROU}.
If instead $i_1=i_2$, then
\begin{align*}
I^*f &= \qd(V(\alpha_0+i_1,a_0+j_1)_{\p 0}) \Phi_{\Omega_{(\p \beta, \p b)}, P(0,j_1-j_2)_{\p 0}} \\
&=
\begin{cases}
0 & \mbox{if } j_1\neq j_2, \\
-r^2 \qd(V(\alpha_0+i_1,a_0+j_1)_{\p 0}) (q - q^{-1}) x_{(0,0,\p 0)} & \mbox{if } j_1=j_2,
\end{cases}
\end{align*}
the final equality following from part (\ref{ite:betaPhiROU2}) of Lemma \ref{lem:betaPhiROU}. Using Lemma \ref{lem:endAlgIsom}, we conclude that $f = -r^2 \tcoev_{V(\alpha_0+i_1,a_0+j_1)_{\p 0}}\circ\ev_{V(\alpha_0+i_1,a_0+j_1)_{\p 0}}$ when $j_1=j_2$. It follows that $\zeta = -r^2$.
\end{proof}


Let $(\p \alpha, \p a) \in \Gr \setminus \XX$. The left hand side of the diagram in Definition \ref{def:ndeg} which defines $\Delta_-$ is equal to
\begin{eqnarray*}
&&\sum_{i,j=0}^{r-1} \qd(V(\alpha_0+i,a_0+j)_{\p 0}) q^{2(\alpha_0+i)(a_0+j) - (\alpha_0+i)} q^{-\alpha_0}\Phi_{V(\alpha_0+i,a_0+j)_{\p 0},V(\alpha_0,0)_{\p 0}} \\
&=&
\sum_{i,j=0}^{r-1} (q^{\alpha_0+i}-q^{-\alpha_0-i})^{-1} q^{2i(a_0+j) - (\alpha_0+i)}q^{i}(1-q^{2\alpha_0}) \Id_{V(\alpha_0,0)_{\p 0}}.
\end{eqnarray*}
The sum over $j$ is zero unless $i=0$, so that the last expression is equal to $-r\Id_{V(\alpha,0)_{\p 0}}$.
Hence, $\Delta_- = -r$, from which it follows that $\Delta_+=r$. 

\subsubsection{Even roots of unity}
\label{sec:relModROUEven}

Let $r\geq 2$ be an even positive integer and $\hbar=\frac{\pi \sqrt{-1}}{r}$. Then $V(\alpha,a)_{\p p}$ is simple if and only if $\alpha \notin \C \setminus r \Z$. Let $\Gr = \C \slash 2\Z \times \C \slash \Z$. For each $(\p \alpha, \p a) \in \Gr$, denote by $\catq_{(\p \alpha,\p a)} \subset \catq$ the full subcategory of modules whose weights are congruent to $(\p \alpha, \p a)$.

\begin{proposition}\label{prop:genericSSROUEven}
The $\Gr$-graded category $\catq \simeq \bigoplus_{(\p \alpha, \p a) \in \Gr}\catq_{(\p \alpha, \p a)}$ is generically semisimple with small symmetric subset $\XX=\{(\p 0, \p 0)\}$. Moreover, if $(\p \alpha, \p a) \in \Gr \setminus \XX$, then $\{V(\alpha,a)_{\p p} \mid (\alpha,a,\p p)\in \p \alpha \times \p a \times \Ztwo\}$ is a completely reduced dominating set for $\catq_{(\p \alpha, \p a)}$.
\end{proposition}

\begin{proof}
This is proved in the same way as Proposition \ref{prop:genericSSROUOdd}.
\end{proof}


Let 
$\Zt=\Z \times \Z \times \Ztwo$. Then
$\{\ve(2nr,n^{\prime} r)_{\p p}\}_{(n,n^{\prime}, \p p)\in \Zt}$ is a free realisation of $\Zt$ in $\catq_{(\p 0, \p 0)}$ and
\[
c_{V,\ve(2nr,n^{\prime} r)_{\p s}}\circ c_{\ve(2n r,n^{\prime} r)_{\p s},V}=q^{-2(2n\p a r+n' \p \alpha r)}\Id_{\ve(2nr,n^{\prime} r)_{\p s} \otimes V}
\]%
for any $V \in \catq_{(\p \alpha, \p a)}$. Thus, $\psi: \Gr \times\Zt \rightarrow \C^{\times}$,
$((\p \alpha,\p a), (n,n^{\prime}, \p s)) \mapsto q^{-2(2n \p a+n' \p \alpha )r}$, is a bicharacter satisfying equation \eqref{eq:psi}. Let $(\p \alpha, \p a) \in \Gr$ with chosen lift $(\alpha_0,a_0) \in \C \times \C$. The set
\[
\Theta(\p \alpha, \p a)
=
\{V(\alpha_0+i, a_0 + j) _{\p 0} \mid i,j \in \{0, \dots, r-1\} \}
\]
satisfies $\Theta(\p \alpha, \p a) \otimes \sigma(\Zt) = \{V(\alpha,a)_{\p p} \mid (\alpha,a,\p p)\in \p \alpha \times \p a \times \Ztwo\}$, which is a completely reduced dominating set for $\catq_{(\p \alpha, \p a)}$ by Proposition \ref{prop:genericSSROUEven}. 

Lemma \ref{lem:betaPhiROU} holds in the present setting with only obvious modifications in its proof. For example, part (\ref{ite:betaPhiROU1}) now follows from the fact that since $r$ does not divide $i$ and $q$ is a primitive $2r$\textsuperscript{th} root of unity, $q^{-2i} \neq 1$.

\begin{theorem}\label{thm:relModROUEven}
The category $\catq$ is a TQFT finite modular $\C \slash 2\Z \times \C \slash \Z$-category relative to $\Zt= \Z \times \Z \times \Ztwo$ and $\XX=\{(\p 0, \p 0)\}$ with relative modularity parameter $\zeta=-r^2$.\end{theorem}

\begin{proof}
This is proved as in Theorem \ref{thm:relModROUOdd}, the only difference being that Proposition \ref{prop:genericSSROUEven} is used in place of Proposition \ref{prop:genericSSROUOdd}.
\end{proof}

A calculation analogous to that of Section \ref{sec:relModROUOdd} gives $\Delta_{\pm} = \pm r$.

\subsection{Integral weight modules}
\label{sec:intMod}

Let $\catInt \subset \catq$ be the full subcategory of modules whose $G$-weights are integral. The results of Sections \ref{sec:weightMod} and \ref{sec:mTracegloo} generalize directly to $\catInt$ with essentially the same proofs. In particular, there exists an m-trace on the ideal of projectives $\Proj^{\textnormal{int}} \subset \catInt$ which is unique up to a global scalar, which we normalize as for $\catq$ and $\catInt$ has a ribbon structure. The categories $\catInt$ also admit various relative modular structures, summarized below, which are integral counterparts of those of Sections \ref{sec:relModArb} and \ref{sec:relModROU}.

\subsubsection{Relative modularity of $\catInt$ for arbitrary $q$}
\label{sec:relModIntArb}

Let $\Gr = \C \slash \frac{2 \pi \sqrt{-1}}{\hbar} \Z$. For each $\p \alpha \in \Gr$, let $\catInt_{\p \alpha} \subset \catInt$ be the full subcategory of modules whose $E$-weights are congruent to $\p \alpha$. Then $\catInt$ is $\Gr$-graded and generically semisimple with $\XX = \frac{\pi \sqrt{-1}}{\hbar} \Z \slash \frac{2 \pi \sqrt{-1}}{\hbar} \Z$. The objects $\{\ve(\frac{2\pi \sqrt{-1} n}{\hbar},n^{\prime})_{\p p}\}_{(n,n^{\prime}, \p p)\in \Zt}$ define a free realisation of $\Zt= \Z \times \Z \times \Ztwo$ in $\catInt_{\p 0}$ and $\psi: \Gr \times\Zt \rightarrow \C^{\times}$, $(\p \alpha, (n,n^{\prime}, \p p)) \mapsto q^{-2\p \alpha n^{\prime}}$, is a bicharacter satisfying equation \eqref{eq:psi}. For each $\p \alpha \in \Gr \setminus \XX$, fix a lift $\alpha_0 \in \C$ and define $\Theta(\p \alpha)=\{V(\alpha_0,0)_{\p 0}\}$. Then $\Theta(\p \alpha) \otimes \sigma(\Zt)$ is a completely reduced dominating set for $\catInt_{\p \alpha}$. 

\begin{theorem}\label{thm:relModIntArb}
The category $\catInt$ is a TQFT finite modular $\C \slash \frac{2 \pi \sqrt{-1}}{\hbar} \Z$-category relative to $(\Zt=\Z \times \Z \times \Ztwo,\XX=\frac{\pi \sqrt{-1}}{\hbar} \Z \slash \frac{2 \pi \sqrt{-1}}{\hbar} \Z)$ with relative modularity parameter $\zeta=-1$.
\end{theorem}


The stabilization coefficients are $\Delta_{\pm} = \pm 1$. 

\begin{remark}
The category $\catq$ does not admit a relative pre-modular structure with respect to the smaller grading $\Gr = \C \slash \frac{2 \pi \sqrt{-1}}{\hbar} \Z$, since no compatible bicharacter $\psi$ exists.
\end{remark}

\subsubsection{Relative modularity of $\catInt$ for $q$ a root of unity}
\label{sec:relModIntROU}

We treat the case of odd roots of unity, leaving the case of even roots of unity to the reader. Let $r\geq 3$ be an odd positive integer, $\hbar=\frac{2 \pi \sqrt{-1}}{r}$ and $q=e^{\frac{2\pi \sqrt{-1}}{r}}$ . Let $\Gr= \C \slash \Z$. For each $\p \alpha \in \Gr$, let $\catInt_{\p \alpha} \subset \catInt$ be the full subcategory of modules whose $E$-weights are congruent to $\p \alpha$. Then $\catInt$ is $\Gr$-graded and generically semisimple with $\XX=\frac{1}{2}  \Z \slash \Z$. The objects $\{\ve(nr,n^{\prime} r)_{\p p}\}_{(n,n^{\prime}, \p p)\in \Zt}$ define a free realisation of $\Zt= \Z \times \Z \times \Ztwo$ in $\catInt_{\p 0}$ and $\psi: \Gr \times\Zt \rightarrow \C^{\times}$, $(\p \alpha, (n,n^{\prime}, \p s)) \mapsto q^{-2n'\p \alpha r}$, is a bicharacter satisfying equation \eqref{eq:psi}. For each $\p \alpha \in \Gr \setminus \XX$, fix a lift $\alpha_0 \in \C$ and define $\Theta(\p \alpha) = \{V(\alpha_0+i, j) _{\p 0} \mid i,j \in \{0, \dots, r-1\} \}$. Then $\Theta(\p \alpha) \otimes \sigma(\Zt)$ is a completely reduced dominating set for $\catInt_{\p \alpha}$.

\begin{theorem}\label{thm:relModIntROUOdd}
The category $\catInt$ is a modular $\C \slash \Z$-category relative to $\Zt= \Z \times \Z \times \Ztwo$ and $\XX=\frac{1}{2}  \Z \slash \Z$ with relative modularity parameter $\zeta=-r^2$. Moreover, $\catInt$ is TQFT finite.
\end{theorem}

The stabilization coefficients are $\Delta_{\pm} = \pm r$.

\section{A TQFT for arbitrary $q$}
\label{sec:TQFTArb}

Work in the setting of Section \ref{sec:relModArb}, so that $q= e^{\hbar} \in \C \setminus\{0,\pm 1\}$ and $\catq$ is given the relative modular structure of Theorem \ref{thm:relModArb}.
Fix $\D = \sqrt{-1}$ so that $\delta=-\sqrt{-1}$. The braiding of $\ZVect_{\C}$ is determined by the pairing
\[
\gamma: \Zt \times \Zt \rightarrow \{ \pm 1\},
\qquad
((b_1,\p p_1), (b_2,\p p_2)) \mapsto (-1)^{\p p_1 \p p_2}.
\]
In this section we study the TQFT $\TQFT: \Cob_{\catq} \rightarrow \ZVect_{\C}$ associated to $\catq$ by Theorem \ref{thm:relModTQFT}. Set $g_0=1$ and $V_{g_0}=V(1,0)_{\p 0}$, in the notation of Section \ref{sec:CGPTQFT}.

Let $(M,T,\coh,n) : \emptyset \rightarrow \emptyset$ be a morphism in $\Cob_{\catq}$. If $b_1(M) \geq 1$, then $\TQFT(M,T,\coh,n)$ is holomorphic in $\coh$. Indeed, in view of the definition of $F^{\prime}_{\catq}(R)$ from Section \ref{sec:CGPInvt}, it suffices to note that both $\langle T_V \rangle$ and $\qd(V(\alpha,a)_{\p p})$ are holomorphic in $\coh$; the latter follows from equation \eqref{eq:mdimTypical}.

\subsection{The state space of a generic surface of genus at least one}
\label{sec:higherGenusArb}

Let $\CS=(\Sigma, \coh)$ be a decorated connected surface without marked points. Assume that the genus of $\Sigma$ is $g \geq 1$.

\begin{lemma}
\label{lem:trivalSpine}
If $\coh \in H^1(\Sigma; \Gr)$ is not in the image of the canonical map $H^1(\Sigma; \XX) \rightarrow H^1(\Sigma; \Gr)$, then there exists a handlebody $\eta$ bounding $\Sigma$ which contains an oriented trivalent spine $\Gamma$ with the property that $\coh(m_e) \in \Gr \setminus \XX$ for all edges $e$ of $\Gamma$.
\end{lemma}

\begin{proof}
Since the small symmetric subset $\XX$ is in fact a subgroup of $\Gr$, the lemma can be proved in the same way as \cite[Proposition 6.5]{BCGP} which, moreover, proves that the spine $\Gamma$ may be taken to be that of Section \ref{sec:finDim}.
%
%
%
%
\end{proof}

Continue to denote by $\Gamma$ the spine from Section \ref{sec:finDim}. Consider the graph $\widetilde{\Gamma}$ obtained by modifying $\Gamma$ near the vertex incident to $e_1$, $e_g$ and $f_{2g-3}$:
\begin{equation*}\label{E:Gamma'Triv}
\widetilde{\Gamma} =
\qquad
\begin{tikzpicture}[anchorbase] 
\draw[thick,decoration={markings, mark=at position 0.275 with {\arrow{>}},mark=at position 0.775 with {\arrow{<}}},postaction={decorate}] (4.5,0) circle (0.35);
\draw[thick,decoration={markings, mark=at position 0.275 with {\arrow{>}},mark=at position 0.775 with {\arrow{<}}},postaction={decorate}] (6.5,0) circle (0.35);
\node at (4.5,0.6)  {\tiny $f_{2g-5}$};
\node at (5.5,0.3)  {\tiny $f_{2g-4}$};
\node at (6.5,0.6)  {\tiny $f_{2g-3}$};
\node at (4.5,-0.6)  {\tiny $e_{g-1}$};
\node at (6.5,-0.6)  {\tiny $e_g$};

\node at (7.75,1.0)  {\tiny $e^{-1}_{-1}$};
\node at (8.5,1.0)  {\tiny $e^{0}_{-1}$};
\node at (9.25,1.0)  {\tiny $e^1_{-1}$};
\node at (7.25,-0.2)  {\tiny $e_{g+1}$};
\node at (8.1,-0.3)  {\tiny $e_{g+1}^{\prime}$};
\node at (8.85,-0.25)  {\tiny $e_1^{\prime}$};
\node at (7.5,-2.0)  {\tiny $e_1$};

\draw[-,thick,decoration={markings, mark=at position 0.15 with {\arrow{<}},mark=at position 0.35 with {\arrow{<}},mark=at position 0.55 with {\arrow{<}}},postaction={decorate}] (6.85,0) to [out=0,in=90] (10.0,-0.95);
\draw[-,thick,decoration={markings, mark=at position 0.5 with {\arrow{>}}},postaction={decorate}] (6.85,-1.7) to [out=0,in=-90] (10.0,-0.95);
\draw[-,thick,decoration={markings, mark=at position 0.5 with {\arrow{<}}},postaction={decorate}] (7.75,0.75) to (7.75,0.025);
\draw[-,thick,decoration={markings, mark=at position 0.5 with {\arrow{<}}},postaction={decorate}] (8.5,0.75) to (8.5,0.05) ;
\draw[-,thick,decoration={markings, mark=at position 0.5 with {\arrow{<}}},postaction={decorate}] (9.25,-0.05) to (9.25,0.75);
\draw[-,thick,decoration={markings, mark=at position 0.5 with {\arrow{>}}},postaction={decorate}] (6.15,0) to (4.85,0);

\node at (3.5,0)  {$\cdots$};
\node at (6.5,-1.7)  {$\dots$};
\end{tikzpicture}.
\end{equation*}
Let $H=\{(0,0)_{\p 0}, (0,-1)_{\p 1}\} = \{0,\delta\}$, interpreted below as labeling the two summands on the right hand side of the decomposition \eqref{E:tensorSimpGen}.

\begin{definition}
\label{def:coloring}
A \emph{coloring of degree $k \in \Zt$ of $\widetilde{\Gamma}$} is a function $\Co: \textnormal{Edge}(\widetilde{\Gamma}) \rightarrow \catq$ assigning to each $e \in \textnormal{Edge}(\Gamma)$ a simple module $V(\alpha_e,a_e)_{\p p_e}$, where $(\alpha_e,a_e,\p p_e) \in \C \times \C \times \Ztwo$ is of degree $\coh(m_e) \in \Gr$, that is, $\alpha_e = \coh(m_e)$, and
\[
\Co(e_{-1}^1) = V_{g_0} = \Co(e_{-1}^{-1}),
\qquad
\Co(e_{-1}^0)=\sigma(k)
\]
such that the following \emph{Balancing Condition} holds: at each trivalent vertex of $\widetilde{\Gamma}$, the algebraic sum of the colors of the incident edges is in $H$.
\end{definition}

Fix a lift $\tilde{\coh} \in H^1(\Sigma; \C \times \C \times \Ztwo)$ of $\coh \in H^1(\Sigma; \Gr)$ and let
\[
\mathfrak{C}_k
=
\{\mbox{degree $k$ colorings of $\widetilde{\Gamma}$} \mid \Co(e_i) = \tilde{\coh}(m_{e_i}), \; i=1,\dots, g \}.
\]
Set $\mathfrak{C} = \bigsqcup_{k \in \Zt} \mathfrak{C}_k$. A coloring $c \in \mathfrak{C}_k$ defines a vector $v_c =(\tilde{\eta},\widetilde{\Gamma}_c,\coh,0) \in \TQFT_{-k}(\CS)$, as explained in Section \ref{sec:finDim}.

\begin{theorem}
\label{thm:genusgStateSpaceArb}
Let $\CS=(\Sigma, \coh)$ be a decorated connected surface without marked points and underlying surface of genus $g \geq 1$ such that $\coh$ is not in the image of the canonical map $H^1(\Sigma; \XX) \rightarrow H^1(\Sigma; \Gr)$. Assume that $\coh$ is such that $\coh(e_1)+g_0 \not\in \XX$. Then $\{v_c \mid c \in \mathfrak{C}_k\}$ spans $\TQFT_{-k}(\CS)$. Moreover, $\TQFT_{-k}(\CS)$ is trivial unless $k=(d, \p d)$ for some $d \in [-g,g] \cap \Z$.
\end{theorem}

\begin{proof}
By Lemma \ref{lem:trivalSpine}, there exists a handlebody $\eta$  containing a generically colored spine $\Gamma$.   Proposition \ref{prop:FDofTQFT} shows that $\TQFT_{-k}(\CS)$ is spanned by vectors $\{(\tilde{\eta}, \Gamma'_c, \coh,0)\}_{c}$, where $c$ is a finite projective $\catq$-coloring of the graph $\Gamma'$ from Section \ref{sec:finDim}. Since $\Gamma$ is generically colored, all projective indecomposables in any such coloring are in fact simple. Moreover, the coupon in $\Gamma'$ is colored by a morphism $l \in \Hom_{\catq}(V_{f_0}, V_{g_0} \otimes \sigma(k) \otimes V_{g_0}^*)$, where $V_{f_0}$ is the simple module coloring $f_0$. The associated vector is in the span of vectors of the form $\{v_c \mid c \in \mathfrak{C}_k\}$, as follows from the skein equivalences
\[
\epsh{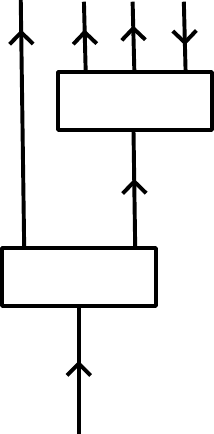}{25ex}
\put(-50,-40){\ms{e_1}}
\put(-62,65){\ms{e_{g+1}}}
\put(-15,10){\ms{f_0}}
\put(-40,65){\ms{g_0}}
\put(-30,65){\ms{\sigma(k)}}
\put(-10,65){\ms{g_0}}
\put(-38,-15){\ms{h}}
\put(-23,33){\ms{l}}
\qquad
=
\qquad
\sum_{h_1} \;\;
\epsh{oneInFourOut}{25ex}
\put(-42,-40){\ms{e_1}}
\put(-60,65){\ms{e_{g+1}}}
\put(-40,65){\ms{g_0}}
\put(-29,65){\ms{\sigma(k)}}
\put(-10,65){\ms{g_0}}
\put(-32,-15){\ms{h_1}}
\qquad
=
\qquad
\sum_{W_i} \sum_{h_1,h_2,h_3}
\qquad \;\;\;
\epsh{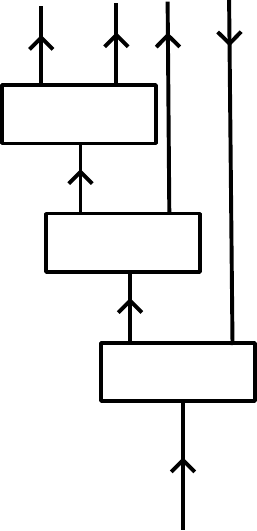}{25ex}
\put(-30,-43){\ms{e_1}}
\put(-60,65){\ms{e_{g+1}}}
\put(-40,65){\ms{g_0}}
\put(-30,65){\ms{\sigma(k)}}
\put(-10,65){\ms{g_0}}
\put(-88,20){\ms{W_i \otimes \sigma(k)^{-1}}}
\put(-45,-10){\ms{W_i}}
\put(-20,-22){\ms{h_1}}
\put(-32,6){\ms{h_2}}
\put(-42,35){\ms{h_3}}
.\]
Here we use that that the linear map
\begin{multline*}
\bigoplus_{W_i}\Hom_{\catq}(V_{f_0},  W_{i} \otimes V_{g_0}^*) \otimes \Hom_{\catq}(W_i,  (W_i\otimes\sigma(k)^{-1})\otimes \sigma(k)) \otimes \Hom_{\catq}(W_i\otimes\sigma(k)^{-1}, V_{f_{2g-2}} \otimes V_{g_0}) \\ \to
\Hom_{\catq}(V_{f_0},V_{f_{2g-2}} \otimes V_{g_0}\otimes \sigma(k) \otimes V_{g_0}^*)
\end{multline*}
given by $f\otimes g \otimes h \mapsto (h\otimes \Id_{V_{g_0}^*\otimes \sigma(k)} ) \circ (g\otimes \Id_{V_{g_0}^*}) \circ f$ is surjective, where the sum runs over all simple modules which satisfy the Balancing Condition at each trivalent vertex (coupon). Explicitly, $W_i$ and $W_i \otimes \sigma(k)^{-1}$ are the colors of $e_1^{\prime}$ and $e_{g+1}^{\prime}$, respectively.

Next, we study the set $\mathfrak{C}$. Assume for the moment that $g \geq 2$. An element of $\mathfrak{C}$ can be constructed as follows. First, for each $i=1, \dots, g$, color $e_i$ by $(\alpha_i,a_i)_{\p p_i} = \tilde{\coh}(m_{e_i})$. Next, color the edges $f_1, f_2, \dots, f_{2g-3}$ recursively, beginning with $f_1$, so as to satisfy the Balancing Conditions. In this way, the colors $\{(\beta_i,b_i)_{\p q_i}\}_i$ of $\{f_i\}_i$ are determined by a tuple $(\epsilon_1, \dots, \epsilon_{2g-3}) \in \{0,1\}^{\times 2g-3}$ through the initial conditions
\[
\begin{cases}
\beta_1=\alpha_1 - \alpha_2\\
b_1=a_1-a_2+\epsilon_1
\end{cases}
\]
and the recursive system
\[
\begin{cases}
\beta_{2i}=\beta_{2i-1}+\alpha_{i+1}\\
b_{2i} =b_{2i-1}+a_{i+1} - \epsilon_{2i}
\end{cases},
\qquad
1 \leq i \leq g-2
\]
and
\[
\begin{cases}
\beta_{2i+1}+ \alpha_{i+2} =\beta_{2i}\\
b_{2i+1}+ a_{i+2} - \epsilon_{2i+1}=b_{2i}\\
\end{cases},
\qquad
1 \leq i \leq g-2.
\]
We solve these equations to obtain
\[
\begin{cases}
\beta_{2i}=\alpha_1 \\
b_{2i}=a_1+ \sum_{j=1}^{2i} (-1)^j \epsilon_j
\end{cases},
\qquad
1 \leq i \leq g-2
\]
and
\[
\begin{cases}
\beta_{2i+1}=\alpha_1 - \alpha_{i+2} \\
b_{2i+1} = a_1 - a_{i+2}+ \sum_{j=1}^{2i+1} (-1)^j \epsilon_j
\end{cases},
\qquad
0 \leq i \leq g-2.
\]
Similar considerations give for the parities
\[
\p q_{2i}=\p p_1+ \sum_{j=1}^{2i} \p \epsilon_j,
\qquad
1 \leq i \leq g-2
\]
and
\[
\p q_{2i+1} = \p p_1 + \p p_{i+2}+ \sum_{j=1}^{2i+1} \p \epsilon_j,
\qquad
0 \leq i \leq g-2.
\]
It remains to color the edges $e_{g+1}$, $e_{g+1}^{\prime}$ and $e_1^{\prime}$. Writing $k=(c,\p r)$, these colors are subject only to the four Balancing Conditions:
\[
(\alpha_1,a_1)_{\p p_1} + g_0 = (\alpha_1^{\prime},a_1^{\prime})_{\p p_1^{\prime}} + \epsilon \delta,
\]
\[
(\alpha_{g+1}^{\prime}, a_{g+1}^{\prime})_{\p p_{g+1}^{\prime}} + (0,c)_{\p r} = (\alpha_1^{\prime},a_1^{\prime})_{\p p_1^{\prime}}
\]
\[
(\alpha_{g+1}^{\prime}, a_{g+1}^{\prime})_{\p p_{g+1}^{\prime}} + \epsilon^{\prime} \delta
=
(\alpha_{g+1}, a_{g+1})_{\p p_{g+1}} + g_0
\]
\[
(\alpha_{g+1},a_{g+1})_{\p p_{g+1}^{\prime}} = (\alpha_{g},a_{g})_{\p p_{g}^{\prime}} + (\beta_{2g-3},b_{2g-3})_{\p q_{2g-3}^{\prime}}.
\]
Here $\epsilon,\epsilon^{\prime}, \mu \in \{0,1\}$. These equations hold if and only if
\[
c=\epsilon^{\prime} - \epsilon- \mu - \sum_{j=1}^{2g-3} \epsilon_j,
\qquad
\p r = \p \epsilon^{\prime} + \p \epsilon+ \overline{\mu}+ \sum_{j=1}^{2g-3} \p \epsilon_j.
\]
and
\[
\alpha_{g+1}=\alpha_1,
\qquad
a_{g+1}=a_1 + \mu+\sum_{j=1}^{2g-3} (-1)^j \epsilon_j,
\qquad
\p p_{g+1}=\p p_1+ \overline{\mu} + \sum_{j=1}^{2g-3} \p \epsilon_j.
\]
Setting $d = \epsilon^{\prime} - \epsilon - \mu + \sum_{j=1}^{2g-2} (-1)^j\epsilon_j$, we have constructed a degree $k=(d, \p d)$ coloring of $\widetilde{\Gamma}$. It is immediate that the above construction recovers each element of $\mathfrak{C}$. In particular, $\mathfrak{C}_k$ is empty unless $k$ is of the form $(d,\p d)$ for some $d \in [-g,g] \cap \Z$.

The case $g=1$, where $e_1=f_0$, can be treated in the same way, leading to the same conclusions. We omit the details. 
\end{proof}

Denote by $\widetilde{\mathfrak{C}}_k \subset \mathfrak{C}_k$ the subset of colorings of degree $k$ with $\epsilon=\epsilon^{\prime}=0$, in the notation of the proof of Theorem \ref{thm:genusgStateSpaceArb}. Set $\widetilde{\mathfrak{C}} = \bigsqcup_{k \in \Zt} \widetilde{\mathfrak{C}}_k$.

\begin{proposition}
\label{prop:dimVanArb}
Work in the setting of Theorem \ref{thm:genusgStateSpaceArb}. For each $d \in [-(g-1),g-1] \cap \Z$, the vectors $\{v_c \mid c \in \widetilde{\mathfrak{C}}_{(d,\p d)}\}$ form a linearly independent subset of $\TQFT_{-(d,\p d)}(\CS)$ of cardinality ${2g-2 \choose g-1-\vert d \vert}$.
\end{proposition}

\begin{proof}
Let $k = (d, \p d)$. We prove linear independence by modifying the argument from \cite[Proposition 6.7]{BCGP}. Let $\p v_c \in \mathcal{V}^{\prime}(\Sigma \sqcup \hS_{-k})$ be the vector obtained from $v_c$ by reversing the orientation of $\eta \setminus \mathring{B}$ and dualizing the colors of $\Gamma^{\prime}$. The pairing $\langle \overline{v}_{c_1}, v_{c_2} \rangle$ can be computed as a connected sum of $\Theta$-graphs, as in \cite[Figure 5]{BCGP}. If $c_1 \neq c_2$, then the Balancing Conditions are not met for some of these $\Theta$-graphs and $\langle \overline{v}_{c_1}, v_{c_2} \rangle=0$. If instead $c_1=c_2$, then all Balancing Conditions are met and a direct calculation shows that the value of each $\Theta$-graph is non-zero, whence $\langle \overline{v}_{c}, v_{c} \rangle \neq 0$. This proves linear independence.

Turning to the cardinality statement, note that if $d \geq 0$, then for each $d \leq l \leq g-1$ there are ${g-1 \choose k+d} {g-1 \choose l}$ elements of $\widetilde{\mathfrak{C}}_{(d, \p d)}$ for which exactly $l+d$ of the elements $\{\epsilon_2,\epsilon_4, \dots, \epsilon_{2g-4}, \mu\}$ are equal to $1$; in such cases necessarily $l$ of the elements $\{\epsilon_1,\epsilon_3, \dots, \epsilon_{2g-3}\}$ are equal to $1$. Together with similar considerations for $d \leq 0$, this gives
\[
\vert \widetilde{\mathfrak{C}}_{(d, \p d)} \vert
=
\sum_{l=0}^{g-1-\vert d \vert} {g-1 \choose l+ \vert d \vert} {g-1 \choose l}
=
{2g-2 \choose g-1-\vert d \vert}. \qedhere
\]
\end{proof}
\subsection{Verlinde formula}
\label{sec:verlindeArb}

We compute the value of $\TQFT$ on trivial circle bundles over surfaces and relate the result to dimensions and Euler characteristics of state spaces of surfaces. This allows us to deduce a Verlinde formulate for $\TQFT$. Our approach is motivated by that of \cite[\S 6.3]{BCGP}.

Let $\CS=(\Sigma,\{p_1,\ldots, p_n\},\coh, \mathcal{L})$ be a decorated connected surface of genus $g \geq 0$. Since the Lagrangian $\mathcal{L}$ does not play a role in this section, we henceforth ignore it. For each $\beta \in \Gr$, let
\[
\CS \times S^1_{\beta}=(\Sigma \times S^1,T=\{p_1, \dots, p_n\}\times S^1,\coh \oplus \beta),
\]
where we use the canonical isomorphism $H^1((\Sigma \times S^1) \setminus T; \Gr) \simeq H^1(\Sigma \setminus \{p_i\}; \Gr) \oplus \Gr$ to extend $\coh \in H^1(\Sigma \setminus \{p_i\}; \Gr)$ to $\omega \oplus \beta$. Assume all colors of $\CS \times S^1_{\beta}$ have degree in $\Gr \setminus \XX$. Then $\TQFT(\CS \times S^1_{\beta})$ can be computed by the following surgery presentation:
\begin{equation*}
\label{diag:surgPres}
L=
\epsh{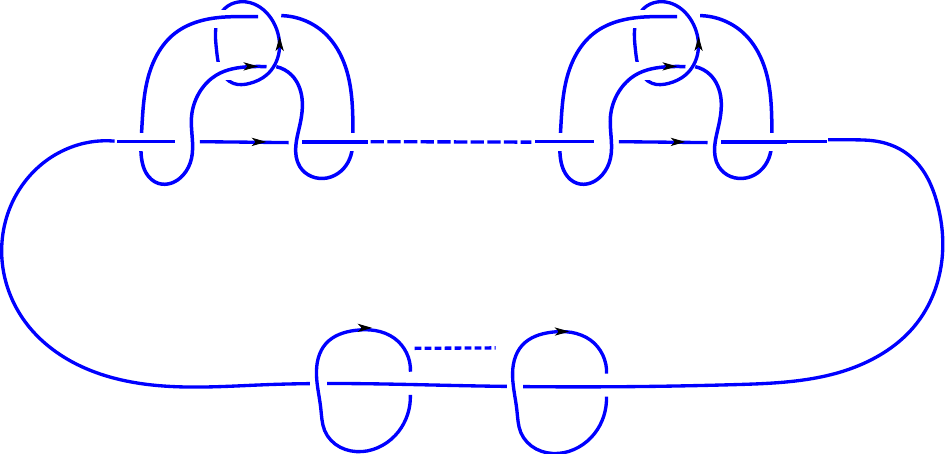}{30ex}
\put(-205,55){\ms{\Omega_{\beta_1}}}\put(-265,40){\ms{\Omega_{\alpha_1}}}\put(-74,55){\ms{\Omega_{\beta_g}}}
\put(-135,40){\ms{\Omega_{\alpha_g}}}\put(-5,-29){{\ms{\Omega_{\beta}}}}\put(-180,-25){\ms{X_1}}
\put(-120,-25){\ms{X_n}}
\;\;.
\end{equation*}
Here $X_i=V(\mu_i,m_i)_{\p q_i}$ is the color of $p_i$ and $\alpha_i = \coh([a_i]), \beta_i=\coh([b_i]) \in \Gr$ for a symplectic basis $\{[a_i],[b_i]\}_{i=1}^g$ of $H_1(\Sigma;\Z)$. By applying equation \eqref{eq:mod} first to the $\Omega_{\alpha_i}$-colored strand and then to the $\Omega_{\beta_i}$-colored strand, we can simplify the surgery presentation using the equalities
$$
\epsh{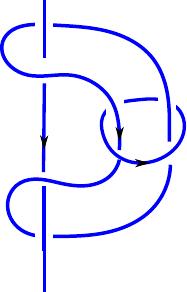}{20ex}
\put(-13,35){\ms{\Omega_{\alpha_i}}}\put(-15,-10){\ms{\Omega_{\beta_i}}}\put(-42,7){\ms{\Omega_\beta}}
=-\qd(V(\beta,0)^*_{\p 0})^{-1}\epsh{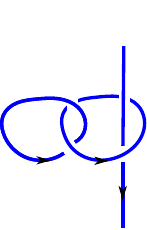}{20ex}
\put(-1,-12){\ms{\Omega_{\beta_i}}}\put(-22,-31){\ms{\Omega_{\beta}}}\put(-55,-23){\ms{V(\beta,0)_{\p 0}}}
= \qd(V(\beta,0)^*_{\p 0})^{-2}\Id_{\Omega_{\beta}}.
$$
Applying this simplification for $i=1, \dots, g$ reduces the link $L$ to an $\Omega_\beta$-colored unknot encircled by $n$ pairwise unlinked unknots colored by $X_1, \dots, X_n$. After noting that $\sigma(L)=0$ and using equation \eqref{eq:dualDimEqual}, evaluating the simplified diagram gives
\[
\TQFT(\CS \times S^1_{\beta})
=
\D^{-2g-2} \qd(V(\beta,0)_{\p 0})^{2-2g} S^{\prime}(\{X_i\},\Omega_{\beta}),
\]
where $S^{\prime}(\{X_i\},\Omega_{\beta})$ is the scalar associated to the $n$ unlinked unknots. Set
\[
\mu = \sum_{i=1}^n \mu_i,
\qquad
m= \sum_{i=1}^n m_i,
\qquad
\p q = \sum_{i=1}^n \p q_i.
\]
We use Lemma \ref{lem:Phi's} to compute
\[
S^{\prime}(\{X_i\},\Omega_{\beta}) = (-1)^{n+\p q} q^{\beta(n-2 m) + \mu} \qd(V(\beta,0)_{\p 0})^{-n}.
\]
Upon using equation \eqref{eq:mdimTypical}, this gives the final result
\begin{equation}
\label{eq:verlindeArb}
\TQFT(\CS \times S^1_{\beta})
=
(-1)^{g-1+n + \p q} q^{\beta(n-2m) + \mu}(q^{\beta} - q^{-\beta})^{2g-2 + n}.
\end{equation}
Since $\TQFT(\CS \times S^1_{\beta})$ is holomorphic in $\coh$, we conclude that equation \eqref{eq:verlindeArb} in fact holds whenever the decorated surface $\CS$ is admissible.

Define the generating function of $\Zt$-graded dimensions of $\TQFT(\CS)$ by
\[
\dim_{(t,s)} \TQFT(\CS) = \sum_{(b, \p p) \in \Zt} (-1)^{\p p} \dim_{\C} \TQFT_{(b,\p p)}(\CS) t^b s^{\p p}.
\]
The sign $(-1)^{\p p}$ reflects the braiding of $\ZVect_{\C}$.

The next result is a Verlinde formula for $\TQFT$.

\begin{theorem}
\label{thm:verlindeArb}
The equality $\TQFT(\CS \times S^1_{\beta}) = \dim_{(q^{-2 \beta},1)} \TQFT(\CS)$ holds.
\end{theorem}

\begin{proof}
Since this is proved similarly to \cite[Theorem 5.9]{BCGP}, we will be brief. Let $\{e_i\}$ be a homogeneous basis of $\TQFT(\CS)$. Write $(b_i, \p p_i) \in \Zt$ for the degree of $e_i$ and $\{e^i\}$ for the corresponding basis of $\TQFT(\overline{\CS})$. The equality
\[
\CS \times S^1_{\beta}
=
\cap_{\CS} \circ \tau_{\sqcup} \circ (\Id_{\CS}^{\beta} \sqcup \Id_{\overline{\CS}}) \circ \cup_{\CS}
\]
holds in $\Cob_{\catq}$, where $\cap_{\CS} : \CS \sqcup \overline{\CS} \rightarrow \emptyset$ and $\cup_{\CS}: \emptyset \rightarrow \overline{\CS} \sqcup \CS$ denote the cylinder on $\CS$, suitably interpreted, $\tau_{\sqcup}$ is the braiding of $\Cob_{\catq}$ and $\Id_{\CS}^{\beta}$ is as in Section \ref{sec:CGPTQFT}. Denoting by $u_0$ a basis of $\TQFT(\emptyset) \simeq \C$, we compute
\begin{eqnarray*}
\TQFT(\CS \times S^1_{\beta})u_0
&=&
\TQFT(\cap_{\CS}) \circ \TQFT(\tau_{\sqcup}) \circ (\TQFT(\Id_{\CS}^{\beta}) \otimes \Id_{\TQFT(\overline{\CS})}) \circ \TQFT(\cup_{\CS}) (u_0) \\
&=&\sum_i q^{-2\beta b_i} \TQFT(\cap_{\CS}) \circ P^{\sqcup} \circ \tau(e_i \otimes e^i) \\
&=&
\sum_i (-1)^{\p p_i} q^{-2\beta b_i} \TQFT(\cap_{\CS}) \circ P^{\sqcup}(e^i \otimes e_i) \\
&=&
\big(\sum_{(b,\p p) \in \Zt} (-1)^{\p p} q^{-2\beta b} \dim_{\C} \TQFT_{(b,\p p)}(\CS) \big)u_0,
\end{eqnarray*}
where $P^{\sqcup}$ is the monoidal coherence data for $\TQFT$. The second equality follows from the discussion in Section \ref{sec:CGPTQFT} surrounding the construction of the refined grading of $\TQFT(\CS)$.
\end{proof}

The right hand side of the equality in Theorem \ref{thm:verlindeArb} can be written more explicitly as
\[
\dim_{(q^{-2 \beta},1)} \TQFT(\CS)=
\sum_{b \in \C} \chi(\TQFT_{(b, \bullet)}(\CS)) q^{-2\beta b},
\]
where $\chi$ denotes the Euler characteristic of a $\Ztwo$-graded vector space.

\begin{example}
\label{ex:VerlindeArb}
When $n=0$ equation \eqref{eq:verlindeArb} becomes $\TQFT(\CS \times S^1_{\beta}) = (-1)^{g-1} (q^{\beta} - q^{-\beta})^{2g-2}$. If $g \geq 1$, then Theorem \ref{thm:genusgStateSpaceArb} implies
\[
\dim_{(t,s)} \TQFT(\CS) = \sum_{b=-g}^g (-1)^{\p b} \dim_{\C} \TQFT_{(b,\p b)}(\CS) t^b s^{\p b}.
\]
Applying Theorem \ref{thm:verlindeArb}, we conclude that
\[
\dim_{\C} \TQFT(\CS)
=
\dim_{(-1,1)} \TQFT(\CS)
=
\lim_{\beta \rightarrow \frac{\pi \sqrt{-1}}{2 \hbar}} \TQFT(\CS \times S^1_{\beta})
=
2^{2g-2}
\]
and
\[
\chi(\TQFT(\CS))
=
\dim_{(1,1)} \TQFT(\CS) \\
=
\lim_{\beta \rightarrow 0} \TQFT(\CS \times S^1_{\beta}) \\
=
\begin{cases}
1 & \mbox{if } g=1, \\
0 & \mbox{if } g\geq 2.
\end{cases}\qedhere
\]
\end{example}

We can now give a complete description of state spaces of generic surfaces of genus at least one.

\begin{corollary}
\label{cor:stateBasisArb}
Work in the setting of Theorem \ref{thm:genusgStateSpaceArb}.  The vectors $\{v_c \mid c \in \widetilde{\mathfrak{C}}\}$ are a basis of $\TQFT(\CS)$. In particular, $\TQFT(\CS)$ is supported in degrees $(d,\p d)$ with $d \in [-(g-1),g-1] \cap \Z$.
\end{corollary}

\begin{proof}
Using the cardinality statement in Proposition \ref{prop:dimVanArb}, we compute
\[
\vert \widetilde{\mathfrak{C}} \vert
=
\sum_{d=-(g-1)}^{g-1} \vert \widetilde{\mathfrak{C}}_{(d,\p d)} \vert
=
\sum_{d=-(g-1)}^{g-1} {2g-2 \choose g-1-\vert d \vert}
=
2^{2g-2},
\]
which is equal to $\dim_{\C} \TQFT(\CS)$ by Example \ref{ex:VerlindeArb}. The corollary now follows from the linear independence statement in Proposition \ref{prop:dimVanArb}.
\end{proof}

\begin{remark}
As mentioned in Section \ref{sec:CGPTQFT}, up to natural isomorphism, $\TQFT$ is independent of the particular choice of $g_0$ and $V_{g_0}$ used to define the decorated $2$-spheres $\hS_k$, $k \in \Zt$. In particular, by choosing a suitable $g_0$, we see that the dimension statement in Corollary \ref{cor:stateBasisArb} holds for all generic $\coh$.
\end{remark}

\subsection{The state space of the torus with non-generic cohomology class}
\label{sec:torusArb}

Let $S^1 \subset \C$ be the unit circle and $\Sigma = S^1 \times S^1$. Let $\CS=(\Sigma,\coh, \mathcal{L})$ be a decoration of $\Sigma$ without marked points. If $\coh$ is generic, then Corollary \ref{cor:stateBasisArb} applies, showing that $\TQFT(\CS)$ is supported in $\Zt$-degree zero and one dimensional. The goal of this section is to give s similarly complete description of $\TQFT(\CS)$ in the non-generic case.

\begin{proposition}
\label{prop:nonGenCohTorusArb}
Assume that there exists an oriented simple closed curve $\gamma$ in $\Sigma$ such that $\coh(\gamma) \in \XX \subset \Gr$. Then $\dim_{\C} \TQFT_0(\CS) = 2$.
\end{proposition}

\begin{proof}
The Lagrangian $\mathcal{L}$ does not affect the proof, so we ignore it. Let $B^2 \subset \C$ be the closed unit disk, $\eta=B^2\times S^1$ and $\overline{\eta}$ the orientation reversal of $\eta$. Consider $\Gamma=\{0\}\times S^1$ as a core of $\eta$.
 
Applying a diffeomorphism if necessary, we can assume that $\coh(m_{\Gamma}) = \frac{n \pi \sqrt{-1}}{\hbar} \in \XX$, $n \in \Z$. Let $\Gamma_{V,f}$ be the graph $\Gamma$ colored by $V$ with coupon $f \in \End_{\catq}(V)$. Noting that $\Gamma$ coincides with the ribbon graph $\Gamma^{\prime \prime}$ of Proposition \ref{P:VSVSS0}, Proposition \ref{prop:HomProj} implies that the vectors \begin{equation}
\label{eq:nonGenCohBasisTorusArb}
\{\mathcal{P}_{n} = (\eta,\Gamma_{P(0)_{\p 0},\Id_{P(0)_{\p 0}}}), \mathcal{P}_{n,x} = (\eta,\Gamma_{P(0)_{\p 0},x_{0,\p 0}})\},
\end{equation}
span $\state(\mathcal{S})$ where, for ease of notation, we have omitted the constant $E$-weight $\frac{n \pi \sqrt{-1}}{\hbar}$. To verify the linear independence of \eqref{eq:nonGenCohBasisTorusArb}, we pair these vectors with various elements of $\mathcal{V}^{\prime}(\CS)$.

Let $[\overline{\eta}_\emptyset]$ be the empty negatively oriented copy of $\eta$ and $[S\overline{\eta}_\emptyset]$ the vector obtained from $[\overline{\eta}_\emptyset]$ by changing the identification of $\partial \eta$ with $S^1 \times S^1$ by the map $S: \Sigma \rightarrow \Sigma$ given by $S(\theta, \theta') = (-\theta', \theta)$. Then $[S\overline{\eta}_\emptyset] \circ \mathcal{P}_n$ is $S^3$ with the core of $\eta$ as a $P(\frac{n \pi \sqrt{-1}}{\hbar},0)_{\p 0}$-colored unknot. Using Lemma \ref{lem:modTrComp}, we find
\[
\CGP_{\catq}([S\overline{\eta}_\emptyset]\circ \mathcal{P}_n)=\D^{-1}\mt_{P(\frac{n \pi \sqrt{-1}}{\hbar},0)_{\p 0}}(\Id_{P(\frac{n \pi \sqrt{-1}}{\hbar},0)_{\p 0}})=0.
\]
Replacing $\mathcal{P}_n$ with $\mathcal{P}_{n,x}$ gives
\[
\CGP_{\catq}([S\overline{\eta}_\emptyset]\circ \mathcal{P}_{n,x})
=
\D^{-1}\mt_{P(\frac{n \pi \sqrt{-1}}{\hbar},0)_{\p 0}}(x_{\frac{n \pi \sqrt{-1}}{\hbar},0,\p 0})
=
-\sqrt{-1}(q-q^{-1})^{-1}.
\]
This shows that $\mathcal{P}_{n,x}$ is non-zero with dual vector (with respect to $\langle -, - \rangle$) $\sqrt{-1}(q-q^{-1})[S\overline{\eta}_\emptyset]$.

Next, consider the pairing with $[\overline{\eta}_{-n}]$, whose core is colored by $\ve(-\frac{n \pi \sqrt{-1}}{\hbar},0)_{\p 0}$. The composition $[\overline{\eta}_{-n}] \circ \mathcal{P}_{n}$ gives $S^2\times S^1$ with a knot of the form ${p}\times S^1$ colored by
\[
P(\frac{n \pi \sqrt{-1}}{\hbar},0)_{\p 0} \otimes \ve(-\frac{n \pi \sqrt{-1}}{\hbar},0)_{\p 0} \simeq P(0,0)_{\p 0} \simeq V(1,0)_{\p 0}\otimes V(1,0)_{\p 0}^*.
\]
The decorated manifold underlying $[\overline{\eta}_{-n}]\circ \mathcal{P}_{n}$ is therefore skein equivalent to two parallel $V(1,0)_{\p 0}$-colored knots in $S^2\times S^1$ with opposite orientation. This has a surgery presentation in $S^3$ given by two parallel knots with opposite orientation colored by $V(1,0)_{\p 0}$ encircled by an $\Omega_1$-colored unknot. Using the modularity condition (equation \eqref{eq:mod}) and Lemmas \ref{lem:modTrComp} and \ref{lem:endAlgIsom}, we compute
\begin{eqnarray*}
\CGP_{\catq}([\overline{\eta}_{-n}]\circ \mathcal{P}_{n})
&=&
\D^{-2}\mt_{V(1,0)_{\p 0}\otimes V(1,0)_{\p 0}^*}(\tcoev_{V(1,0)_{\p 0}} \circ \ev_{V(1,0)_{\p 0}})(-1) \qd(V(1,0)_{\p 0})^{-1}\\
&=&
-\D^{-2} (q-q^{-1}) \qd(V(1,0)_{\p 0}) \mt_{P(0,0)_{\p 0}}(x_{0,0,\p 0}) \qd(V(1,0)_{\p 0})^{-1} \\
&=&
1.
\end{eqnarray*}
Similarly, pairing $[\overline{\eta}_{-n}]$ with $\mathcal{P}_{n,x}$ gives a surgery presentation in $S^3$ with a $P(0,0)_{\p 0}$-colored unknot with coupon $x_{0,0,\p 0}$ encircled by an $\Omega_1$-colored unknot. This gives
\[
\CGP_{\catq}([\overline{\eta}_{-n}]\circ \mathcal{P}_{n,x})
=
-\D^{-2} (q-q^{-1}) \qd(V(1,0)_{\p 0}) \mt_{P(0,0)_{\p 0}}(x^2_{0,0,\p 0}) \qd(V(1,0)_{\p 0})^{-1}
=
0,
\]
since $x^2_{0,0,\p 0}=0$. We conclude that \eqref{eq:nonGenCohBasisTorusArb} is a basis of $\state(\CS)$.

To complete the proof, recall the isomorphism $\TQFT_0(\CS) \simeq \state(\CS)$ of Proposition \ref{P:VSVSS0}.
\end{proof}

For later use, note that the proof of Proposition \ref{prop:nonGenCohTorusArb} shows that a basis dual to \eqref{eq:nonGenCohBasisTorusArb} is
\begin{equation}
\label{eq:nonGenCohDBasisTorusArb}
\{[\overline{\eta}_{-n}], \sqrt{-1} (q-q^{-1})[S\overline{\eta}_\emptyset]\}.
\end{equation}

Next, we prove the vanishing of the summands $\TQFT_k(\CS)$ for non-zero $k$. To do so, we require the following preliminary result.

\begin{lemma}
\label{lem:weightBasis}
Let $V \in \catq$ with $G$-weight decomposition $V = \bigoplus_{a \in \C}V[a]$. Then $V$ admits a $G$-weight basis $\mathcal{B}$ with the property that if $v \in Y \cdot V[a] \cap \mathcal{B}$ for some $a \in \C$, then $v =Y w$ for some $w \in V[a] \cap \mathcal{B}$.
\end{lemma}

\begin{proof}
Since $\catq$ is Krull--Schmidt, it suffices to prove that statement for $V$ indecomposable. In this case, the $G$-weight decomposition takes the form $V = \bigoplus_{j=0}^N V[a-j]$ for some $a \in \C$ and $N \in \Z_{\geq 0}$ with each $V[a-j]$ non-zero. Choose a basis $\mathcal{B}^a=\{v_{1}^a, \dots, v_{d_a}^a\}$ of $V[a]$. Then $Y \cdot \mathcal{B}^a = \{Yv_{1}^a, \dots, Yv_{d_a}^a\}$ spans the subspace $Y\cdot V[a] \subset V[a-1]$. Let $\mathcal{B}^{Y,a} \subset Y \cdot \mathcal{B}^a$ be a basis of $Y\cdot V[a]$ and extend it to a basis $\mathcal{B}^{a-1}$ of $V[a-1]$. In the same way, we construct from $\mathcal{B}^{a-1}$ a basis of $\mathcal{B}^{a-2}$ of $V[a-2]$ which extends a basis $\mathcal{B}^{Y, a-1} \subset Y \cdot \mathcal{B}^{a-1}$ of $Y \cdot V[a-1]$. Repeating, we obtain a basis $\mathcal{B}=\bigsqcup_{j=0}^{N} \mathcal{B}^{a-j}$ of $V$ with the desired properties.
\end{proof}

\begin{prop}
\label{prop:nonGenCohTorusHighDegArb}
In the setting of Proposition \ref{prop:nonGenCohTorusArb}, we have $\TQFT_k(\CS) = 0$ if $k \neq 0$.
\end{prop}

\begin{proof}
By Proposition \ref{prop:FDofTQFT}, the state space $\TQFT_{-k}(\CS)$ is spanned by finite projective $\catq$-colorings of $\Gamma^{\prime}$. Then $e_1$ is colored by $P(0,0)_{\p 0}$ and the coupon by a composition
\[
P(0,0)_{\p 0}^{\otimes 2} \simeq P(0,1)_{\p 1} \oplus P(0,0)_{\p 0}^{\oplus 2} \oplus P(0,-1)_{\p 1}
\rightarrow
P
\rightarrow 
V_{g_0} \otimes \sigma(k) \otimes V_{g_0}^* \simeq P(0,b)_{\p p},
\]
where $P$ is a projective indecomposable of degree $0$ and $k = (b,\p p)$. By Proposition \ref{prop:HomProj}, this morphism is zero unless $(b,\p p) \in \{(0,\p 0), (\pm 1,\p 1), (\pm 2,\p 0)\}$. The case $k=(0, \p 0)$ is dealt with in Proposition \ref{prop:nonGenCohTorusArb}. When $k=(\pm 2,\p 0)$, the state space is spanned by vectors obtained by coloring $\Gamma^{\prime}$ so that its coupon is colored by an element of $\Hom_{\catq}( P(0,0)_{\p 0}^{\otimes 2}, V_{g_0}\otimes \sigma(\pm 2,\p 0) \otimes V_{g_0}^*)$. When such a vector is paired with any vector, the obtained closed $3$-manifold has a surgery presentation in which the edges leaving the coupon are encircled by a surgery component. Thus, this surgery presentation is skien equivalent to a graph containing a coupon filled with a morphism in $\Hom_{\catq}(P(0,0)_{\p 0}^{\otimes 2},\sigma(\pm 2,\p 0))$, which is the zero space. Thus, $\TQFT_{(\pm2,\p 0)}(\CS)$ is zero.

Suppose then that $k=(1,\p 1)$; the case $k=(-1,\p 1)$ is analogous. The state space $\TQFT_{(1,\p 1)}(\CS)$ is spanned by the vector $v$ obtained by coloring $\Gamma^{\prime}$ with $P(0,0)_{\p 0}$. We claim that $v \in \TQFT(\CS)$ is the zero vector. By part \eqref{ite:zero} of Lemma \ref{lem:PropOfCGP}, it suffices to show that $\langle w, v \rangle=0$ for all $w \in \mathcal{V}^{\prime}(\CS)$. By isotopy and skein equivalence, the pairing $\langle w, v \rangle$ can be computed as a sum of evaluations of diagrams of the form
\[
\epsh{origDiag}{15ex}
\put(-33,-17){\ms{\psi}}
\put(-33,22){\ms{\phi}}
\put(-65,20){\ms{V_1}}
\put(-47,2){\ms{V_l}}
\put(-75,-42){\ms{P(0,0)_{\p 0}}}
\put(-65,-11){\tiny \vdots}
\put(5,0){\ms{\ve(0,1)_{\p 1}}}
\]
for indecomposables $V_1, \dots, V_l \in \catq$ of $E$-weight $0$ and morphisms $\phi$, $\psi$. Note here the coupon labeled with $\phi$ receives contributions from both the vector $w \in \mathcal{V}^{\prime}(\CS)$ and possible surgery components coming from the pairing. Also, as in the previous step, the edge colored with $\ve(0,1)_{\p 1}$ is obtained from a skein relation removing a coupon with edges colored with $V_{g_0}, \ve(0,1)_{\p 1}$ and $V_{g_0}^*$ which in the pairing is encircled  by a surgery component and thus can be simplified. 

  To prove that $\langle w, v \rangle=0$, it suffices to prove that, for any indecomposable $V \in \catq$ of $E$-weight $0$ and $\phi \in \Hom_{\catq}(\ve(0,1)_{\p 1}, V^* \otimes V)$, the sub-diagram 
\begin{equation}
\label{eq:finalDiag}
\epsh{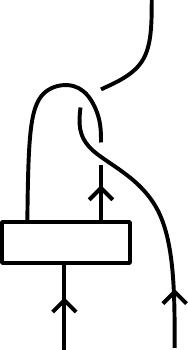}{20ex}
\put(-35,-17){\ms{\phi}}
\put(-33,-5){\ms{V}}
\put(5,-35){\ms{P}}
\put(-65,-35){\ms{\ve(0,1)_{\p 1}}}
\end{equation}
evaluates to zero. While we require this statement only for $P=P(0,0)_{\p 0}$, we prove that it holds for all $P$.

Let $V \in \catq$ be indecomposable with $E$-weight $0$ and $G$-weight space decomposition $V = \bigoplus_{i=0}^N V[a-i]$. The weight space $V[a]$ is homogeneous, say of degree $\p s \in \Ztwo$. Then $V[a-i]$ is homogeneous of degree $\p s + \p i$. Fix a basis $\mathcal{B} = \bigsqcup_{j=0}^N \mathcal{B}^{a-j}=\bigsqcup_{j=0}^N \bigsqcup_i \{v_i^{a-j}\}$ of $V$ of the form guaranteed by Lemma \ref{lem:weightBasis} and its proof. Here and below we have written the $G$-weight of a vector as a superscript.

The image of a basis vector of $\ve(0,1)_{\p 1}$ under $\phi$ takes the form
\begin{equation}
\label{eq:imageBasisve}
\sum_{j=0}^N\sum_{i \in \mathcal{B}^{a-j}} f_i^{-a+j+1} \otimes v_i^{a-j} \in V^* \otimes V
\end{equation}
for some $f_i^{-a+j+1} \in V^*$. The restrictions on weights comes from the fact that $\ve(0,1)_{\p 1}$ is concentrated in $G$-degree $1$. Similar considerations for parity give $\vert f_i^{-a+j+1} \vert = \p s + \p j + \p 1$. With this notation, the portion of diagram \eqref{eq:finalDiag}
below the braidings maps a homogeneous vector $p \in P$ to
\[
\Big(\sum_{j=0}^N\sum_{i \in \mathcal{B}^{a-j}} f_i^{-a+j+1} \otimes v_i^{a-j}\Big) \otimes p.
\]
In view of the eventual application of $\tev_V$, weight considerations imply that the only terms of the double (inverse) braiding of $V$ with $P$ which can contribute to diagram \eqref{eq:finalDiag} arise from the identity in the lower braiding and the term $-(q-q^{-1})(X \otimes Y)(K \otimes K^{-1})$ in the upper braiding. It follows that
diagram \eqref{eq:finalDiag} maps $p \in P$ to
\begin{equation}
\label{eq:diagVal}
-(q-q^{-1}) \left(\sum_{j=0}^N (-1)^{\p s + \p j} \tev_V \Big( \sum_{i \in \mathcal{B}^{a-j}} f_i^{-a+j+1} \otimes Yv_i^{a-j} \Big) \right) \otimes Xp
\end{equation}
which, since $Y^2=0$, is equal to
\begin{equation*}
\label{eq:diagValSimp}
-(q-q^{-1}) \sum_{j=0}^N (-1)^{\p s + \p j} \tev_V \Big( \sum_{i \in \mathcal{B}^{a-j} \setminus \mathcal{B}^{Y,a-j+1}} f_i^{-a+j+1} \otimes Yv_i^{a-j} \Big) \otimes Xp.
\end{equation*}
Fix $j \in \{0, \dots, N\}$. A direct computation shows
\[
\tev_V \Big( \sum_{i \in \mathcal{B}^{a-j} \setminus \mathcal{B}^{Y,a-j+1}} f_i^{-a+j+1} \otimes Yv_i^{a-j} \Big)
=
(-1)^{\p s + \p j } \tev_V \Big( \sum_{i \in \mathcal{B}^{a-j} \setminus \mathcal{B}^{Y,a-j+1}} Yf_i^{-a+j+1} \otimes v_i^{a-j} \Big).
\]
We claim that the vector $\sum_{i \in \mathcal{B}^{a-j} \setminus \mathcal{B}^{Y,a-j+1}} Yf_i^{-a+j+1} \otimes v_i^{a-j}$ is in fact zero. This will imply that the previous evaluations, and hence \eqref{eq:diagVal}, vanish, thereby completing the proof. To prove the claim, note that since $Y$ annihilates $\ve(0,1)_{\p 1}$, it annihilates the vector \eqref{eq:imageBasisve}. This gives the relation
\[
0=\sum_{j=0}^N\sum_{i \in \mathcal{B}^{a-j}} \Big( Yf_i^{-a+j+1} \otimes v_i^{a-j} + (-1)^{\p r + \p j + \p 1} f_i^{-a+j+1} \otimes Yv_i^{a-j} \Big).
\]
Projecting this relation to the subspace $V^* \otimes V[a-j]$ gives
\begin{equation}
\label{eq:vanish}
0= \sum_{i \in \mathcal{B}^{a-j}} Yf_i^{-a+j+1} \otimes v_i^{a-j} + \sum_{i \in \mathcal{B}^{a-j+1}} (-1)^{\p s + \p j} f_i^{-a+j-1+1} \otimes Yv_i^{a-j+1}
\end{equation}
where we interpret each $v_i^{a+1}$ as the zero vector. The decomposition
\[
\mathcal{B}^{a-j} = \mathcal{B}^{Y,a-j+1} \sqcup (\mathcal{B}^{a-j} \setminus \mathcal{B}^{Y,a-j+1})
\]
established in the proof of Lemma \ref{lem:weightBasis} allows equation \eqref{eq:vanish} to be rewritten as
\begin{multline*}
0= \sum_{i \in \mathcal{B}^{a-j} \setminus \mathcal{B}^{Y,a-j+1}} Yf_i^{-a+j+1} \otimes v_i^{a-j} + \\ \sum_{i \in \mathcal{B}^{Y,a-j+1}} Yf_i^{-a+j+1} \otimes v_i^{a-j} + \sum_{i \in \mathcal{B}^{a-j+1}} (-1)^{\p s + \p j} f_i^{-a+j-1+1} \otimes Yv_i^{a-j+1}.
\end{multline*}
The definition of the basis $\mathcal{B}^{a-j}$ implies that the first sum and the sum of the second two sums vanish independently; the former is the claimed vanishing.
\end{proof}

\begin{remark}
Proposition \ref{prop:nonGenCohTorusHighDegArb} is surprising in view of a conjecture for the TQFTs constructed from $\Uqsl$ at a root of unity \cite[Conjecture 6.4]{BCGP}, which states that non-generic tori state spaces are not concentrated in degree zero. We expect that the method of proof of Proposition \ref{prop:nonGenCohTorusHighDegArb} can be applied in the setting of $\Uqsl$ to disprove this conjecture.
\end{remark}

Next, we compute the action of the mapping class group $SL(2,\Z)$ of $\Sigma$ on $\TQFT_0(\CS)$. Restrict attention to the case $\coh=0$, so that $SL(2,\Z)$ fixes $\coh$. Use the basis \eqref{eq:nonGenCohBasisTorusArb} of $\TQFT_0(\CS)$ and its dual basis \eqref{eq:nonGenCohDBasisTorusArb} and drop $n$ from the notation, since it is zero.

In general, let $\CS=(\Sigma,\omega=0, \mathcal{L})$ be a decorated surface without marked points and trivial cohomology class and $f: \Sigma \rightarrow \Sigma$ an orientation preserving diffeomorphism. Consider the endomorphism
\[
N_f = \state(L_{f_* \mathcal{L},\mathcal{L}} \circ M_f) : \state(\CS) \rightarrow \state(\CS),
\]
where $M_f : (\Sigma, \omega, \mathcal{L}) \rightarrow (\Sigma, \omega, f_*\mathcal{L})$ is the mapping cylinder on $f$ and $L_{f_* \mathcal{L},\mathcal{L}} : (\Sigma, \omega, f_*\mathcal{L}) \rightarrow (\Sigma, \omega, \mathcal{L})$ is the decorated cobordism $(\Sigma \times [0,1], \pi_{\Sigma}^* \coh, 0)$, as defined in \cite[\S 3.3.1]{BCGP}. The assignment $f \mapsto N_f$ defines a projective representation of the mapping class group of $\Sigma$ on $\state(\CS)$:
\[
N_f \circ N_g = \delta^{-\mu(f_* \mathcal{L}, \mathcal{L}, g^{-1}_* \mathcal{L})} N_{f \circ g}.
\]
Here $\mu$ denotes the Maslov index.

Retuning to the case of the torus, fix an oriented meridian $m$ and oriented longitude $l$ of $\Sigma$ with intersection number $m \cdot l = 1$. The elements $T$ and $S$ given by Dehn twist along $m$ and the map which sends $m$ to $l$ and $l$ to $-m$, respectively, generate the mapping class group of $\Sigma$. Take for the Lagrangian $\mathcal{L} = \R \cdot [m]$. With this choice, we have $\mu(S_* \mathcal{L}, \mathcal{L}, \mathcal{L})=0$ and $\mu(T_* \mathcal{L}, \mathcal{L}, \mathcal{L})=0$, as in the proof of \cite[Theorem 6.28]{BCGP}. It follows that $N_S= \state(M_S)$ and $N_T=\state(M_T)$ when acting on either $\mathcal{P}_{n,x}$ or $\mathcal{P}_{n,x}$. In other words, there are no corrections by the Maslov index.

Following the computations in the proof of Proposition \ref{prop:nonGenCohTorusArb}, the composition $[\overline{\eta}_{\emptyset}] \circ N_S\mathcal{P}_{x}$ gives an $S^3$ containing the core of $\eta$ as a $P(0,0)_{\p 0}$-colored unknot with a coupon labeled by $x_{0,\p 0}$. Using this, we compute
\[
\CGP_{\catq}([\overline{\eta}_{\emptyset}]\circ N_S\mathcal{P}_{x})=\D^{-1}(q-q^{-1})^{-1}.
\]
Similarly, we compute $\CGP_{\catq}([S\overline{\eta}_\emptyset]\circ N_S\mathcal{P}_{x})=0$. It follows that $[\overline{\eta}_{\emptyset}]$ is dual to $\D(q-q^{-1})N_S\mathcal{P}_{x}$. However, from the proof of Proposition \ref{prop:nonGenCohTorusArb}, $\mathcal{P}$ is dual to $[\overline{\eta}_{\emptyset}]$,
whence $N_S\mathcal{P}_{x}=\D^{-1}(q-q^{-1})^{-1}\mathcal{P}$.
Similar computations give $\CGP_{\catq}([\overline{\eta}_{\emptyset}]\circ N_S\mathcal{P})=0$ and
\[
\CGP_{\catq}([S\overline{\eta}_\emptyset]\circ N_S\mathcal{P})=\D^{-2}\mt_{P(0)_{\p 0}}(x_{0,\p 0})(-1)(\qd(V(1,0))_{\p 0}))^{-1}=1.
\]
Hence, $N_S \mathcal{P}$ is dual to $[S\overline{\eta}_\emptyset]$. Since $\sqrt{-1}(q-q^{-1})\mathcal{P}_{x}$ is also dual to $[S\overline{\eta}_\emptyset]$, we have $N_S\mathcal{P}=\sqrt{-1}(q-q^{-1})\mathcal{P}_{x}$. Continuing, a direct calculation gives $\theta_{P(0)_{\p 0}}=\Id_{P(0)_{\p 0}} + (q-q^{-1})x_{0,\p 0}$. It follows that $N_T\mathcal{P}_{x}=\mathcal{P}_{x}$ and $N_T\mathcal{P}=\mathcal{P}+ (q-q^{-1})\mathcal{P}_{x}$.

\begin{theorem}
\label{thm:MCGArb}
In the basis $\{\mathcal{P}_{x},\mathcal{P}\}$ of $\TQFT_0(\CS)$, the mapping class group action is given by
\[
S=\left(\begin{array}{cc}0 & \sqrt{-1}(q-q^{-1}) \\ -\sqrt{-1}(q-q^{-1})^{-1} & 0\end{array}\right) ,
\qquad
T=\left(\begin{array}{cc}1 & q-q^{-1} \\0 & 1\end{array}\right).
\]
Moreover, this representation is faithful modulo its center.
\end{theorem}

\begin{proof}
We compute $(ST)^3 = - \sqrt{-1} \left( \begin{smallmatrix} 1 & 0 \\ 0 & 1 \end{smallmatrix} \right)$ and $S^2 = \left( \begin{smallmatrix} 1 & 0 \\ 0 & 1 \end{smallmatrix} \right)$, so that we have a projective representation of $SL(2,\Z)$ on $\TQFT_0(\CS)$.
It follows that $\tilde{S}= -\sqrt{-1} S$ and $T$ generate a (non-projective) representation of $SL(2,\Z)$. Setting $A= \left(\begin{smallmatrix} 1 & 0 \\ 0 & q-q^{-1} \end{smallmatrix}\right)$, we have $A \tilde{S} A^{-1} = \left( \begin{smallmatrix} 0 & -1 \\ 1 & 0 \end{smallmatrix} \right)$ and $A T A^{-1} = \left( \begin{smallmatrix} 1 & 1 \\ 0 & 1 \end{smallmatrix} \right)$, that is, $A$ defines an isomorphism from the representation determined by $\tilde{S}$ and $T$ to the fundamental representation, which is faithful.
\end{proof}

\subsection{The $4$-punctured sphere and the Alexander polynomial}
\label{sec:alexPoly}

In this section we briefly explain a connection between $\TQFT_{\catq}$ and the Alexander polynomial. Incarnations of the Alexander polynomial in quantum topology are well-known and include Reshetikhin--Turaev-type constructions from quantizations of $\mathfrak{sl}(n \vert n)$, $n \geq 1$, \cite{kauffman1991} and $\gloo$ \cite{reshetikhin1992,Viro,sartori2015}. The point we wish to emphasize is that the Alexander polynomial emerges from the TQFT $\TQFT_{\catq}$ in much the same way as the Jones polynomial emerges from $SU(2)$ Chern--Simons theory \cite[\S 4]{witten1989}.

Fix $(\alpha,a) \in (\C \setminus \frac{\pi \sqrt{-1}}{\hbar} \Z) \times \C$. Let $\CS$ be the decorated $2$-sphere with four $V(\alpha,a)_{\p 0}$-colored points, exactly two of which are positively oriented. There is a unique compatible cohomology class $\coh$.

\begin{proposition}
The vector space $\state(\CS) \simeq \TQFT_0(\CS)$ is two dimensional.
\end{proposition}

\begin{proof}
Set $V = V(\alpha,a)_{\p 0}$. By Proposition \ref{P:VSVSS0}, the vector space $\state(\CS)$ is spanned by $\coh$-compatible $\catq$-colorings of the ribbon graph $\Gamma^{\prime \prime}$ in $B^3$ consisting of a single coupon with four legs. The coupon is therefore colored by an element of
\[
\Hom_{\catq}(\mathbb{C}, V \otimes V \otimes V^* \otimes V^*) \simeq \End_{\catq}(P(0,0)_{\p 0}) \simeq \C[x_{0,0,\p 0}] \slash \langle x_{0,0,\p 0}^2 \rangle,
\]
the first isomorphism following from equation \eqref{E:tensorSimpNongen} and the second from Proposition \ref{prop:HomProj}. Hence, the dimension of $\state(\CS)$ is at most two.

Given $\epsilon \in \{\pm 1\}$, denote by $\mathcal{M}_{\epsilon} = (B^3, \Gamma_{\epsilon}, \coh, 0) \in \state(\CS)$ and $\mathcal{M}_{\epsilon}^{\prime} = (S^3 \setminus \mathring{B}^3, \Theta_{\epsilon}, \coh, 0) \in \mathcal{V}^{\prime}(\CS)$ the vectors depicted by
\[
\mathcal{M}_{\epsilon}
=
\epsh{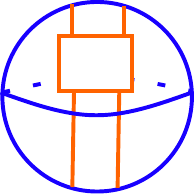}{12ex}
\put(-36,12){\ms{c_{V,V}^{\epsilon}}}
\qquad
\mbox{and}
\qquad
\mathcal{M}^{\prime}_{\epsilon}
=
\epsh{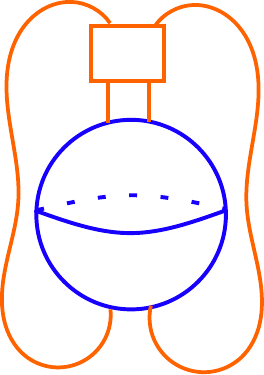}{20ex}
\put(-42,35){\ms{c_{V,V}^{\epsilon}}}.
\]
Linear independence of the vectors $\{\mathcal{M}_+, \mathcal{M}_-\}$ follows by computing
\[
\langle \mathcal{M}_+^{\prime}, \mathcal{M}_+\rangle \neq 0,
\qquad
\langle \mathcal{M}_-^{\prime}, \mathcal{M}_+\rangle = 0,
\qquad
\langle \mathcal{M}_+^{\prime}, \mathcal{M}_-\rangle = 0,
\qquad
\langle \mathcal{M}_-^{\prime}, \mathcal{M}_-\rangle \neq 0.
\]
The vanishings follow from the observation that the CGP invariant of the disjoint union of two $V$-colored unknots in $S^3$ is zero. The non-vanishings follow from the fact that the CGP invariant of a $V$-colored Hopf link in $S^3$ is non-zero; see Lemma \ref{lem:Phi's}.
\end{proof}

Let $\{v, v^{\prime}\}$ be the weight basis of $V(\alpha,a)_{\p 0}$ from Section \ref{sec:UqSimples}. In the basis
\[
\{v \otimes v, v \otimes v^{\prime}, v^{\prime} \otimes v, v^{\prime} \otimes v^{\prime}\} \subset V(\alpha,a)_{\p 0} \otimes V(\alpha,a)_{\p 0},
\]
we compute
\[
c_{V(\alpha,a)_{\p 0},V(\alpha,a)_{\p 0}}
=
\left(
\begin{matrix}
q^{-2 \alpha a} & 0 & 0 & 0 \\
0 & 0 & q^{-2\alpha a + \alpha} & 0 \\
0 & q^{-2\alpha a + \alpha} & q^{-2\alpha a + \alpha}(q^{\alpha}-q^{-\alpha}) & 0 \\
0 & 0 & 0 & -q^{-2\alpha a+2\alpha}
\end{matrix}
\right).
\]
A direct calculation then shows that
\[
q^{2 \alpha a-\alpha}c_{V(\alpha,a)_{\p 0},V(\alpha,a)_{\p 0}}- q^{-(2 \alpha a-\alpha)}c_{V(\alpha,a)_{\p 0},V(\alpha,a)_{\p 0}}^{-1} = (q^{\alpha}-q^{-\alpha}) \Id_{V(\alpha,a)_{\p 0} \otimes V(\alpha,a)_{\p 0}}.
\]
Writing $\mathcal{M}_{\Id} \in \state(\CS)$ for the vector obtained by coloring $\Gamma^{\prime \prime}$ with $\Id_{V(\alpha,a)_{\p 0} \otimes V(\alpha,a)_{\p 0}}$, the previous equation becomes the linear relation
\[
q^{2 \alpha a-\alpha}\mathcal{M}_+ + q^{-(2 \alpha a-\alpha)}\mathcal{M}_- = (q^{\alpha}-q^{-\alpha}) \mathcal{M}_{\Id}
\]
in $\state(\CS)$. It follows from this equation that, upon renormalization, $\TQFT_{\catq}$ defines an invariant of framed oriented links in $S^3$ that satisfies the Conway skein relation, as stated in \cite[Eqn. (24)]{Viro}, specialized to $t^2=q^{\alpha}$. More precisely, recall from the proof of Theorem \ref{thm:ribbonCat} that $\theta_{V(\alpha,a)_{\p 0}} = q^{-2\alpha a+\alpha} \Id_{V(\alpha,a)_{\p 0}}$. Then $\mathcal{F}^{\prime}_{\catq}(L)=(q^{2\alpha a-\alpha})^{\wre(L)} F^{\prime}_{\catq}(L)$, where $\wre(L)$ is the writhe of $L$, is an invariant of $V(\alpha,a)_{\p 0}$-colored oriented links which satisfies the Conway skein relation.

In \cite[\S 6.7]{BCGP}, the TQFT constructed from the relative modular category of weight $\Uqsl$-modules for $q$ a primitive fourth root of unity was shown to encode the multivariable Alexander polynomial by arguments similar to those above. Moreover, it was shown that this TQFT evaluates to the canonically normalized Reidemeister torsion on closed $3$-manifolds. We expect these arguments to generalize to the present setting without significant change, thereby connecting $\TQFT_{\catq}$ with Reidemeister torsion. We leave a detailed verification to future work.

\section{A TQFT for root of unity $q$}
\label{sec:TQFTROU}

We work in the setting of Section \ref{sec:relModROU}. Fix a positive integer $r \geq 2$. Set $\hbar=\frac{2 \pi \sqrt{-1}}{r}$ if $r$ is odd and $\hbar=\frac{\pi \sqrt{-1}}{r}$ if $r$ is even. Consider $\catq$ with the relative modular structures of Theorems \ref{thm:relModROUOdd} and \ref{thm:relModROUEven}, so that
\[
\Gr = 
\begin{cases}
\C \slash \Z \times \C \slash \Z & \mbox{if $r$ is odd,} \\
\C \slash 2\Z \times \C \slash \Z & \mbox{if $r$ is even,}
\end{cases}
\qquad
\Zt = \Z \times \Z \times \Ztwo,
\qquad
\XX =
\begin{cases}
\{(\p 0, \p 0), (\p {\frac{1}{2}}, \p 0)\} & \mbox{if $r$ is odd,} \\
\{(\p 0, \p 0)\} & \mbox{if $r$ is even}
\end{cases}
\]
and $\Delta_{\pm}=\pm r$. Fix $\D = r \sqrt{-1}$ so that $\delta=-\sqrt{-1}$. The braiding of $\ZVect_{\C}$ is determined by
\[
\gamma: \Zt \times \Zt \rightarrow \{ \pm 1\},
\qquad
((n_1, n^{\prime}_1,\p p_1), (n_2,n^{\prime}_2,\p p_2)) \mapsto (-1)^{\p p_1 \p p_2}.
\]
In this section we study the TQFT $\TQFT : \Cob_{\catq} \rightarrow \ZVect_{\C}$ associated to $\catq$ by Theorem \ref{thm:relModTQFT}. In the notation of Section \ref{sec:CGPTQFT}, we fix
\[
g_0 =
\begin{cases}
(\p {\frac{1}{2}}, \p 0) & \mbox{if } r=2 \mbox{ or $r$ is odd}, \\
(\p 0, \p 0) & \mbox{if $r \geq 4$ is even,}
\end{cases}
\qquad
V_{g_0} =
\begin{cases}
V(1,0)_{\p 0} & \mbox{if } r=2 \mbox{ or $r$ is odd}, \\
V(2,0)_{\p 0} & \mbox{if $r \geq 4$ is even.}
\end{cases}
\]
As in Section \ref{sec:TQFTArb}, the invariant $\TQFT(\mathcal{M})$ of a morphism $\mathcal{M} : \emptyset \rightarrow \emptyset$ in $\Cob_{\catq}$ depends holomorphically on the cohomology class $\coh$ appearing in the definition of $\mathcal{M}$.

\subsection{The state space of a generic surface of genus at least one}
\label{sec:higherGenusROU}

We modify the arguments of Section \ref{sec:higherGenusArb} to construct a spanning set of the state space of a surface of genus at least one.

Let $\CS=(\Sigma, \coh)$ be a decorated connected surface without marked points and underlying surface $\Sigma$ of genus $g \geq 1$. Since $\XX \subset \Gr$ is again a subgroup, Lemma \ref{lem:trivalSpine} continues to hold. Definition \ref{def:coloring} applies in the present setting, where now $\Gr$ and $\sigma$ are as in the root of unity case.  Fix a lift $\tilde{\coh} \in H^1(\Sigma; \C \times \C \times \Ztwo)$ of $\coh \in H^1(\Sigma; \Gr)$ and let
\[
\mathfrak{C}^{(r)}_k
=
\{\mbox{degree $k$ colorings of } \widetilde{\Gamma} \mid \Co(e_i) \in \tilde{\coh}(m_{e_i}) + \{(j,k)_{\p 0}\}_{0 \leq j,k \leq r-1}, \; i=1,\dots, g \}
\]
and $\mathfrak{C}^{(r)} = \bigsqcup_{k \in \Zt} \mathfrak{C}^{(r)}_k$.

\begin{theorem}
\label{thm:genusgStateSpaceROU}
Let $\CS=(\Sigma, \coh)$ be a decorated connected surface without marked points and underlying surface $\Sigma$ of genus $g \geq 1$. If $r=2$, assume that $\tilde{\coh}(e_1)+g_0 \not\in \XX$. The set $\{v_c \mid c \in \mathfrak{C}^{(r)}_k\}$ spans $\TQFT_{-k}(\CS)$. Moreover, $\TQFT_{-k}(\CS)$ is trivial unless $k=(0,d, \p d)$ for some $d \in [-g,g] \cap r\Z$.
\end{theorem}

\begin{proof}
The proof is a variation of that of Theorem \ref{thm:genusgStateSpaceArb}. Note that if $r \neq 2$, then $g_0 =0$ and the condition $\tilde{\coh}(e_1)+g_0 \not\in \XX$ holds automatically for all generic $\coh$. The main difference from Theorem \ref{thm:genusgStateSpaceArb} is that there are now $r^2$ choices for the color of each edge $e_i$, $i=1, \dots, g$. The colors of the edges $\{f_i\}$ and $e_{g+1}$ of $\widetilde{\Gamma}$ are determined by a tuple $(\epsilon_1, \dots, \epsilon_{2g-3},\mu) \in \{0,1\}^{\times 2g-2}$ through the same recursive system of equations as in the proof of Theorem \ref{thm:genusgStateSpaceArb}. Writing $k=(nr,n^{\prime}r,\p t)$, the colors of the remaining the edges $e_1^{\prime}$ and $e_{g+1}^{\prime}$ are determined by the equations
\[
(\alpha_1,a_1)_{\p p_1} + g_0 = (\alpha_1^{\prime},a_1^{\prime})_{\p p_1^{\prime}} + \epsilon \delta,
\]
\[
(\alpha_{g+1}^{\prime}, a_{g+1}^{\prime})_{\p p_{g+1}^{\prime}} + (nr,n^{\prime}r)_{\p t} = (\alpha_1^{\prime},a_1^{\prime})_{\p p_1^{\prime}}
\]
\[
(\alpha_{g+1}^{\prime}, a_{g+1}^{\prime})_{\p p_{g+1}^{\prime}} + \epsilon^{\prime} \delta
=
(\alpha_{g+1},a_{g+1})_{\p p_{g+1}} + g_0.
\]
Here $\epsilon,\epsilon^{\prime} \in \{0,1\}$. These equations hold if and only if
\[
nr=0,
\qquad
n^{\prime}r=\epsilon - \epsilon^{\prime}- \mu - \sum_{j=1}^{2g-3} (-1)^j \epsilon_j,
\qquad
\p t = \p \epsilon + \p \epsilon^{\prime} + \p \mu + \sum_{j=1}^{2g-3} \p \epsilon_j.
\]
The first equation implies that $n=0$ while the second implies that $r$ divides $\epsilon^{\prime} - \epsilon- \mu - \sum_{j=1}^{2g-3} \epsilon_j$.
\end{proof}

Denote by $\widetilde{\mathfrak{C}}^{(r)}_k \subset \mathfrak{C}^{(r)}_k$ the subset of colorings of degree $k$ with $\epsilon=\epsilon^{\prime}=0$, in the notation of the proof of Theorem \ref{thm:genusgStateSpaceROU}. Set $\widetilde{\mathfrak{C}}^{(r)} = \bigsqcup_{k \in \Zt} \widetilde{\mathfrak{C}}^{(r)}_k$.

\begin{proposition}
\label{prop:dimVanROU}
Work in the setting of Theorem \ref{thm:genusgStateSpaceROU}. For each $d \in [-(g-1),g-1] \cap r\Z$, the vectors $\{v_c \mid c \in \widetilde{\mathfrak{C}}^{(r)}_{(0,d,\p d)}\}$ form a linearly independent subset of $\TQFT_{-(0,d,\p d)}(\CS)$ of cardinality $r^{2g}{2g-2 \choose g-1-\vert d \vert}$.
\end{proposition}

\begin{proof}
This is proved in the same way as Proposition \ref{prop:dimVanArb}.
\end{proof}

\subsection{A Verlinde formula}
\label{sec:verlindeROU}

Consider now the analogue of Section \ref{sec:verlindeArb} for $q$ a root of unity.

Let $\CS=(\Sigma,\{p_1,\ldots, p_n\},\coh)$ be a decorated connected surface of genus $g \geq 0$. Let $(\p \beta, \p b) \in \Gr$ and consider $\CS \times S^1_{(\p \beta, \p b)}=(\Sigma \times S^1,T=\{p_1, \dots, p_n\}\times S^1,\coh \oplus (\p \beta, \p b))$. Assume that all colors of $\CS \times S^1_{(\p \beta, \p b)}$ are in $\Gr \setminus \XX$. Then $\TQFT(\CS \times S^1_{(\p \beta, \p b)})$ can be computed as in Section \ref{sec:verlindeArb}. The main difference is that Kirby colors are now formal linear combination of $r^2$ simple objects and $\zeta = - r^2$. Taking this into account, we find
\[
\TQFT(\CS \times S^1_{(\p \beta, \p b)})
=
\D^{-2g-2} r^{4g}\sum_{i,j=0}^{r-1}\qd(V(\beta_0+i,b_0+j)_{\p 0})^{2-2g} S^{\prime}(\{X_i\},\Omega_{\beta}),
\]
where $(\beta_0, b_0) \in \C \times \C$ is a chosen lift of $(\p \beta, \p b) \in \Gr$. Write the color of $p_k$ as $X_k=V(\mu_k,m_k)_{\p q_k}$. With this notation, we have
\begin{equation}
\label{eq:verlindeROU}
\TQFT(\CS \times S^1_{(\p \beta,\p b)})
=
(-1)^{g+1+n + \p q} r^{2g-2}\sum_{i,j=0}^{r-1}q^{-2 ((\beta_0+i) \p m + (b_0+j) \p \mu)} (q^{\beta_0+i}-q^{-\beta_0-i})^{2g-2+n}.
\end{equation}
By holomorphicity, this equality holds whenever $\CS$ is admissible.


Define the generating function of $\Zt$-graded dimensions of $\TQFT(\CS)$ by
\[
\dim_{(t_1,t_2,s)} \TQFT(\CS)
=
\sum_{(n,n^{\prime},\p p) \in \Zt} (-1)^{\p p} \dim_{\C} \TQFT_{(n,n^{\prime},\p p)} (\CS) t_1^{n} t_2^{n^{\prime}} s^{\p p}.
\]
We can now state the Verlinde formula in the present setting.

\begin{theorem}
\phantomsection
\label{thm:verlindeROU}
\begin{enumerate}
\item If $r$ is odd, then $\TQFT(\CS \times S^1_{(\p \beta, \p b)}) = \dim_{(q^{-2 r b},q^{-2 r \beta},1)} \TQFT(\CS)$.

\item If $r$ is even, then $\TQFT(\CS \times S^1_{(\p \beta, \p b)}) = \dim_{(q^{-4 r b},q^{-2 r \beta},1)} \TQFT(\CS)$.
\end{enumerate}
\end{theorem}

\begin{proof}
This is a direct modification of the proof of Theorem \ref{thm:verlindeArb}. For example, when $r$ is odd, the appearance of the pairing
\[
\psi: \Gr \times\Zt \rightarrow \C^{\times},
\qquad
((\p \alpha,\p a), (n,n^{\prime}, \p s)) \mapsto q^{-2(2n \p a+n' \p \alpha )r}
\]
leads to the replacement of $q^{-2 \beta b_i}$ with $q^{-2(\beta b_i +b \delta_i)}$ in the second equality of computation from the proof of Theorem \ref{thm:verlindeArb}.
\end{proof}

When $r$ is odd, the right hand side of the equality in Theorem \ref{thm:verlindeROU} can be written as
\[
\dim_{(q^{-2 r b},q^{-2 r \beta},1)} \TQFT(\CS) = \sum_{(n,n^{\prime}) \in \Z^2} \chi(\TQFT_{(n,n^{\prime},\bullet)}(\CS)) q^{-2r(\overline{\beta} n^{\prime} + \overline{b} n)}.
\]
There is a similar formula when $r$ is even.

\begin{example}
\label{ex:VerlindeROU}
When $n=0$, equation \eqref{eq:verlindeROU} becomes
\[
\TQFT(\CS \times S^1_{(\p \beta, \p b)})
=
(-1)^{g+1} r^{2g-1}\sum_{i=0}^{r-1}(q^{\beta_0+i}-q^{-\beta_0-i})^{2g-2}.
\]
Assume that $g \geq 1$ and $r$ is odd. Theorem \ref{thm:genusgStateSpaceROU} implies $\dim_{\C} \TQFT(\CS) = \lim_{\p{\beta} \rightarrow \frac{1}{4}}\TQFT(\CS \times S^1_{(\p \beta, \p b)})$. To compute the limit, recall the identities
\begin{equation}
\sum_{n=0}^{N-1} \cos(a+nb)
=
\begin{cases}
\label{eq:aritSumCos}
\displaystyle \frac{\sin(\frac{N b}{2})}{\sin(\frac{b}{2})} \cos \left( \frac{(N-1)b}{2}+a \right) & \mbox{if } b \notin 2\pi \Z, \\
N \cos a & \mbox{if } b \in 2\pi \Z
\end{cases}
\end{equation}
and
\begin{equation}
\label{eq:cosEvenPower}
\cos^{2N} \theta = \frac{1}{2^{2N}} {2N \choose N} + \frac{(-1)^N}{2^{2N-1}}\sum_{k=0}^{N-1} (-1)^k {2N \choose k} \cos\left( 2(N-k)\theta \right)
\end{equation}
where $a,b \in \R$ and $N \in \Z_{>0}$. See \cite[Eqn. (4.4.1.5)]{prudnikov1986} and \cite[Eqn. (I.1.9)]{prudnikov1986}, respectively. We compute
\begin{eqnarray*}
\dim_{\C} \TQFT(\CS)
&=&
r^{2g-1}\sum_{i=0}^{r-1}\Big( 2\sin \Big(\frac{2\pi (\frac{1}{4}+i)}{r} \Big) \Big)^{2g-2} \\
&=&
r^{2g}{2g-2 \choose g-1} + (-1)^{g-1} 2 r^{2g-1} \sum_{k=0}^{g-2} (-1)^k {2g-2 \choose k} \sum_{i=0}^{r-1} \cos\left( \frac{4\pi (g-1-k) (\frac{1}{4}+i)}{r} \right) \\
&=&
r^{2g}{2g-2 \choose g-1} + (-1)^{g-1} 2 r^{2g-1} \sum_{\substack{k=0 \\ r \mid g-1-k}}^{g-2} (-1)^k {2g-2 \choose k} r \cos\left( \frac{\pi (g-1-k)}{r} \right) \\
&=&
r^{2g}{2g-2 \choose g-1} + (-1)^{g-1} 2 r^{2g} \sum_{n^{\prime}=1}^{\left \lfloor{\frac{g-1}{r}}\right \rfloor} (-1)^{g-1-n^{\prime} r} {2g-2 \choose g-1-n^{\prime} r} \cos(n^{\prime} \pi) \\
&=&
r^{2g} \sum_{n^{\prime}=-\left \lfloor{\frac{g-1}{r}}\right \rfloor}^{\left \lfloor{\frac{g-1}{r}}\right \rfloor} {2g-2 \choose g-1-\vert n^{\prime} \vert r}.
\end{eqnarray*}
The second equality follows from equation \eqref{eq:aritSumCos}, the third from equation \eqref{eq:cosEvenPower} and the fourth from writing $g-1-k = n^{\prime} r$. Similarly, we compute
\[
\chi(\TQFT(\CS))
=
\lim_{\p{\beta} \rightarrow 0}\TQFT(\CS \times S^1_{(\p \beta,\p b)})
=
r^{2g-1}\sum_{i=0}^{r-1}\left( 2\sin \left(\frac{2\pi i}{r} \right) \right)^{2g-2}
=
r^{2g} {2g-2 \choose g-1},
\]
the final equality following from \cite[Eqn. (4.4.2.1)]{prudnikov1986}.

The same expressions hold for $\dim_{\C}\TQFT(\CS)$ and $\chi(\TQFT(\CS))$ when $r$ is even.
\end{example}

\begin{corollary}
\label{cor:stateBasisROU}
Work in the setting of Theorem \ref{thm:genusgStateSpaceArb}. The vectors $\{v_c \mid c \in \widetilde{\mathfrak{C}}^{(r)}\}$ are a basis of $\TQFT(\CS)$. In particular, $\TQFT(\CS)$ is supported in degrees $(0,d,\p d)$ with $d \in [-(g-1),g-1] \cap r\Z$.
\end{corollary}

\begin{proof}
Using Proposition \ref{prop:dimVanROU}, we compute
\[
\vert \widetilde{\mathfrak{C}}^{(r)} \vert
=
\sum_{d \in [(-g-1),g-1)] \cap r\Z} \vert \widetilde{\mathfrak{C}}^{(r)}_{(d,\p d)} \vert
=
\sum_{d \in [(-g-1),g-1)] \cap r\Z}  r^{2g}{2g-2 \choose g-1-\vert d \vert},
\]
which is equal to $\dim_{\C} \TQFT(\CS)$ by Example \ref{ex:VerlindeROU}.
\end{proof}


\subsection{The state space of the torus with non-generic cohomology class}
\label{sec:torusROU}
 
Let $\Sigma=S^1 \times S^1$ and $\CS=(\Sigma,\coh, \mathcal{L})$. Let $\eta=B^2\times S^1$ with core $\Gamma=\{0\}\times S^1$. By Corollary \ref{cor:stateBasisROU}, when $\coh$ is generic, $\TQFT(\CS)$ is concentrated in $\Zt$-degree zero and is $r^2$ dimensional. The next result gives a non-generic counterpart of this result.

\begin{proposition}
\label{prop:nonGenCohTorusROU}
Assume that there exists an oriented simple closed curve $\gamma$ in $\Sigma$ such that $\coh(\gamma) \in \XX \subset \Gr$. Then $\TQFT(\CS)$ is concentrated in $\Zt$-degree $0$ and is of dimension $r^2+1$.
\end{proposition}

\begin{proof}
Up to diffeomorphism, we may assume that $\coh(m_{\Gamma}) \in \XX$. For concreteness, suppose that $\coh(m_{\Gamma}) =(\p 0, \p 0)$; the case in which $r$ is odd and $\coh(m_{\Gamma}) =(\p {\frac{1}{2}}, \p 0)$ is similar. Given $1 \leq i \leq r-1$ and $0 \leq j \leq r-1$, denote by $\mathcal{M}_{i,j}$, $\mathcal{P}_j$ and $\mathcal{P}_{j,x}$ the vectors in $\state(\CS)$ with underlying $3$-manifold $\eta$ with core is colored by $V(i,j)_{\p 0}$, $P(0,j)_{\p 0}$ and $P(0,j)_{\p 0}$ with coupon $x_{0,j,\p 0}$, respectively. The set
\begin{equation}
\label{eq:basisNonGenTorusROU}
\{\mathcal{M}_{i,j}, \mathcal{P}_j \mid 1 \leq i \leq r-1, \; 0 \leq j \leq r-1\} \cup \{\mathcal{P}_x=\sum_{j=0}^{r-1} \mathcal{P}_{j,x}\}
\end{equation}
spans $\state(\CS)$, as follows from Propositions \ref{P:VSVSS0} and \ref{prop:HomProj}. For linear independence,
consider first the $\CGP$-pairings of $[\overline{\eta}_{-V(i,j)_{\p 0}}]$, a negatively oriented copy of $\eta$ with meridian colored by $V(i,j)_{\p 0}$, with the vectors in question.

Denote by $[\overline{\eta}_{-V(i,j)_{\p 0}}]$ a negatively oriented copy of $\eta$ with meridian colored by $V(i,j)_{\p 0}$. The decorated manifold underlying $[\overline{\eta}_{-V(i,j)_{\p 0}}] \circ \mathcal{M}_{k,l,\p 0}$ has a surgery presentation in $S^3$ given by two parallel knots of opposite orientation labeled by $V(i,j)$ and $V(k,l)$ encircled by an $\Omega_{(\p \alpha, \p a)}$-colored unknot, for any generic $(\p \alpha, \p a)$. Using this, we compute
\begin{eqnarray*}
\CGP_{\catq}([\eta_{-V(i,j)_{\p 0}}] \circ \mathcal{M}_{k,l,\p 0})
&=&
\frac{\delta_{i,k}\delta_{j,l} \zeta}{\D^2 \qd(V(i,j)_{\p 0})} \mt_{V(i,j)_{\p 0}\otimes V(i,j)_{\p 0}^*}(\tcoev_{V(i,+j)_{\p 0}} \circ \ev_{V(i,j)_{\p 0}}) \\
&=&
\frac{\delta_{i,k}\delta_{j,l} \zeta (q-q^{-1}) \qd(V(i,j)_{\p 0})}{\D^2 \qd(V(\alpha_0+i,a_0+j)_{\p 0})} \mt_{P(0,0)_{\p 0}}(x_{0,0,\p 0})\\
&=&
\delta_{i,k} \delta_{j,l}.
\end{eqnarray*}
The first equality follows from the relative modularity condition \eqref{eq:mod}, the second from Lemma \ref{lem:endAlgIsom} and the third from Lemma \ref{lem:modTrComp}. The pairing $\CGP_{\catq}([\overline{\eta}_{-V(i,j)_{\p 0}}] \circ \mathcal{P}_k)$ vanishes since it is proportional to the modified trace of $\Phi_{\Omega_{(\p \alpha, \p a)},V(i,j)_{\p 0} \otimes P(0,k)_{\p 0}}$, which vanishes by part (\ref{ite:betaPhiROU3}) of Lemma \ref{lem:betaPhiROU}. For similar reasons, we have $\CGP_{\catq}([\overline{\eta}_{-V(i,j)_{\p 0}}] \circ \mathcal{P}_{k,x}) = 0$. It follows that $\{\mathcal{M}_{i,j}\}_{i,j}$ are linearly independent and are linearly independent from $\{\mathcal{P}_k, \mathcal{P}_{k,x}\}_k$. 

Continuing, denote by $[\overline{\eta}_{-l}]$ a negatively oriented copy of $\eta$ with meridian colored by $\ve(0,l)_{\p 0}$. Using Lemma \ref{lem:modTrComp} and part (\ref{ite:betaPhiROU2}) of Lemma \ref{lem:betaPhiROU}, we compute
\[
\CGP_{\catq}([\overline{\eta}_{-l}] \circ \mathcal{P}_k)
=
\D^{-2} \mt_{P(0,k-l)} \Phi_{\Omega_{(\p \alpha, \p a)},P(0,k-l)_{\p 0}}
=\delta_{k,l}
\]
and
\begin{eqnarray*}
\CGP_{\catq}([\overline{\eta}_{-l}] \circ \mathcal{P}_{k,x})
&=&
\D^{-2} \mt_{P(0,k-l)} \left(x_{0,k-l, \p 0} \circ \Phi_{\Omega_{(\p \alpha, \p a)},P(0,k-l)_{\p 0}} \right)\\
&=&
-\D^{-2} (q-q^{-1}) r^2 \mt_{P(0,0)_{\p 0}}(x^2_{0,0,\p 0}) \delta_{k,l} \\
&=&
0.
\end{eqnarray*}
Hence, $\{\mathcal{P}_k\}_k$ are linearly independent and are linearly independent from $\{\mathcal{P}_{k,x}\}_k$. We claim that $\{\mathcal{P}_{k,x}\}_k$ spans a one dimensional subspace of $\state(\CS)$. By Proposition \ref{prop:HomProj}, the endomorphisms $a_{j, \p p}^-, a_{j+1, \p p +\p 1}^+ \in \End_{\catq}(P(0,j)_{\p p} \oplus P(0,j+1)_{\p p + \p 1})$ satisfy $[a_{j, \p p}^-,a_{j+1, \p p +\p 1}^+]=x_{0,j,\p p} \oplus x_{0,j+1,\p p + \p 1}$.
The vector defined by coloring the core with $P(0,j)_{\p p} \oplus P(0,j+1)_{\p p + \p 1}$ and coupon $x_{0,j,\p p} \oplus x_{0,j+1,\p p + \p 1}$ is zero, since the coupon is a commutator. On the other hand, this vector is equal to $\mathcal{P}_{j,x}- \mathcal{P}_{j+1,x}$.
Hence, $\mathcal{P}_{j,x} = \mathcal{P}_{k,x}$ for all $j,k$. Finally, in the notation of Section \ref{sec:torusArb}, we have
\[
\CGP_{\catq}([S\overline{\eta}_{-l}]\circ \mathcal{P}_{j,x})
=
\D^{-1}\mt_{P(\frac{(j-l) \pi \sqrt{-1}}{\hbar},0)_{\p 0}}(x_{\frac{(j-l) \pi \sqrt{-1}}{\hbar},0,\p 0})
=
\D^{-1}(q-q^{-1})^{-1}.
\]
for all $j,k$. In particular, $\sum_{j=0}^{r-1} \mathcal{P}_{j,x} \neq 0$.

Finally, to prove that $\TQFT(\CS)$ is concentrated in $\Zt$-degree $0$, recall from Proposition \ref{prop:FDofTQFT} that $\TQFT_{-k}(\CS)$, $k \in \Zt$, is spanned by finite projective $\catq$-colorings of the ribbon graph $\Gamma^{\prime}$. In genus $1$, the coupons of $\Gamma^{\prime}$ are represented by a composition
\[
V(i,j)_{\p 0} \otimes V(i,j)_{\p 0}^* \simeq P(0,0)_{\p 0}
\rightarrow
P
\rightarrow 
V_{g_0} \otimes \sigma(k) \otimes V_{g_0}^* \simeq P(nr,nr^{\prime})_{\p p}
\]
or 
\[
P(0,0)_{\p 0}^{\otimes 2} \simeq P(0,1)_{\p 1} \oplus P(0,0)_{\p 0}^{\oplus 2} \oplus P(0,-1)_{\p 1}
\rightarrow
P
\rightarrow 
V_{g_0} \otimes \sigma(k) \otimes V_{g_0}^* \simeq P(nr,nr^{\prime})_{\p p},
\]
where $P$ is a projective indecomposable of degree $(\overline{0},\overline{0})$ and $k = (n,n^{\prime},\p p)$. By Proposition \ref{prop:HomProj}, if either of these compositions is non-zero, then $k=0$.
\end{proof}

We claim that the basis dual to \eqref{eq:basisNonGenTorusROU} is
\begin{equation}
\label{eq:dualBasisNonGenTorusROU}
\{[\overline{\eta}_{-V(i,j)_{\p 0}}], [\overline{\eta}_{-j}] \mid 1 \leq i \leq r-1, \; 0 \leq j \leq r-1\} \cup \{ s= \frac{\sqrt{-1} (q-q^{-1})}{r} \sum_{j=0}^{r-1}[S\overline{\eta}_{-j}]\}.
\end{equation}
Most of this statement was verified in the proof of Proposition \ref{prop:nonGenCohTorusROU}. That $s$ pairs trivially with each $\mathcal{P}_j$ follows from Lemma \ref{lem:modTrComp}. Using equation \eqref{eq:mdimTypical}, we compute
\begin{eqnarray*}
\CGP_{\catq}(s \circ \mathcal{M}_{i,j})
&=&
\frac{\sqrt{-1}(q-q^{-1})}{r} \sum_{l=0}^{r-1} \CGP_{\catq}([S\overline{\eta}_{-l}] \circ \mathcal{M}_{i,j}) \\
&=&
\frac{\sqrt{-1}(q-q^{-1})}{r} \D^{-1} \sum_{l=0}^{r-1} q^{-2il} (q^i-q^{-i})^{-1} \\
&=&
0
\end{eqnarray*}
because $q^{-2i} \neq 1$ is an $r$\textsuperscript{th} root of unity. Similarly, part (\ref{ite:betaPhiROU1}) of Lemma \ref{lem:betaPhiROU} implies that $\CGP_{\catq}([\overline{\eta}_{-l}] \circ \mathcal{M}_{i,j}) =0$ for all $i,j,l$. This establishes the claim.

Next, we compute the action of the mapping class group of $\Sigma$ on $\TQFT(\CS)$ when $\coh=0$ using the basis \eqref{eq:basisNonGenTorusROU} and its dual basis \eqref{eq:dualBasisNonGenTorusROU}. As in Section \ref{sec:torusArb}, we take $\mathcal{L} = \R \cdot [m]$ and find that there are no Maslov corrections. Using Lemmas \ref{lem:Phi's} and \ref{lem:modTrComp}, we compute
\[
\CGP_{\catq}([\overline{\eta}_{-V(i,j)_{\p 0}}] \circ N_S \mathcal{M}_{k,l})
=
- \D^{-1} q^{-2(il+jk) + i + k}
\]
and
\[
\CGP_{\catq}([\overline{\eta}_{-j}] \circ N_S \mathcal{M}_{k,l})
=
\D^{-1} \frac{q^{-2jk}}{q^k-q^{-k}}.
\]
The calculations checking that \eqref{eq:dualBasisNonGenTorusROU} is a dual basis show that $\CGP_{\catq}(s \circ N_S \mathcal{M}_{k,l})=0$. Using Lemmas \ref{lem:Phi's} and \ref{lem:modTrComp}, we compute
\[
\CGP_{\catq}([\overline{\eta}_{-V(i,j)_{\p 0}}] \circ N_S \mathcal{P}_l)
=
- \D^{-1} q^{-2il}(q^i-q^{-i})
\]
and $\CGP_{\catq}([\overline{\eta}_{-j}] \circ N_S \mathcal{P}_l) =0$. Using Lemma \ref{lem:betaPhiROU}(\ref{ite:betaPhiROU2}), we find
\[
\CGP_{\catq}(s \circ  N_S \mathcal{P}_l)
=
\frac{\sqrt{-1}(q-q^{-1})}{r} \sum_{j=0}^{r-1} \CGP_{\catq}([S \overline{\eta}_{-j}] \circ N_S \mathcal{P}_l) \\
=
\frac{\sqrt{-1}(q-q^{-1})}{r},
\]
the only contribution coming from the $j=l$ summand. Part (\ref{ite:betaPhiROU3}) of Lemma \ref{lem:betaPhiROU} implies $\CGP_{\catq}([\overline{\eta}_{-V(i,j)_{\p 0}}] \circ N_S \mathcal{P}_x)=0$ while Lemma \ref{lem:modTrComp} implies
\[
\CGP_{\catq}([\overline{\eta}_{-j}] \circ N_S\mathcal{P}_x)
=
- \D^{-1} \frac{r}{q-q^{-1}}.
\]
Finally, Lemma \ref{lem:Phi's} and the fact that $x_{0,k,\p 0}^2=0$ for all $k$ gives
\[
\CGP_{\catq}(s \circ N_S \mathcal{P}_x)
=
\frac{\sqrt{-1}(q-q^{-1})}{r} \sum_{j,l=0}^{r-1} \CGP_{\catq}([S\overline{\eta}_{-j}] \circ N_S \mathcal{P}_{l,x})
=0.
\]
The above calculations prove the first part of the following result.

\begin{theorem}
\label{thm:MCGROU}
Let $\CS$ be a connected decorated surface of genus one without marked points and with non-generic cohomology class. The action of the mapping class group $SL(2,\Z)$ on $\TQFT_{\catq}(\CS)$ in the basis \eqref{eq:basisNonGenTorusROU} is determined by the equations
\[
N_S \mathcal{M}_{i,j}
=
\frac{\sqrt{-1}}{r}\sum^{r-1}_{\substack{k,l=0 \\ k \neq 0}} q^{-2(k(j-\frac{1}{2})+i(l-\frac{1}{2}))} \mathcal{M}_{k,l} -\frac{\sqrt{-1}}{r} \sum_{l=0}^{r-1} \frac{q^{-2il}}{q^i-q^{-i}} \mathcal{P}_{l},
\]
\[
N_S \mathcal{P}_j
=
\frac{\sqrt{-1}(q-q^{-1})}{r} \mathcal{P}_{x} +\frac{\sqrt{-1}}{r} \sum_{\substack{k,l=0 \\ k \neq 0}}^{r-1}  q^{-jk}(q^k-q^{-k}) \mathcal{M}_{k,l},
\]
\[
N_S \mathcal{P}_{x}
=
\frac{\sqrt{-1}}{q-q^{-1}} \sum_{l=0}^{r-1} \mathcal{P}_l
\]
and
\[
N_T \mathcal{M}_{i,j}= q^{-2i(j-\frac{1}{2})} \mathcal{M}_{i,j},
\qquad
N_T\mathcal{P}_j=\mathcal{P}_j+ \frac{(q-q^{-1})}{r}\mathcal{P}_x,
\qquad
N_T\mathcal{P}_x=\mathcal{P}_x.
\]
In particular, the Dehn twist acts with infinite order.
\end{theorem}

\begin{proof}
It remains to compute the action of the Dehn twist $T$. Recall from the proof of Theorem \ref{thm:ribbonCat} that $\theta_{V(i,j)_{\p 0}}=q^{-2i(j-\frac{1}{2})}\Id_{V(i,j)_{\p 0}}$. A similar computation gives $\theta_{P(0,j)_{\p 0}}=\Id_{P(0,j)_{\p 0}} + (q-q^{-1})x_{0,j,\p 0}$. The action of $T$ is therefore as stated. The final statement follows from the observation that $N^k_T \mathcal{P}_j = \mathcal{P}_j + k \frac{q-q^{-1}}{r} \mathcal{P}_x$, $k \geq 1$.
\end{proof}

\section{TQFTs from integral modules}
\label{sec:intTQFT}

The relative modular structures on $\catInt$ constructed in Sections \ref{sec:relModIntArb} and \ref{sec:relModIntROU} give integral counterparts of the TQFTs of Sections \ref{sec:TQFTArb} and \ref{sec:TQFTROU}. Since the computations for $\TQFT_{\catInt}$ are straightforward modifications of those for $\TQFT_{\catq}$, we limit our discussion to a summary.

\subsection{Arbitrary $q$}
\label{sec:intTQFTArb}

Work in the setting of Section \ref{sec:relModIntArb} and give $\catInt$ the relative modular structure of Theorem \ref{thm:relModIntArb}.
Fix $\D = \sqrt{-1}$. The braiding of $\ZVect_{\C}$ is determined by
\[
\gamma: \Zt \times \Zt \rightarrow \{ \pm 1\},
\qquad
((n_1,n^{\prime}_1,\p p_1), (n_2,n^{\prime}_2,\p p_2)) \mapsto (-1)^{\p p_1 \p p_2}.
\]

Since each $\Theta(\p \alpha)$, $\p \alpha \in \Gr \setminus \XX$, is a singleton, it is straightforward to verify that equation \eqref{eq:verlindeArb} and the Verlinde formula (Theorem \ref{thm:verlindeArb}) hold as written for $\TQFT_{\catInt}$. \emph{A priori}, the difference between the free realisation groups of $\catInt$ and $\catq$, being $\Zt_{\catInt}=\Z \times \Z \times \Ztwo$ and $\Zt_{\catq}=\C \times \Ztwo$, could lead to significant differences between the state spaces of generic surfaces. However, we have seen in Corollary \ref{cor:stateBasisArb} that $\TQFT_{\catq}(\CS)$ is concentrated in integral degrees and spanned by graphs colored by integral modules. For this reason, when $\CS$ is of genus $g \geq 1$, we obtain canonical isomorphisms
\begin{equation}
\label{eq:intStateSpace}
\TQFT_{\catInt,(0,d,\p d)}(\CS)
\simeq
\TQFT_{\catq,(d,\p d)}(\CS)
,
\qquad
d \in [-(g-1),g-1] \cap \Z
\end{equation}
which define an isomorphism $\TQFT_{\catInt}(\CS) \simeq \TQFT_{\catq}(\CS)$ which is compatible with the group homomorphism $\Zt_{\catInt} \rightarrow \Zt_{\catq}$, $(n,n^{\prime}, \p p) \mapsto (n^{\prime}, \p p)$. The computations of Sections \ref{sec:torusArb} and \ref{sec:alexPoly} go through with only obvious modifications and no changes to the end results.

\subsection{$q$ a root of unity}
\label{sec:intTQFTROU}

Work in the setting of Section \ref{sec:relModIntROU} with $\D$ and $\gamma$ defined as in Section \ref{sec:TQFTROU}. The TQFT $\TQFT_{\catInt}$ has the same properties as the TQFT $\TQFT_{\catq}$ of Section \ref{sec:TQFTROU}, with analogues of Theorems \ref{thm:genusgStateSpaceROU} and \ref{thm:verlindeROU}, equation \eqref{eq:verlindeROU} and Proposition \ref{prop:nonGenCohTorusROU} holding with only obvious changes. The isomorphism \eqref{eq:intStateSpace} again holds in the current setting, with the additional assumption that $d$ is divisible by $r$. The computations of Sections \ref{sec:torusROU} and \ref{sec:alexPoly} go through with only obvious modifications and no changes to the end results.

\section{Comparison with supergroup Chern--Simons and Wess--Zumino--Witten theories coupled to background flat $\C^{\times}$-connections}
\label{sec:physCompare}

Let $\cat$ be a modular $\Gr$-category relative to $(\Zt,\XX)$. Recently, there has been significant interest in the physics literature in constructing quantum field theories which realize (a differential graded or derived enhancement of) the TQFT $\TQFT_{\cat}$ \cite{gukov2021,creutzig2021,costantino2021}. Such a quantum field theory is expected to admit $\Gr$ as a group of global symmetries so that it can be coupled to background flat $\Gr$-connections, thereby producing local invariants of $3$-manifolds with flat $\Gr$-connection. With this problem in mind, in \cite{creutzig2021} a physical quantum field theory $\mathcal{T}^A_{n,r}$ was constructed as a topological $A$-twist of $3$d $\mathcal{N}=4$ Chern--Simons-matter theory with gauge group $SU(n)$ at level $r-n$ which conjecturally realizes the TQFT $\TQFT_{U^H_q(\mathfrak{sl}(n))}$ associated to the category of weight $U^H_q(\mathfrak{sl}(n))$-modules at a primitive $2r$\textsuperscript{th} root of unity. Using sophisticated techniques from quantum field theory and vertex operators algebras, a number of calculations were made for the theory $\mathcal{T}^A_{2,r}$ and the results were shown to match known properties of $\TQFT_{U^H_q(\mathfrak{sl}(2))}$, thereby presenting evidence for the conjecture when $n=2$.

In this section we argue that the TQFTs $\TQFT_{\catInt}$ of Sections \ref{sec:intTQFTArb} and \ref{sec:intTQFTROU} have comparatively simple physical realizations, namely, as $\psl$ and $\Uoo$ Chern--Simons theories, respectively. Supergroup Chern--Simons theories have the benefit of being formally similar to Chern--Simons theory with compact gauge group \cite{witten1989}, allowing for the import of intuition from the compact case. By work of Kapustin and Saulina \cite{kapustin2009b}, Chern--Simons theories with gauge supergroup arise as $B$-twists of $3$d $\mathcal{N}=4$ supersymmetric quantum field theories and, for particular choices of gauge supergroup, are $3$d mirrors of the quantum field theories $\mathcal{T}^A_{n,r}$ constructed in \cite{creutzig2021}.

\subsection{$\psl$ Chern--Simons theory}
\label{sec:pslCompare}

Let $\pgl$ be the quotient of $\gloo$ by its center $\C \cdot \left( \begin{smallmatrix} 1 & 0 \\ 0 & 1 \end{smallmatrix} \right)$ and $\psl$ the quotient of $\pgl$ by the Lie ideal $\C \cdot \overline{\left( \begin{smallmatrix} 1 & 0 \\ 0 & 0 \end{smallmatrix} \right)}$. The Lie superalgebra $\psl$ is two dimensional, purely odd and abelian. The group $\C^{\times}$ acts on $\psl$ with weights $+1$ and $-1$. More generally, $GL(2,\C)$ acts on $\psl$ by Lie superalgebra automorphisms, but only the action of the anti-diagonal $\C^{\times} \leq GL(2,\C)$ lifts to $\gloo$.

Chern--Simons theory with gauge supergroup $\psl$ was studied by Mikhaylov \cite{Mik2015}. Since $\psl$ is purely odd, this theory has a number of peculiarities compared to the compact case, including no quantization of the level and being defined by the Lie superalgebra $\psl$ without the choice of an associated Lie supergroup. This theory is closely related to a number of other well-known models in physics, including the Rozansky--Witten theory of the cotangent bundle $T^{\vee} \C$ \cite{rozansky1997} and the $B$-twist of a $3$d $\mathcal{N}=4$ free hypermultiplet \cite{kapustin2009b,Mik2015,costello2019,creutzig2021}. The $\C^{\times}$-action on $\psl$ lifts to a global symmetry of $\psl$ Chern--Simons theory, thereby allowing for the theory to be coupled to background flat $\C^{\times}$-connections.

\begin{proposal}
\label{proposal:psl}
For arbitrary $q$, the TQFT $\TQFT_{\catInt}$ described in Section \ref{sec:intTQFTArb} is the homological truncation of $\psl$ Chern--Simons theory coupled to flat $\C^{\times}$-connections.
\end{proposal}

At present, there is no mathematical definition of $\psl$ Chern--Simons theory so that Proposal \ref{proposal:psl} cannot be formulated as a theorem. Instead, we offer evidence for the proposal. We identify the grading group $\Gr = \C \slash \frac{2 \pi \sqrt{-1}}{\hbar} \Z$ with $\C^{\times}$ so that the cohomology classes $\coh \in H^1(-; \Gr)$ appearing as decoration data of the category $\Cob_{\catInt}$ can be interpreted as isomorphism classes of flat $\C^{\times}$-connections.
\begin{itemize}
\item $\psl$ Chern--Simons theory admits Wilson operators labeled by a link whose components are colored by representations of $\pgl$ \cite[\S 2.4]{Mik2015}. Indeed, the tensor product of the dynamical $\psl$-connection  (used to define the $\psl$ Chern--Simons Lagrangian) with the background $\C^{\times}$-connection defines a $\pgl$-connection from which Wilson operators can be constructed, analogously to the construction of Wilson operators in Chern--Simons theory with compact gauge group \cite[\S 2.1]{witten1989}. In the notation of Section \ref{sec:weightMod}, the simple representations of $\pgl$ are $\ve(0,b)_{\p p}$, $(b,\p p) \in \C \times \Ztwo$. The restriction to integral modules, $b \in \Z$, reflects the desired interpretation of the $\ve(0,1)_{\p 0}$-labeled Wilson loop as an operator which changes the $\spin^c$-structure by a single unit \cite[\S 2.4]{Mik2015}. In the physical approach, $\spin^c$-structures enter in the definition of the phase of partition functions. Wilson operators labeled by the projective indecomposables $P(0,b)_{\p p}$ of $\pgl$ are studied in \cite[\S 5.3.3]{Mik2015}.

\item $\psl$ Chern--Simons theory also admits monodromy operators \cite[\S 2.4]{Mik2015}. These operators take as input a framed link $L$ and prescribe the holonomy of the background flat $\C^{\times}$-connection along the meridians of $L$. More generally, one can combine this operator with a Wilson operator to obtain a \emph{combined (Wilson/monodromy) operator}. In our set up, the meridian holonomy is captured by the $E$-weights. Note that $E$ is not a generator of $\pgl$. For example, coloring a knot by a simple module $V(\alpha,b)_{\p p} \in \catInt$ corresponds to the combined operator with prescribed holonomy $q^{\alpha} \neq 1$ and Wilson factor $\ve(0,b)_{\p p}$. When the meridian holonomy is trivial, $q^{\alpha}=1$, so that $\alpha =\frac{n \pi \sqrt{-1}}{\hbar}$ for some $n \in \Z$, we realize the combined operator as coloring by the projective indecomposable $P(\frac{n \pi \sqrt{-1}}{\hbar},b)_{\p p}$ with a coupon labeled by its non-zero idempotent. Viewing this data as the projective cover of $\ve(\frac{n \pi \sqrt{-1}}{\hbar},b)_{\p p}$, as in Lemma \ref{lem:projIndec}, we see that each simple $\Uq$-module corresponds to a unique combined operator. In particular, enlarging $\pgl$ to $\Uq$ allows for the incorporation of both Wilson and monodromy operators of $\psl$ Chern--Simons theory.

\item The partition function of trivial circle bundles over surfaces, given by equation \eqref{eq:verlindeArb}, agrees with \cite[Eqns.\  (5.3), (5.8)]{Mik2015}. Similarly, the graded dimensions of state spaces of generic surfaces (Corollary \ref{cor:stateBasisArb}) agrees with the results of Mikhaylov's formal application of geometric quantization to super Chern--Simons theory \cite[Eqns.\ (5.6-7)]{Mik2015}.

\item That the dimension of the degree zero summand $\TQFT_{0}(\CS)$ of the state space of a non-generic torus is two (Proposition \ref{prop:nonGenCohTorusArb}) agrees with \cite[\S 5.3.1]{Mik2015}. By Proposition \ref{prop:nonGenCohTorusHighDegArb}, the summands $\TQFT_k(\CS)$ vanish for non-zero $k$. On the other hand, Mikhaylov predicted that the state space of derived $\psl$ Chern--Simons theory contains a factor of $\left( \bigwedge\nolimits^{\bullet} H^1(\CS;\C)\right)[-1]$, the degree zero summand of which is $H^1(\CS;\C) \simeq \C^2$ \cite[\S 5.2]{Mik2015}. In view of the expectation that $\TQFT$ describes the homological truncation of derived $\psl$ Chern--Simons theory, this suggests that the summands $\bigwedge\nolimits^{k} H^1(\CS;\C) \simeq \C$, $k=0,2$, appear only at the derived level. It is interesting to note that these summands also appear in the state space of the TQFT of Frohman and Nicas \cite{frohman1991,kerler2003}.

\item That the mapping class group action on the degree zero state space of non-generic tori is projectively isomorphic to the fundamental representation of $SL(2,\Z)$ matches with \cite[\S 5.3.1]{Mik2015}.

\item The results of Section \ref{sec:alexPoly} match physical expectations that $\psl$ Chern--Simons theory recovers the multivariable Alexander polynomial \cite[\S 5.2.2]{Mik2015}.
\end{itemize}

\subsection{$\Uoo$ Chern--Simons theory}
\label{sec:UCompare}

The physical study of Chern--Simons theories with gauge supergroups $\Uoo$ and $\GLoo$ was initiated by Rozansky and Saleur \cite{rozansky1992,rozansky1993,rozansky1994} under the assumption of the existence of a super generalization of the Chern--Simons/Wess--Zumino--Witten correspondence. Mikhaylov studied $\Uoo$ Chern--Simons theory without reference to Wess--Zumino--Witten theory, viewing it instead as the orbifold of $\psl$ Chern--Simons theory by a finite cyclic group \cite[\S 4.3]{Mik2015}.

The group $\C^{\times}$ acts on $\gloo$ by Lie algebra automorphisms with weight decomposition
\[
\gloo_{-1} = \C \cdot Y,
\qquad
\gloo_0 = \C \cdot G \oplus \C \cdot E,
\qquad
\gloo_{+1} = \C \cdot X.
\]
This $\C^{\times}$-action lifts to a global symmetry of Chern--Simons theories with gauge supergroups $\GLoo$ and $\Uoo$, allowing each theory to be coupled to flat $\C^{\times}$-connections.

\begin{proposal}
\label{proposal:uoo}
For $q$ a primitive $r$\textsuperscript{th} root of unity, the TQFT $\TQFT_{\catInt}$ described in Section \ref{sec:intTQFTROU} is the homological truncation of $\Uoo$ Chern--Simons theory at level $r$ coupled to flat $\C^{\times}$-connections.
\end{proposal}

We outline evidence for this proposal. For concreteness, suppose that $r$ is odd. Again, we can identify the grading group $\Gr$ with $\C^{\times}$.
\begin{itemize}
\item Mirroring the discussion from Section \ref{sec:pslCompare}, the $\catInt$-coloring of ribbon graphs used in this paper matches the labelings of the combined Wilson/monodromy operators constructed in \cite[\S 4.5]{Mik2015}. At first sight, there is an ambiguity of whether to view colorings by simple modules of the form $V(i,j)_{\p 0}$, $\{0 \leq i,j \leq r-1\}$, as Wilson or monodromy operators. (The integrality of the $E$-weights of such modules implies that their classical limits lift to representations of $\Uoo$, whence can be viewed as labeling Wilson operators.) However, physical arguments suggest that these two operators coincide for such modules \cite[\S 3.2]{mikhaylov2015b}, \cite[\S 4.5]{Mik2015}.

\item The partition function of trivial circle bundles over surfaces, given by equation \eqref{eq:verlindeROU}, agrees with \cite[Eqn.\ (123)]{rozansky1993}. We are not aware of results in the physics literature which compute the dimension of state spaces of generic surfaces, as in Corollary \ref{cor:stateBasisROU}.

\item The dimensions of state spaces of non-generic tori, as computed in Proposition \ref{prop:nonGenCohTorusROU}, agree with \cite[\S 5.4]{Mik2015} and the proposal of Aghaei, Gainutdinov, Pawelkiewicz and Schomerus \cite{aghaei2018}, who construct candidate state spaces of non-generic tori using combinatorial quantization based on the small quantum group of $\gloo$. State spaces of generic tori do not seem to have been studied in the physics literature.

\item Let $\CS$ be a decorated torus without marked points and with trivial cohomology class. The mapping class group action on $\TQFT(\CS)$ obtained in Theorem \ref{thm:MCGROU} agrees with the regularized mapping class group action obtained using $\Uoo$ Wess--Zumino--Witten theory \cite[\S 3]{rozansky1993}. Theorem \ref{thm:MCGROU} is closely related to the mapping class group actions of \cite[\S 5.4]{Mik2015} and \cite[Eqns.\ (4.61-3), (4.57-9)]{aghaei2018}. Explicitly, the relation between the basis \eqref{eq:basisNonGenTorusROU} and that of \cite[\S 5.4]{Mik2015} is
\[
\mathcal{M}_{i,j} \longleftrightarrow \vert L_{i,j} \rangle
\qquad
\mathcal{P}_j \longleftrightarrow \vert L_{j} \rangle
\qquad
\mathcal{P}_{x} \longleftrightarrow r \cdot v_0 \otimes \vert 0_a \rangle.
\]
Under this correspondence, the only difference in mapping class group actions is that the trigonometric factors $q^i-q^{-i}$ in Theorem \ref{thm:MCGROU} appear inverted in \cite{Mik2015,aghaei2018}.
\end{itemize}

\linespread{1}
\newcommand{\etalchar}[1]{$^{#1}$}


\begin{thebibliography}{DRGPM20}

\bibitem[AGPM21]{AGP}
C.~Anghel, N.~Geer, and B.~Patureau-Mirand.
\newblock Relative (pre)-modular categories from special linear {L}ie
  superalgebras.
\newblock {\em J. Algebra}, 586:479--525, 2021.

\bibitem[AGPS18]{aghaei2018}
N.~Aghaei, A.~Gainutdinov, M.~Pawelkiewicz, and V.~Schomerus.
\newblock Combinatorial quantization of {C}hern-{S}imons theory {I}: {T}he
  torus.
\newblock ar{X}iv:1811:09123, 2018.

\bibitem[And92]{andersen1992}
H.~Andersen.
\newblock Tensor products of quantized tilting modules.
\newblock {\em Comm. Math. Phys.}, 149(1):149--159, 1992.

\bibitem[BCGPM16]{BCGP}
C.~Blanchet, F.~Costantino, N.~Geer, and B.~Patureau-Mirand.
\newblock Non-semi-simple {TQFT}s, {R}eidemeister torsion and {K}ashaev's
  invariants.
\newblock {\em Adv. Math.}, 301:1--78, 2016.

\bibitem[BI22]{bao2022}
Y.~Bao and N.~Ito.
\newblock {$\mathfrak{gl}(1 \vert 1)$}-{A}lexander polynomial for
  {$3$}-manifolds.
\newblock ar{X}iv:2022:00238, 2022.

\bibitem[BNW91]{barnatan1991}
D.~Bar-Natan and E.~Witten.
\newblock Perturbative expansion of {C}hern-{S}imons theory with noncompact
  gauge group.
\newblock {\em Comm. Math. Phys.}, 141(2):423--440, 1991.

\bibitem[CDGG21]{creutzig2021}
T.~Creutzig, T.~Dimofte, N.~Garner, and N.~Geer.
\newblock A {QFT} for non-semisimple {TQFT}.
\newblock ar{X}iv:2112.01559, 2021.

\bibitem[CG19]{costello2019}
K.~Costello and D.~Gaiotto.
\newblock Vertex {O}perator {A}lgebras and 3d {$\mathcal{N}=4$} gauge theories.
\newblock {\em J. High Energy Phys.}, (5):018, 37, 2019.

\bibitem[CGP21]{costantino2021}
F.~Costantino, S.~Gukov, and P.~Putrov.
\newblock Non-semisimple {TQFT}s and {BPS} $q$-series.
\newblock ar{X}iv:2107.14238, 2021.

\bibitem[CGPM14]{CGP14}
F.~Costantino, N.~Geer, and B.~Patureau-Mirand.
\newblock Quantum invariants of 3-manifolds via link surgery presentations and
  non-semi-simple categories.
\newblock {\em J. Topol.}, 7(4):1005--1053, 2014.

\bibitem[CGPM15]{CGP2}
F.~Costantino, N.~Geer, and B.~Patureau-Mirand.
\newblock Some remarks on the unrolled quantum group of {$\mathfrak{sl}(2)$}.
\newblock {\em J. Pure Appl. Algebra}, 219(8):3238--3262, 2015.

\bibitem[CMY22]{creutzig2022}
T.~Creutzig, R.~McRae, and J.~Yang.
\newblock Tensor structure on the {K}azhdan-{L}usztig category for affine
  {$\mathfrak{gl}(1|1)$}.
\newblock {\em Int. Math. Res. Not. IMRN}, (16):12462--12515, 2022.

\bibitem[CR13]{Creutzig:2011cu}
T.~Creutzig and D.~Ridout.
\newblock Relating the archetypes of logarithmic conformal field theory.
\newblock {\em Nuclear Phys. B}, 872(3):348--391, 2013.

\bibitem[DEF{\etalchar{+}}99]{deligne1999}
P.~Deligne, P.~Etingof, D.~Freed, L.~Jeffrey, D.~Kazhdan, J.~Morgan,
  D.~Morrison, and E.~Witten, editors.
\newblock {\em Quantum fields and strings: a course for mathematicians. {V}ol.
  1, 2}.
\newblock American Mathematical Society, Providence, RI; Institute for Advanced
  Study (IAS), Princeton, NJ, 1999.

\bibitem[DGLZ09]{dimofte2009}
T.~Dimofte, S.~Gukov, J.~Lenells, and D.~Zagier.
\newblock Exact results for perturbative {C}hern-{S}imons theory with complex
  gauge group.
\newblock {\em Commun. Number Theory Phys.}, 3(2):363--443, 2009.

\bibitem[DR22]{derenzi2022}
M.~De~Renzi.
\newblock Non-semisimple extended topological quantum field theories.
\newblock {\em Mem. Amer. Math. Soc.}, 277(1364):v+161, 2022.

\bibitem[DRGPM20]{derenzi2020}
M.~De~Renzi, N.~Geer, and B.~Patureau-Mirand.
\newblock Nonsemisimple quantum invariants and {TQFT}s from small and unrolled
  quantum groups.
\newblock {\em Algebr. Geom. Topol.}, 20(7):3377--3422, 2020.

\bibitem[EGNO15]{etingof2015}
P.~Etingof, S.~Gelaki, D.~Nikshych, and V.~Ostrik.
\newblock {\em Tensor categories}, volume 205 of {\em Mathematical Surveys and
  Monographs}.
\newblock American Mathematical Society, Providence, RI, 2015.

\bibitem[FN91]{frohman1991}
C.~Frohman and A.~Nicas.
\newblock The {A}lexander polynomial via topological quantum field theory.
\newblock In {\em Differential geometry, global analysis, and topology
  ({H}alifax, {NS}, 1990)}, volume~12 of {\em CMS Conf. Proc.}, pages 27--40.
  Amer. Math. Soc., Providence, RI, 1991.

\bibitem[GHN{\etalchar{+}}21]{gukov2021}
S.~Gukov, P.-S. Hsin, H.~Nakajima, S.~Park, D.~Pei, and N.~Sopenko.
\newblock Rozansky-{W}itten geometry of {C}oulomb branches and logarithmic knot
  invariants.
\newblock {\em J. Geom. Phys.}, 168:Paper No. 104311, 22, 2021.

\bibitem[GKPM11]{GKP1}
N.~Geer, J.~Kujawa, and B.~Patureau-Mirand.
\newblock Generalized trace and modified dimension functions on ribbon
  categories.
\newblock {\em Selecta Math. (N.S.)}, 17(2):453--504, 2011.

\bibitem[GKPM22]{GKP3}
N.~Geer, J.~Kujawa, and B.~Patureau-Mirand.
\newblock {M}-traces in (non-unimodular) pivotal categories.
\newblock {\em Algebr Represent Theor}, 125:759--776, 2022.

\bibitem[GPM07]{geer2007}
N.~Geer and B.~Patureau-Mirand.
\newblock Multivariable link invariants arising from {$\mathfrak{sl}(2|1)$} and
  the {A}lexander polynomial.
\newblock {\em J. Pure Appl. Algebra}, 210(1):283--298, 2007.

\bibitem[GPM10]{GP1}
N.~Geer and B.~Patureau-Mirand.
\newblock Multivariable link invariants arising from {L}ie superalgebras of
  type {I}.
\newblock {\em J. Knot Theory Ramifications}, 19(1):93--115, 2010.

\bibitem[GPMR21]{geer2021}
N.~Geer, B.~Patureau-Mirand, and M.~Rupert.
\newblock Some remarks on relative modular categories.
\newblock ar{X}iv:2110.15518, 2021.

\bibitem[GPMT09]{GPT}
N.~Geer, B.~Patureau-Mirand, and V.~Turaev.
\newblock Modified quantum dimensions and re-normalized link invariants.
\newblock {\em Compos. Math.}, 145(1):196--212, 2009.

\bibitem[GQS07]{Gotz:2006qp}
G.~G\"{o}tz, T.~Quella, and V.~Schomerus.
\newblock The {WZNW} model on {$\rm PSU(1,1|2)$}.
\newblock {\em J. High Energy Phys.}, (3):003, 48, 2007.

\bibitem[Guk05]{gukov2005}
S.~Gukov.
\newblock Three-dimensional quantum gravity, {C}hern-{S}imons theory, and the
  {A}-polynomial.
\newblock {\em Comm. Math. Phys.}, 255(3):577--627, 2005.

\bibitem[GW10]{GaiottoWitten-Janus}
D.~Gaiotto and E.~Witten.
\newblock Janus configurations, {C}hern-{S}imons couplings, and the
  {$\theta$}-angle in {$\mathcal{N}=4$} super {Y}ang-{M}ills theory.
\newblock {\em J. High Energy Phys.}, (6):097, 58, 2010.

\bibitem[Ha18]{ha2018}
N.~P. Ha.
\newblock Topological invariants from quantum group
  {$U_{\xi}\mathfrak{sl}(2|1)$} at roots of unity.
\newblock {\em Abh. Math. Semin. Univ. Hambg.}, 88(1):163--188, 2018.

\bibitem[Ha22]{ha2022}
N.~P. Ha.
\newblock Anomaly-free {TQFT}s from the super {L}ie algebra
  {$\mathfrak{sl}(2\vert 1)$}.
\newblock {\em J. Knot Theory Ramifications}, 31(5):Paper No. 2250029, 14,
  2022.

\bibitem[Hen96]{hennings1996}
M.~Hennings.
\newblock Invariants of links and {$3$}-manifolds obtained from {H}opf
  algebras.
\newblock {\em J. London Math. Soc. (2)}, 54(3):594--624, 1996.

\bibitem[Hor90]{horne1990}
J.~Horne.
\newblock Skein relations and {W}ilson loops in {C}hern-{S}imons gauge theory.
\newblock {\em Nuclear Phys. B}, 334(3):669--694, 1990.

\bibitem[Ker03]{kerler2003}
T.~Kerler.
\newblock Homology {TQFT}'s and the {A}lexander-{R}eidemeister invariant of
  3-manifolds via {H}opf algebras and skein theory.
\newblock {\em Canad. J. Math.}, 55(4):766--821, 2003.

\bibitem[KL01]{kerler2001}
T.~Kerler and V.~Lyubashenko.
\newblock {\em Non-semisimple topological quantum field theories for
  3-manifolds with corners}, volume 1765 of {\em Lecture Notes in Mathematics}.
\newblock Springer-Verlag, Berlin, 2001.

\bibitem[KS91]{kauffman1991}
L.~Kauffman and H.~Saleur.
\newblock Free fermions and the {A}lexander-{C}onway polynomial.
\newblock {\em Comm. Math. Phys.}, 141(2):293--327, 1991.

\bibitem[KS09]{kapustin2009b}
A.~Kapustin and N.~Saulina.
\newblock Chern-{S}imons-{R}ozansky-{W}itten topological field theory.
\newblock {\em Nuclear Phys. B}, 823(3):403--427, 2009.

\bibitem[KT91]{khoroshkin1991}
S.~Khoroshkin and V.~Tolstoy.
\newblock Universal {$R$}-matrix for quantized (super)algebras.
\newblock {\em Comm. Math. Phys.}, 141(3):599--617, 1991.

\bibitem[Kul89]{kulish1989}
P.~Kulish.
\newblock Quantum {L}ie superalgebras and supergroups.
\newblock In {\em Problems of modern quantum field theory ({A}lushta, 1989)},
  Res. Rep. Phys., pages 14--21. Springer, Berlin, 1989.

\bibitem[Man19]{manion2019}
A.~Manion.
\newblock On the decategorification of {O}zsv\'{a}th and {S}zab\'{o}'s bordered
  theory for knot {F}loer homology.
\newblock {\em Quantum Topol.}, 10(1):77--206, 2019.

\bibitem[Mik15]{Mik2015}
V.~Mikhaylov.
\newblock Analytic torsion {$3d$} mirror symmetry and supergroup
  {C}hern--{S}imons theories.
\newblock ar{X}iv:1505.03130, 2015.

\bibitem[MR20]{manion2020}
A.~Manion and R.~Rouquier.
\newblock Higher representations and cornered {H}eegaard {F}loer homology.
\newblock ar{X}iv:2009.09627, 2020.

\bibitem[MW15]{mikhaylov2015b}
V.~Mikhaylov and E.~Witten.
\newblock Branes and supergroups.
\newblock {\em Comm. Math. Phys.}, 340(2):699--832, 2015.

\bibitem[PBM86]{prudnikov1986}
A.~Prudnikov, Y.~Brychkov, and O.~Marichev.
\newblock {\em Integrals and series. {V}ol. 1}.
\newblock Gordon \& Breach Science Publishers, New York, 1986.
\newblock Translated from the Russian and with a preface by N. M. Queen.

\bibitem[QS07]{Quella:2007hr}
T.~Quella and V.~Schomerus.
\newblock Free fermion resolution of supergroup {WZNW} models.
\newblock {\em J. High Energy Phys.}, (9):085, 51, 2007.

\bibitem[Res92]{reshetikhin1992}
N.~Reshetikhin.
\newblock Quantum supergroups.
\newblock In {\em Quantum field theory, statistical mechanics, quantum groups
  and topology ({C}oral {G}ables, {FL}, 1991)}, pages 264--282. World Sci.
  Publ., River Edge, NJ, 1992.

\bibitem[RS92]{rozansky1992}
L.~Rozansky and H.~Saleur.
\newblock Quantum field theory for the multi-variable {A}lexander-{C}onway
  polynomial.
\newblock {\em Nuclear Phys. B}, 376(3):461--509, 1992.

\bibitem[RS93]{rozansky1993}
L.~Rozansky and H.~Saleur.
\newblock {$S$}- and {$T$}-matrices for the super {${\rm U}(1,1)\ {\rm WZW}$}
  model. {A}pplication to surgery and {$3$}-manifolds invariants based on the
  {A}lexander-{C}onway polynomial.
\newblock {\em Nuclear Phys. B}, 389(2):365--423, 1993.

\bibitem[RS94]{rozansky1994}
L.~Rozansky and H.~Saleur.
\newblock Reidemeister torsion, the {A}lexander polynomial and {${\rm U}(1,1)$}
  {C}hern-{S}imons theory.
\newblock {\em J. Geom. Phys.}, 13(2):105--123, 1994.

\bibitem[RT90]{reshetikhin1990}
N.~Reshetikhin and V.~Turaev.
\newblock Ribbon graphs and their invariants derived from quantum groups.
\newblock {\em Comm. Math. Phys.}, 127(1):1--26, 1990.

\bibitem[RT91]{reshetikhin1991}
N.~Reshetikhin and V.~Turaev.
\newblock Invariants of {$3$}-manifolds via link polynomials and quantum
  groups.
\newblock {\em Invent. Math.}, 103(3):547--597, 1991.

\bibitem[RW97]{rozansky1997}
L.~Rozansky and E.~Witten.
\newblock Hyper-{K}\"ahler geometry and invariants of three-manifolds.
\newblock {\em Selecta Math. (N.S.)}, 3(3):401--458, 1997.

\bibitem[Sar15]{sartori2015}
A.~Sartori.
\newblock The {A}lexander polynomial as quantum invariant of links.
\newblock {\em Ark. Mat.}, 53(1):177--202, 2015.

\bibitem[Saw06]{sawin2006}
S.~Sawin.
\newblock Quantum groups at roots of unity and modularity.
\newblock {\em J. Knot Theory Ramifications}, 15(10):1245--1277, 2006.

\bibitem[Tan92]{tanisaki1992}
T.~Tanisaki.
\newblock Killing forms, {H}arish-{C}handra isomorphisms, and universal
  {$R$}-matrices for quantum algebras.
\newblock In {\em Infinite analysis, {P}art {A}, {B} ({K}yoto, 1991)},
  volume~16 of {\em Adv. Ser. Math. Phys.}, pages 941--961. World Sci. Publ.,
  River Edge, NJ, 1992.

\bibitem[Tur94]{Tu}
V.~Turaev.
\newblock {\em Quantum invariants of knots and 3-manifolds}, volume~18 of {\em
  De Gruyter Studies in Mathematics}.
\newblock Walter de Gruyter \& Co., Berlin, 1994.

\bibitem[TW93]{turaev1993}
V.~Turaev and H.~Wenzl.
\newblock Quantum invariants of {$3$}-manifolds associated with classical
  simple {L}ie algebras.
\newblock {\em Internat. J. Math.}, 4(2):323--358, 1993.

\bibitem[Vir06]{Viro}
O.~Viro.
\newblock Quantum relatives of the {A}lexander polynomial.
\newblock {\em Algebra i Analiz}, 18(3):63--157, 2006.

\bibitem[Wit89a]{witten1989}
E.~Witten.
\newblock Quantum field theory and the {J}ones polynomial.
\newblock {\em Comm. Math. Phys.}, 121(3):351--399, 1989.

\bibitem[Wit89b]{witten1989b}
E.~Witten.
\newblock Topology-changing amplitudes in {$(2+1)$}-dimensional gravity.
\newblock {\em Nuclear Phys. B}, 323(1):113--140, 1989.

\bibitem[Wit91]{witten1991}
E.~Witten.
\newblock Quantization of {C}hern-{S}imons gauge theory with complex gauge
  group.
\newblock {\em Comm. Math. Phys.}, 137(1):29--66, 1991.

\end{thebibliography}
\end{document}